\documentclass{article}
\usepackage{amsmath, amsthm, amssymb, mathrsfs}
\allowdisplaybreaks[1]
\theoremstyle{plain}
\newtheorem{thm}{Theorem}
\newtheorem{prop}{Proposition}
\newtheorem{cor}{Corollary}
\newtheorem{lem}{Lemma}
\newtheorem*{assumption}{Assumption}

\theoremstyle{definition}
\newtheorem{rem}{Remark}

\newcommand{\Avg}{\mathrm{Avg}}

\newcommand{\Sym}{\mathrm{Sym}}

\title{On the value-distribution of the logarithms of symmetric power $L$-functions in the level aspect}
\author{
Philippe Lebacque,
Kohji Matsumoto,
Masahiro Mine
\\
and 
Yumiko Umegaki
}
\date{}
\begin{document}
\maketitle
%
%
\begin{abstract}
We consider the value distribution of logarithms of symmetric power $L$-functions associated with newforms of even weight and prime power level at $s=\sigma > 1/2$. 
In the symmetric square case, 
under certain plausible analytical conditions,
we prove that certain averages of those values in the level aspect, involving continuous bounded or Riemann integrable test functions $\Psi(u)$, can be written as integrals involving a density function (the ``$M$-function'') which is related with the Sato-Tate measure.
Moreover, even in the case of general symmetric power $L$-functions, we show the same type of formula when $\Psi(u)=cu$.   
We see that a kind of parity phenomenon of the density function exists.
\end{abstract}
%
%
\section{Introduction and the statement of main results}\label{secIntro}
\par
The prototype of the theory of $M$-functions is the limit theorem of Bohr and Jessen \cite{bj} for the Riemann zeta-function $\zeta(s)$, where $s=\sigma+i\tau\in\mathbb{C}$.
Let $R$ be a rectangle in the complex plane with the edges parallel to the axes.
Let $\sigma>1/2$, and for any $T>0$, let $L_{\sigma}(T,R)$ be the 1-dimensional Lebesgue measure of the set $\{\tau\in[0,T]\mid \log\zeta(\sigma+i\tau)\in R\}$.
Then the Bohr-Jessen theorem asserts the existence of a continuous non-negative function $\mathcal{M}_{\sigma}(w,\zeta)$ defined over $\mathbb{C}$, for which
\begin{align}\label{BJ-thm}
\lim_{T\to\infty}\frac{L_{\sigma}(T,R)}{T}=\int_R \mathcal{M}_{\sigma}(w,\zeta)
\frac{dudv}{2\pi}
\end{align}
holds with $w=u+iv\in\mathbb{C}$.
Here, the left-hand side is the average of $\log\zeta(s)$ in the $\tau$-aspect.
\par
Analogous results for Dirichlet $L$-functions $L(s,\chi)$ (and more generally, $L$-functions over number fields and function fields) were discussed much later by Ihara \cite{ihara} and a series of joint papers by Ihara and the second-named author
\cite{im_2010} \cite{im_2011} \cite{im-moscow} \cite{im_2014}.
Their results assert the existence of a continuous non-negative function $\mathcal{M}_{\sigma}(w,L)$ for which formulas of the form
\begin{align}\label{IM-thm}
{\rm Avg}_{\chi}\Phi(\log L(s,\chi))=\int_{\mathbb{C}}\mathcal{M}_{\sigma}(w,L)\Phi(w)
\frac{dudv}{2\pi}
\end{align}
are valid, where $\Phi$ is a certain test function and ${\rm Avg}_{\chi}$ means a certain average with respect to $\chi$.
The functions $\mathcal{M}_{\sigma}$ appearing on the right-hand sides of \eqref{BJ-thm} and \eqref{IM-thm} are now called the $M$-functions associated with $\log\zeta(s)$ and
with $\log L(s,\chi)$, respectively.
\par
After those studies, several mathematicians tried to obtain the same type of formulas for other zeta and $L$-functions.
In particular, $M$-functions associated with automorphic $L$-functions have been studied in recent years.
\par
This direction of research was first cultivated by two papers published in 2018.
The level and the modulus aspects of automorphic $L$-functions were discussed by \cite{lz}, while the difference of logarithms of two symmetric power $L$-functions
were treated in \cite{mu}.
\par
As for the $\tau$-aspect, formulas analogous to \eqref{BJ-thm} were obtained in \cite{mu-JNT} for automorphic $L$-functions, and in \cite{mu-ropar} for symmetric power $L$-functions.
The article \cite{mineLMJ} of the third-named author is also to be mentioned, in which a more general framework is treated.
\par
Recently, the third-named author \cite{mine} obtained a quite satisfactory analogue of \eqref{IM-thm} for automorphic $L$-functions in the level aspect.
\par
The aim of the present paper is to obtain an analogue of \eqref{IM-thm} for symmetric power $L$-functions, under certain analytical conditions, in the level aspect.
We will show the full analogue of \eqref{IM-thm} only in the case of symmetric square $L$-functions, but we will also obtain a result in a very special case for general symmetric power $L$-functions.
\par
In order to state the results, now let us prepare the notations.
Let $f$ be a primitive form of weight $k$ and level $N$, which means that it is a normalized common Hecke-eigen newform of weight $k$ for $\Gamma_0(N)$. 
We denote by $S_k(N)$ the set of all cusp forms of weight $k$ and level $N$. 
We have a Fourier expansion of $f$ at infinity of the form
\[
f(z)=\sum_{n=1}^{\infty}\lambda_f(n)n^{(k-1)/2}e^{2\pi inz},
\]
where $\lambda_f(1)=1$ and the Fourier coefficients $\lambda_f(n)$ are real numbers. 
We consider the $L$-function 
\[
L(f, s)
=
\sum_{n=1}^{\infty}\frac{\lambda_f(n)}{n^s} 
\]
associated with $f$ where $s=\sigma+i\tau\in\mathbb{C}$.
It is absolutely convergent when $\sigma >1$, and can be continued to the 
whole complex plane as an entire function.
\par
We denote by $\mathbb{P}$ the set of all prime numbers.
For $N$, let $\mathbb{P}_N$ be the set of primes which do not divide $N$.
We know that $L(f, s)$ $(\sigma>1)$ has the Euler product expansion
\begin{align*}
L(f, s)
=&
\prod_{p \mid N}(1-\lambda_f(p) p^{-s})^{-1}
\prod_{p \in \mathbb{P}_N}(1-\lambda_f(p) p^{-s} + p^{-2s})^{-1}
\\
=&
\prod_{p \mid N}(1-\lambda_f(p) p^{-s})^{-1}
\prod_{p \in \mathbb{P}_N}(1-\alpha_f(p) p^{-s})^{-1}(1-\beta_f(p) p^{-s})^{-1},
\end{align*}
where $\alpha_f(p)$ and $\beta_f(p)$ satisfy $\alpha_f(p)+\beta_f(p)=\lambda_f(p)$ and $|\alpha_f(p)|=|\beta_f(p)|=1$, and they are the complex conjugate of each other.
This Euler product is deduced from the relations
\begin{equation}\label{euler}
\lambda_f(p^{\ell})=
\begin{cases}
\displaystyle \lambda_f^{\ell}(p) & p \mid N,\\
\displaystyle \sum_{h=0}^{\ell}\alpha_f^{\ell-h}(p)\beta_f^h(p) & p \nmid N.
\end{cases}
\end{equation}
From these relations, we see that $|\lambda_f(n)|\leq d(n)$, where $d(n)$ is the number of divisors of $n$.
Further, for $p \mid N$, we have $\lambda_f^2(p)=p^{-1}$ if the $p$-component of $N$ is
just $p$, and $\lambda_f(p)=0$ if $p^2|N$ (see Miyake \cite{mi}).
\par
For any positive integer $r$ and $\sigma>1$, we denote the (partial) $r$th symmetric power $L$-function by
\[
L_{\mathbb{P}_N}(\Sym_f^{r}, s)=\prod_{p\in \mathbb{P}_N}\prod_{h=0}^{r}(1-\alpha_f^{r-h}(p)\beta_f^h(p) p^{-s})^{-1}.
\]
When $r=1$, this is nothing but the partial $L$-function
\[
L_{\mathbb{P}_N}(f, s)=
\prod_{p \in \mathbb{P}_N}(1-\alpha_f(p) p^{-s})^{-1}(1-\beta_f(p) p^{-s})^{-1}.
\]
Let
\begin{align}\label{log-L-P-def}
\log L_{\mathbb{P}_N}(\Sym_f^{r}, s)
=&
-
\sum_{p \in \mathbb{P}_N}
\sum_{h=0}^{r}
\mathrm{Log}(1-\alpha_f^{r-h}(p)\beta_f^h(p)p^{-s}),
\end{align}
where $\mathrm{Log}$ means the principal branch.
%
%
%
\par
By the recent achievement by Newton and Thorne \cite{NT2021}, there are local factors $L_p(\Sym_f^{r}, s)$ for $p\mid N$ such that 
\begin{gather*}
L(\Sym_f^{r}, s)
= L_{\mathbb{P}_N}(\Sym_f^{r}, s) \prod_{p \mid N} L_{p}(\Sym_f^{r}, s) 
\end{gather*}
agrees with the $L$-function of an automorphic representation of $\mathrm{GL}_{r+1}(\mathbb{A}_\mathbb{Q})$. 
By Godement and Jacquet \cite{GJ1972}, we thus know that $L(\Sym_f^{r}, s)$ can be meromorphically continued to the whole complex plane. 
Some properties for the symmetric power $L$-functions are collected in Section \ref{sec:sym}. 
\par
In the present paper, we assume the generalized Riemann hypothesis.
\begin{assumption}[GRH]
Let $f$ be a primitive form of weight $k$ with $2\leq k < 12$ or $k=14$ and level $q^m$, where $q$ is a prime number.
The $L$-functions $L(\Sym_f^r, s)$ satisfies the Generalized Riemann Hypothesis which means that $L(\Sym_f^r, s)$ has no zero in the strip $1/2 < \sigma \leq 1$.
\end{assumption}
\begin{rem}\label{grh}
It might be possible to assume some zero-density theorem for $L(\Sym_f^r ,s)$
instead of the GRH to obtain theorems similar to the main results in the present paper.    However, here we choose the way of assuming the GRH. This assumption makes it possible to discuss the $M$-function for any symmetric power $L$-function $L(\Sym_f^r, s)$ (see Theorem~\ref{main2}).
\end{rem}
\begin{rem}\label{weight_level}
We consider the average over primitive forms, because we are interested in the values or zeros of the automorphic $L$-functions associated with newforms.
In previous articles  \cite{im_2010} \cite{im_2011} \cite{im-moscow} 
\cite{im_2014} by Ihara and the second-named author, and also in
\cite{lz} by the first-named author and Zykin, similar value-distribution
results in character aspects were obtained with the use of orthogonality of
characters.    In the present situation, 
Petersson's formula restricted to primitive forms plays this role, but we
need to control the error terms.
In \cite{ichihara}, the fourth-named author obtained such a formula for the space of newforms of weight $k$ with $2\leq k < 12$ or $k=14$ and level $q^m$, where $q$ is a prime number,
which we will use in the present paper.
Therefore, in our results we assume the same conditions (see Remark~\ref{restriction_weight} below).
\end{rem}
%
%
\par
Let $q$ be a prime number and $m$ be a positive integer. 
Our aim is to study certain averages of the value of
\[
\log L_{\mathbb{P}_q}(\Sym_f^{r}, \sigma)
=
\log L_{\mathbb{P}_{q^m}}(\Sym_f^r, \sigma)
\]
for $\sigma >1/2$, 
where $f$ varies over the set of primitive forms in $S_k(q^m)$ such that $L(\Sym_f^{r}, s)$ has an analytic continuation to $\mathbb{C}$ as an entire function. 
Because of Assumption (GRH), there is no problem how to choose the branch of the logarithm. 

Here, we present the exact definitions of the averages. 
Define $\mathscr{P}_k(q^m)$ as the set of all primitive forms in $S_k(q^m)$. 
Let $g$ be any function on $\mathscr{P}_k(q^m)$. 
For any subset $X \subset \mathscr{P}_k(q^m)$, we introduce the abbreviation
\begin{gather*}
\sideset{}{'}\sum_{f \in X} g(f)
= \frac{1}{C_k(1-C_q(m))}\sum_{f \in X} \frac{1}{\langle f,f \rangle} g(f), 
\end{gather*}
where $\langle,\rangle$ denotes the Petersson inner product, and
\[
C_k=\frac{(4\pi)^{k-1}}{\Gamma(k-1)},\quad
C_q(m)
=
\begin{cases}
  0 & m=1,\\
  q(q^2-1)^{-1} & m=2,\\
  q^{-1} & m\geq 3. 
\end{cases}
\]
Then, we define two subsets of $\mathscr{P}_k(q^m)$ as follows: 
\begin{align*}
\mathscr{P}_k^\mathrm{E}(q^m)
&= \{f \in \mathscr{P}_k(q^m) \mid \text{$L(\Sym_f^{r}, s)$ is an entire function}\}, \\
\mathscr{P}_k^\mathrm{P}(q^m)
&= \{f \in \mathscr{P}_k(q^m) \mid \text{$L(\Sym_f^{r}, s)$ has poles}\}. \\
\end{align*}
We will see in Section \ref{sec:sym} that $\mathscr{P}_k^\mathrm{P}(q^m) \neq \emptyset$ if and only if $f$ is of CM-type and $r \equiv 0 \pmod{4}$. 
In particular, $\mathscr{P}_k^\mathrm{P}(q^m)$ is always empty for $r=1$ or $2$. 
Let $\Psi$ be a $\mathbb{C}$-valued function defined on $\mathbb{R}$. 
Finally, we define the average
\begin{align*}
\Avg \Psi(\log L_{\mathbb{P}_q}(\Sym_f^{r},\sigma))
= \lim_{q^m\to\infty} 
\sideset{}{'}\sum_{f \in \mathscr{P}_k^\mathrm{E}(q^m)}
\Psi(\log L_{\mathbb{P}_q}(\Sym_f^{r},\sigma)), 
\end{align*}
where the symbol $q^m\to\infty$ means that 
$m$ tends to $\infty$ while a prime $q$ is fixed (Case I), 
or 
$q$ tends to $\infty$ while $m$ is fixed (Case II).
We let
\begin{align*}
\mathbb{P}_*=\begin{cases}
\mathbb{P}_q\;\;\text{in Case I},\\
\mathbb{P} \;\; \;\text{in Case II}. 
\end{cases}
\end{align*}
%
%
\par
Our ultimate objective is to construct the ``$M$-function'' $\mathcal{M}_{\sigma}(\Sym^r, u)$ which satisfies
\begin{align*}
\Avg \Psi (\log L_{\mathbb{P}_q}(\Sym_f^{r},\sigma))
=
\int_{\mathbb{R}}\mathcal{M}_{\sigma}(\Sym^r, u)\Psi (u)
\frac{du}{\sqrt{2\pi}}
\end{align*}
for any function $\Psi:\mathbb{R}\to\mathbb{C}$ of exponential growth under the GRH. 
This type of result was established in \cite[Theorem 4]{im-moscow} for Dirichlet $L$-functions of global fields.
We may expect that the same type of result would hold for symmetric power $L$-functions.
In the present paper, however, we work with $\Psi$ of much stronger
conditions.
%
%
\par
Let
\[
\rho
=
\begin{cases}
  [r/2] & r\;\text{is odd}\\
  [r/2]-1 & r\;\text{is even}
\end{cases}
=
\begin{cases}
  \displaystyle \frac{r-1}{2} & r\;\text{is odd},\\
  \displaystyle \frac{r}{2}-1 & r\;\text{is even}.
\end{cases}
\]
The first main aim of the present paper is to establish the following theorem, which treats the cases $r=1$ or $2$ which means $\rho=0$.
%
%
\begin{thm}\label{main1}
Let $k$ an even integer which satisfies $2\leq k <12$ or $k=14$.
For $r=1, 2$, we suppose Assumption (GRH).
Then, for $\sigma>1/2$, there exists a continuous function 
\begin{gather}\label{eq:M}
\mathcal{M}_{\sigma}(\Sym^r, u)
=\mathcal{M}_{\sigma,\mathbb{P}_*}(\Sym^r, u):\mathbb{R}\to\mathbb{R}_{\geq 0}
\end{gather}
which can be explicitly constructed, and for which the formula
\begin{align}\label{main-formula}
\Avg \Psi (\log L_{\mathbb{P}_q}(\Sym_f^r,\sigma))
=
\int_{\mathbb{R}}\mathcal{M}_{\sigma}(\Sym^r, u)\Psi (u)
\frac{du}{\sqrt{2\pi}}
\qquad (r=1,2) 
\end{align}
holds for the function $\Psi:\mathbb{R} \to \mathbb{C}$ which is a bounded continuous function or a compactly supported Riemann integrable function.
In Case I, the $M$-function 
$\mathcal{M}_{\sigma}=\mathcal{M}_{\sigma,
\mathbb{P}_q}$ depends on $q$, while in Case II, 
$\mathcal{M}_{\sigma}=\mathcal{M}_{\sigma,
\mathbb{P}}$ is independent of $q$.
\end{thm}
%
%
\begin{rem}
In \cite{im_2011} and \cite{mu}, the function $\Psi$ was assumed to be a bounded continuous function or a compactly supported characteristic function.
But the second-named author \cite{matsumoto} noticed that ``compactly supported characteristic function'' should be replaced by ``compactly supported Riemann integrable function''.
\end{rem}
%
%
\begin{rem}\label{abs-conv-case}
The function $\mathcal{M}_{\sigma}(\Sym^r, u)$ ($r=1,2$)
tends to $0$ as $|u|\to\infty$ (Proposition~\ref{M} (3) in Section~\ref{sec3}).
In particular, if $\sigma>1$, then $\mathcal{M}_{\sigma}(\Sym^r, u)$ is compactly supported, and hence \eqref{main-formula} is valid for any continuous $\Psi$.
For the details, see the end of Section \ref{sec4}.
Moreover, in view of \cite{im-moscow}, it is plausible that \eqref{main-formula}
is valid for any $\sigma>1/2$ and for any $\Psi$ which is continuous and of exponential growth.
\end{rem}
We mention here a corollary of Theorem~\ref{main1}.
Consider other averages
\begin{align*}
&\Avg_m\Psi (\log L_{\mathbb{P}_q}(\Sym_f^r,\sigma))
\\
=&
\lim_{X \to \infty}
\frac{1}{\pi(X)} \sum_{\substack{q \leq X\\ q :\mathrm{prime}\\ m : \mathrm{fixed}}} 
\sideset{}{'}\sum_{f \in \mathscr{P}_k^\mathrm{E}(q^m)}
\Psi (\log L_{\mathbb{P}_q}(\Sym_f^r, \sigma))
\end{align*}
where $\pi(X)$ denotes the number of prime numbers not larger than $X$, and
\begin{align*}
&\Avg^* \Psi (\log L_{\mathbb{P}_q}(\Sym_f^r,\sigma))
\\
=&
\lim_{X \to \infty}
\frac{1}{\pi^*(X)}
\underset{q^m \leq X}{\sum_{q :\text{prime}}\sum_{1 \leq m}}
\sideset{}{'}\sum_{f \in \mathscr{P}_k^\mathrm{E}(q^m)}
\Psi (\log L_{\mathbb{P}_q}(\Sym_f^r,\sigma)),
\end{align*}
where $\pi^*(X)$ denotes the number of all pairs $(q,m)$ of a prime number $q$ and a positive integer $m$ with $q^m\leq X$.
Theorem~\ref{main1} implies the following corollary.
The proof is the same as that in \cite{mu}.
\begin{cor}\label{main-cor}
Under the same assumptions as Theorem~\ref{main1}, we have
\begin{align*}
&\Avg_m \Psi (\log L_{\mathbb{P}_q}(\Sym_f^r, \sigma))
=
\Avg^* \Psi (\log L_{\mathbb{P}_q}(\Sym_f^r, \sigma))
\\
=&
\int_{\mathbb{R}}\mathcal{M}_{\sigma,\mathbb{P}}(\Sym^r, u)\Psi (u)
\frac{du}{\sqrt{2\pi}}
\end{align*}
when $\Psi$ is bounded.
\end{cor}
After some preparations in Section \ref{sec2}, we will
introduce the $M$-function $\mathcal{M}_{\sigma}(\Sym^r,u)$ and its
Fourier transform in Section \ref{sec3}.
In Section \ref{sec4} we will
state the key lemma (Lemma \ref{keylemma}), and then
complete the proof of Theorem~\ref{main1} and Remark~\ref{abs-conv-case} 
at the end of Section~\ref{sec4}.
The proof of Lemma \ref{keylemma} will be given in Section \ref{sec5}.
The basic stream of the proof is similar to that in \cite{mu}, but a big
difference is that we will use the Sato-Tate measure in the present paper.

The main point of this theorem is the case $r=2$, that is the case of symmetric square $L$-functions.
The case $r=1$ is nothing but the usual $L$-function attached to $f$, and in this case the third-named author \cite{mine} proved a stronger unconditional result.
Nevertheless we include the case $r=1$ in the theorem, because it can be treated in parallel with the case $r=2$, and also, the case $r=1$ is important in view of the next result. 

%
%
Regarding the second main result, we further assume that there are not so many primitive forms $f$ for which $L(\Sym_f^r, s)$ have poles. 
The following assumption is an analogue of \cite[Hypothesis 11.2]{ST2016}, which was originally used to study of low-lying zeros of automorphic $L$-functions in families. 
%
\begin{assumption}[$\mathscr{P}_k^\mathrm{P}(q^m)$ condition]
Let $k$ an even integer which satisfies $2\leq k <12$ or $k=14$. 
For any $r \geq1$, there exists a positive constant 
$\Delta=\Delta(r)$ 
depending at most on $r$ such that 
\begin{gather}\label{hol}
\sideset{}{'}\sum_{f \in \mathscr{P}_k^\mathrm{P}(q^m)}
1
\ll_r q^{-\Delta m}
\end{gather}
holds, where $\ll$ stands for the Vinogradov symbol, the same as Landau's $O$-symbol.
(The suffix here means that the constant implied by $\ll$ depends on $r$. 
In what follows, similarly, the implied constants do not depend on parameters which are not written explicitly, unless otherwise indicated.) 
\end{assumption}
\begin{thm}\label{main2}
Let $k$ an even integer which satisfies $2\leq k <12$ or $k=14$.
Let $r$ be any positive integer.
Suppose Assumption (GRH) and Assumption ($\mathscr{P}_k^\mathrm{P}(q^m)$ condition). 
Let $\Psi_1(x)=cx$ with a constant $c\neq 0$.
For $\sigma>1/2$, we obtain
\begin{align}\label{thm2-1}
\Avg \Psi_1 (\log L_{\mathbb{P}_q}(\Sym_f^r,\sigma))
=
\int_{\mathbb{R}}\mathcal{M}_{\sigma,\mathbb{P}_*}(\Sym^{r_0}, u)\Psi_1 (u)
\frac{du}{\sqrt{2\pi}}
\end{align}
where $\mathbb{P}_*=\mathbb{P}_q$ or $\mathbb{P}$ as in Theorem \ref{main1},
and the right-hand side is further 
\begin{align}\label{thm2-2}
=
\begin{cases}
  -2^{-1}\sum_{p\in\mathbb{P}_*}p^{-2\sigma} & r\;\text{is odd},\\
  \sum_{p\in\mathbb{P}_*}\sum_{j>1} j^{-1}p^{-j\sigma} & r\;\text{is even},
\end{cases}
\end{align}
where $r_0$ is $2$ if $r$ is even, $1$ if $r$ is odd.
\end{thm}
We will prove Theorem \ref{main2} in Sections \ref{sec-Mine}
and \ref{sec-Mine2}.
This theorem is proved only for the very special test function
$\Psi_1(x)=cx$, but we may conjecture that \eqref{thm2-1} is valid for
any test function $\Psi$ which satisfies the same condition as in the
statement of Theorem \ref{main1}.
In particular, we may expect a ``parity phenomenon'' for $M$-functions;
$\mathcal{M}_{\sigma}(\Sym^r, u)=\mathcal{M}_{\sigma}(\Sym^{2}, u)$
for all even $r$, and $=\mathcal{M}_{\sigma}(\Sym^{1}, u)$ for all
odd $r$.
%
%
\begin{rem}\label{Psi_1-valid}
Formula \eqref{thm2-1} for $r=1,2$ implies that Theorem \ref{main1} is valid for
$\Psi=\Psi_1$.
If we further assume that \cite[Theorem~1.5]{mu} is also valid for $\Psi=\Psi_1$,
then we can deduce \eqref{thm2-1} for general $r$ from the case $r=1,2$.
This argument, which will
be described in Section \ref{sec-1to2}, gives a reasoning why the above
``parity phenomenon'' exists.
\end{rem}
%
%

%
%
\section{Preparation}\label{sec2}
\par
In the proof of our main Theorem~\ref{main1}, we will use the following formula (\eqref{P} below) for a prime number $q$, which was shown by the fourth-named author~\cite[Lemma~3]{ichihara}, using Petersson's formula.
%
%
\par
When $2\leq k < 12$ or $k=14$, for the primitive form $f$ of weight $k$ and level $q^m$, we have
\begin{equation}\label{P}
\sideset{}{'}\sum_{f \in \mathscr{P}_k(q^m)}
\lambda_f(n)
=
\delta_{1, n}
+
\begin{cases} 
O(n^{(k-1)/2} q^{-k+1/2})
& m=1, \\
O(n^{(k-1)/2} q^{m(-k+1/2)}q^{k-3/2})
& m\geq 2, 
\end{cases}
\end{equation}
where $\delta_{1, n}=1$ if $n=1$ and $0$ otherwise.
%
%
\begin{rem}\label{restriction_weight}
The implied constants on the right-hand side depend on $k$, but in the present paper we restrict $k$ to $2\leq k<12$ or $k=14$, so we may say that the implied constants above are absolute.
This restriction of the range of $k$ is coming from the matter how to construct the basis of the space of old forms in \cite{ichihara} (see \cite[Remarks 1 and 4]{ichihara}).
\end{rem}
We denote the error term in \eqref{P} by $n^{(k-1)/2} E(q^m)$, that is
\[
\sideset{}{'}\sum_{f \in \mathscr{P}_k(q^m)}
\lambda_f(n)
-
\delta_{1, n}
=
n^{(k-1)/2}E(q^m).
\]
Then we have
\begin{equation*} 
E(q^m)\ll q^{-k+1/2}
\end{equation*}
for any $m$, and 
\begin{align}\label{E2}
E(q^m)\ll& \begin{cases}
          q^{-3/2} & m=1,\\
          q^{-5/2} & m=2,\\
         q^{-1-m}   & m\geq 3,
          \end{cases}
\nonumber \\
\ll& q^{-m-1/2}
\end{align}
for any $m \geq 1$.
Also in the case $n=1$, the formula \eqref{P} implies
\begin{equation}\label{P1}
\sideset{}{'}\sum_{f \in \mathscr{P}_k(q^m)}
\lambda_f(1)
=
\sideset{}{'}\sum_{f \in \mathscr{P}_k(q^m)}
1
=
1+E(q^m) \ll 1.
\end{equation}
%
%
\par
Let $\mathcal{P}_q$ be a finite subset of $\mathbb{P}_q$.
For a primitive form $f$ of weight $k$ and level $q^m$, we define
the finite truncation
\[
L_{\mathcal{P}_q}(\Sym_f^r, s)
=
\prod_{p \in \mathcal{P}_q}
\prod_{h=0}^{r}
(1-\alpha_f^{r-h}(p)\beta_f^h(p)p^{-s})^{-1}
\]
for $\sigma > 1/2$.
When $r$ is even, this can be written as
\begin{align*}
L_{\mathcal{P}_q}(\Sym_f^r, s)
=
\prod_{p \in \mathcal{P}_q}
\Big(\prod_{\substack{h=0 \\ r-2h\neq 0}}^{r}
(1-\alpha_f^{r-h}(p)\beta_f^h(p)p^{-s})^{-1}
\Big)
(1-p^{-s})^{-1}.
\end{align*}
Since $\alpha_f^{r-h}(p)\beta_f^h(p)=\alpha_f^{r-2h}(p)=\beta_f^{2h-r}(p)$, using $\rho$ we may write
\begin{align}\label{alternative-exp}
\lefteqn{L_{\mathcal{P}_q}(\Sym_f^{r}, s)}\nonumber\\
&=
\prod_{p\in\mathcal{P}_q}\prod_{h=0}^{\rho}(1-\alpha_f^{r-2h}(p)p^{-s})^{-1}
(1-\beta_f^{r-2h}(p)p^{-s})^{-1}(1-\delta_{r,\text{even}}p^{-s})^{-1},
\end{align}
where $\delta_{r,\text{even}}=1$ if $r$ is even, and $0$ otherwise.
%
%
\par
Let $\mathcal{T}_{\mathcal{P}_q}=\prod_{p\in \mathcal{P}_q} \mathcal{T}$ with $\mathcal{T}=\{t\in\mathbb{C}\mid |t|=1\}$.
For a fixed $\sigma>1/2$ and $t_p\in\mathcal{T}$ we define 
\begin{align*}
g_{\sigma, p}(t_p)=&-\log (1-t_pp^{-\sigma})
\\
\mathscr{G}_{\sigma, p}(t_p)
=&
g_{\sigma, p}(t_p)+g_{\sigma, p}(\overline{t_p})+g_{\sigma, p}(\delta_{r,\text{even}})
\\
=&
-2\log|1-t_p p^{-\sigma}|-\log (1-\delta_{r,\text{even}}p^{-\sigma}),
\end{align*}
where $\overline{t_p}$ is the complex conjugate of $t_p$.
Here $g_{\sigma, p}(t_p)$ is the same symbol as in \cite{mu}.
Putting $t_p=e^{i\eta_p}$, we may also write
\begin{align}\label{cos-expression}
\mathscr{G}_{\sigma, p}(e^{i\eta_p})=
-\log(1-2 p^{-\sigma}\cos\eta_p+p^{-2\sigma})-\log (1-\delta_{r,\text{even}}p^{-\sigma}).
\end{align}
Further, for 
$t_{\mathcal{P}_q}=(t_p)_{p\in\mathcal{P}_q} \in\mathcal{T}_{\mathcal{P}_q}$, let
\begin{align*}
g_{\sigma, \mathcal{P}_q}(t_{\mathcal{P}_q})
=&
\sum_{p\in\mathcal{P}_q} g_{\sigma, p}(t_p),
\\
g_{\sigma, \mathcal{P}_q}(\delta_{r,\text{even}})
=&
\sum_{p\in\mathcal{P}_q} 
g_{\sigma, p}(\delta_{r,\text{even}}),
\\
\mathscr{G}_{\sigma, \mathcal{P}_q}(t_{\mathcal{P}_q})
=&
\sum_{p\in\mathcal{P}_q} \mathscr{G}_{\sigma, p}(t_p)
=
\sum_{p\in\mathcal{P}_q}
g_{\sigma, p}(t_p)
+
\sum_{p\in\mathcal{P}_q}
g_{\sigma, p}(\overline{t_p})
+
\sum_{p\in\mathcal{P}_q} g_{\sigma, p}(\delta_{r,\text{even}})
\\
=&
g_{\sigma, \mathcal{P}_q}(t_{\mathcal{P}_q})
+
g_{\sigma, \mathcal{P}_q}(\overline{t_{\mathcal{P}_q}})
+
g_{\sigma, \mathcal{P}_q}(\delta_{r,\text{even}}).
\end{align*}
Using \eqref{alternative-exp}, for any $\sigma >1/2$ we have 
\begin{align*}
&
\log L_{\mathcal{P}_q} (\Sym_f^{r}, \sigma)
\nonumber\\
=&
\sum_{p \in \mathcal{P}_q}
\sum_{h=0}^{\rho}
\Big(
-\log(1-\alpha_f^{r-2h}(p)p^{-\sigma})
-\log(1-\beta_f^{r-2h}(p)p^{-\sigma})
\Big)
\nonumber\\
&
-
\sum_{p\in\mathcal{P}_q}
\log(1-\delta_{r,\text{even}}p^{-\sigma})
\nonumber\\
=&
\sum_{p \in \mathcal{P}_q}
\sum_{h=0}^{\rho}
\Big(
g_{\sigma, p}(\alpha_f^{r-2h}(p))
+g_{\sigma, p}(\beta_f^{r-2h}(p))
\Big)
+
\sum_{p\in\mathcal{P}_q}
g_{\sigma, p}(\delta_{r,\text{even}})
\nonumber\\
=&
\sum_{h=0}^{\rho}
\Big(
g_{\sigma, \mathcal{P}_q}(\alpha_f^{r-2h}(\mathcal{P}_q))
+g_{\sigma, \mathcal{P}_q}(\beta_f^{r-2h}(\mathcal{P}_q))
\Big)
+
g_{\sigma, \mathcal{P}_q}(\delta_{r,\text{even}}),
\end{align*}
where $\alpha_f^{\mu}(\mathcal{P}_q)=(\alpha_f^{\mu}(p))_{p\in\mathcal{P}_q}$ and $\beta_f^{\mu}(\mathcal{P}_q)=(\beta_f^{\mu}(p))_{p\in\mathcal{P}_q}$.
Especially, in the case $r=1, 2$ which means $\rho=0$, we have
\begin{equation}\label{g_sigma}
\log L_{\mathcal{P}_q} (\Sym_f^r, \sigma)
=
\mathscr{G}_{\sigma, \mathcal{P}_q}(\alpha_f^r(\mathcal{P}_q)).
\end{equation}
We sometimes write $\alpha_f(p)=e^{i\theta_f(p)}$ and $\beta_f(p)=e^{-i\theta_f(p)}$ for $\theta_f(p)\in [0,\pi]$.
\par
In the case $\rho=0$ and $\sigma>1$, we deal with the value $L_{\mathbb{P}_q}(\Sym_f^{r}, s)$ as the limit of the value $L_{\mathcal{P}_q}(\Sym_f^{r}, s)$
as $\mathcal{P}_q$ tends to $\mathbb{P}_q$.
%
In the case $\rho=0$ and $1\geq \sigma>1/2$, we will prove the relation between
$\log L_{\mathbb{P}_q}(\Sym_f^r, \sigma)$ and $\log L_{\mathcal{P}_q}(\Sym_f^r, \sigma)$ with a suitable finite subset $\mathcal{P}_q \subset \mathbb{P}_q$ depending on $q^m$ and will consider the averages of them.
\par

%
%
\section{On the density function ${\mathcal{M}}_{\sigma}(\Sym^r, u)$ and its Fourier transform}\label{sec3}
\par
In this section we first construct the density function ${\mathcal{M}}_{\sigma, \mathcal{P}_q}(\Sym^r, u)$ for a finite set $\mathcal{P}_q \subset \mathbb{P}_q$
in the case $r=1, 2$.
By $|\mathcal{P}_q|$ we denote the number of the elements of $\mathcal{P}_q$.
Let
\[
\Theta_{\mathcal{P}_q}
=\prod_{p\in\mathcal{P}_q} [0, \pi),
\]
and define the modified Sato-Tate measure on $\Theta_{\mathcal{P}_q}$ by
\[
d^{\rm ST} \theta_{\mathcal{P}_q}
=\prod_{p\in\mathcal{P}_q}
\bigg(\frac{2\sin^2 \theta_p}{\pi} d\theta_p\bigg),
\]
where $\theta_{\mathcal{P}_q}=(\theta_p)_{p\in\mathcal{P}_q}\in \Theta_{\mathcal{P}_q}$.
We also define the normalized Haar measure on $\mathcal{T}_{\mathcal{P}_q}$ by
\[
d^H t_{\mathcal{P}_q}
=
\prod_{p\in\mathcal{P}_q} d^H t_p
=
\prod_{p\in\mathcal{P}_q} 
 \frac{dt_p}{2\pi it_p}.
 \]

The following proposition is an analogue of \cite[Proposition 3.1]{mu}, 
with the big difference that we use here the modified Sato-Tate measure.
%
%
\begin{prop}~\label{M_P}
Let $r=1$ or $2$ {\rm(}and hence $\rho=0${\rm)}.
For any $\sigma >1/2$, there exists an $\mathbb{R}$-valued, non-negative function $\mathcal{M}_{\sigma, \mathcal{P}_q}(\Sym^r, u)$ defined on $\mathbb{R}$ which satisfies the following two properties.
\begin{itemize}
\item The support of $\mathcal{M}_{\sigma, \mathcal{P}_q}(\Sym^r, u)$ is compact.
\item For any continuous function $\Psi$ on $\mathbb{R}$, we have
\begin{align}\label{prop-1-1}
&\int_{\mathbb{R}}
\mathcal{M}_{\sigma, \mathcal{P}_q}(\Sym^r, u) \Psi(u)
\frac{du}{\sqrt{2\pi}}
=
\int_{\Theta_{\mathcal{P}_q}}
\Psi\big(
\mathscr{G}_{\sigma, \mathcal{P}_q} (e^{i\theta_{\mathcal{P}_q}r})
\big)
d^{\rm ST} \theta_{\mathcal{P}_q}\nonumber\\
&\qquad=
\int_{\mathcal{T}_{\mathcal{P}_q}} 
\Psi(
\mathscr{G}_{\sigma, \mathcal{P}_q}( t_{\mathcal{P}_q}^r)
)
\prod_{p\in\mathcal{P}_q}
\bigg(\frac{t_p^2 -2 + t_p^{-2}}{-2}\bigg)
d^H t_{\mathcal{P}_q}.
\end{align}
\end{itemize}
In particular, taking $\Psi\equiv 1$ in \eqref{prop-1-1}, we have
\begin{align*} 
\int_{\mathbb{R}} \mathcal{M}_{\sigma, \mathcal{P}_q}(\Sym^r, u) 
\frac{du}{\sqrt{2\pi}}=1. 
\end{align*}
\end{prop}
%
%
\begin{proof}
We construct the function $\mathcal{M}_{\sigma, \mathcal{P}_q}(\Sym^r, u)$ by using a method similar to that in \cite{im_2011}.
In the case $|\mathcal{P}_q|=1$ namely $\mathcal{P}_q=\{p\}$, we define a one-to-one correspondence $\theta_p\mapsto u$ from the interval $[0, \pi)$ to another interval $A(\sigma, p)= (-2\log (1+p^{-\sigma})-\log(1-\delta_{r,\text{even}}p^{-\sigma}), -2\log (1-p^{-\sigma})-\log(1-\delta_{r,\text{even}}p^{-\sigma})] \subset \mathbb{R}$
by 
\begin{align*}
u
=&
-\log(1-2p^{-\sigma}\cos \theta_p+p^{-2\sigma})
-\log(1-\delta_{r,\text{even}}p^{-\sigma})
\\
=&
\begin{cases}
-2\log |1-e^{i\theta_p}p^{-\sigma}| & r=1
\\
-\log(1-(2\cos\theta_p+1)p^{-\sigma}+(2\cos\theta_p+1)p^{-2\sigma}
-p^{-3\sigma}) & r=2.
\end{cases}
\end{align*}
%
%
The definition of $\mathcal{M}_{\sigma, \{p\}}(\Sym^r, u)=\mathcal{M}_{\sigma, p}(\Sym^r, u)$ is 
\[
\mathcal{M}_{\sigma, p}(\Sym^r, u)
=
\begin{cases}
  \displaystyle
\frac{\sqrt{2\pi}|1-e^{i\theta_p}p^{-\sigma}|^2\sin^2(\theta_p/r)}{-\pi p^{-\sigma}\sin\theta_p}
  & \displaystyle u \in A(\sigma,  p),\\
  0 & \text{otherwise},
\end{cases}
\]
where we see that
\[
\frac{\sqrt{2\pi}|1-e^{i\theta_p}p^{-\sigma}|^2\sin^2(\theta_p/r)}{-\pi p^{-\sigma}\sin\theta_p}
=
\begin{cases}
\displaystyle \frac{\sqrt{2\pi}|1-e^{i\theta_p}p^{-\sigma}|^2\sin \theta_p}{-\pi p^{-\sigma}} & r=1\\
\displaystyle \frac{\sqrt{2\pi}|1-e^{i\theta_p}p^{-\sigma}|^2\tan(\theta_p/2)}{-2 \pi p^{-\sigma}} & r=2
\end{cases}
\]
for $\theta_p\in [0, \pi)$.
\par
We show that this function satisfies the properties required by Proposition~\ref{M_P}.
In fact, using 
\[
\frac{du}{d\theta_p}
=
-\frac{2p^{-\sigma} \sin \theta_p}{1-2p^{-\sigma}\cos \theta_p +p^{-2\sigma}}
=
-\frac{2p^{-\sigma} \sin \theta_p}{|1-e^{i \theta_p}p^{-\sigma}|^2},
\]
we obtain 
\begin{align}\label{st_p0}
&
\int_{\mathbb{R}}
\Psi(u)    
\mathcal{M}_{\sigma, p}(\Sym^r, u) \frac{du}{\sqrt{2\pi}}
=
\int_{A(\sigma, p)}
\Psi(u) \mathcal{M}_{\sigma, p}(\Sym^r, u) \frac{du}{\sqrt{2\pi}}
\nonumber\\
=&
\int_0^{\pi}
\Psi\Big(-\log(1-2p^{-\sigma}\cos \theta_p +p^{-2\sigma})
-\log(1-\delta_{r,\text{even}}p^{-\sigma})
\Big)
\nonumber\\
&\times
\frac{|1-e^{i\theta_p}p^{-\sigma}|^2\sin^2(\theta_p/r)}{-\pi p^{-\sigma}\sin\theta_p}
\frac{-2p^{-\sigma} \sin \theta_p}{|1-e^{i \theta_p}p^{-\sigma}|^2}
d\theta_p
\nonumber\\
=&
\int_0^{\pi}
\Psi\Big(-\log(1-2p^{-\sigma}\cos \theta_p +p^{-2\sigma})
-\log(1-\delta_{r,\text{even}}p^{-\sigma})
\Big)
\frac{2\sin^2(\theta_p/r)}{\pi}
d\theta_p.
\end{align} 
If $r=2$, (putting $\theta_p=2\theta_p^{\prime}$) we see that the right-hand side is
\begin{align*}
=
\int_0^{\pi/2}
\Psi\Big(-\log(1-2p^{-\sigma}\cos 2\theta_p^{\prime} +p^{-2\sigma})
-\log(1-p^{-\sigma})
\Big)
\frac{4\sin^2\theta_p^{\prime}}{\pi}
d\theta_p^{\prime},
\end{align*}
which is further equal to
\begin{align*}
\int_0^{\pi}
\Psi\Big(-\log(1-2p^{-\sigma}\cos 2\theta_p +p^{-2\sigma})
-\log(1-p^{-\sigma})
\Big)
\frac{2\sin^2\theta_p}{\pi}
d\theta_p,
\end{align*}
because putting $\theta_p^{\prime\prime}=\pi-\theta_p^{\prime}$ we have
\begin{align*}
&
\int_0^{\pi/2}
\Psi\Big(-\log(1-2p^{-\sigma}\cos 2\theta_p^{\prime} +p^{-2\sigma})
-\log(1-p^{-\sigma})
\Big)
\frac{2\sin^2\theta_p^{\prime}}{\pi}
d\theta_p^{\prime}
\\
=&
\int_{\pi/2}^{\pi}
\Psi\Big(-\log(1-2p^{-\sigma}\cos 2\theta_p^{\prime\prime} +p^{-2\sigma})
-\log(1-p^{-\sigma})
\Big)
\frac{2\sin^2\theta_p^{\prime\prime}}{\pi}
d\theta_p^{\prime\prime}.
\end{align*}
Therefore, for $r=1, 2$, we have
\begin{align}\label{st_p}
&
\int_{\mathbb{R}}
\Psi(u)    
\mathcal{M}_{\sigma, p}(\Sym^r, u) \frac{du}{\sqrt{2\pi}}
\nonumber\\
=&
\int_0^{\pi}
\Psi\Big(-\log(1-2p^{-\sigma}\cos (r\theta_p^{\prime}) +p^{-2\sigma})
-\log(1-\delta_{r,\text{even}}p^{-\sigma})
\Big)
\frac{2\sin^2 \theta_p}{\pi}d\theta_p
\nonumber\\
=&
\int_0^{\pi}
\Psi\big(\mathscr{G}_{\sigma, p}(e^{i \theta_p r})\big)
d^{\rm ST} \theta_p
\end{align}
by \eqref{cos-expression}, where $d^{\rm ST}\theta_p=d^{\rm ST}\theta_{\{p\}}$.
\par
We further calculate the right-hand side of \eqref{st_p}.
Since putting $\eta_p=2\pi-\theta_p$ we have
\begin{align*}
&
\int_{\pi}^{2\pi}
\Psi\Big(-\log(1-2p^{-\sigma}\cos(r\theta_p) +p^{-2\sigma})
-\log(1-\delta_{r,\text{even}}p^{-\sigma})
\Big)
\frac{2\sin^2\theta_p}{\pi}d\theta_p
\\
=&
\int_{0}^{\pi}
\Psi\Big(-\log(1-2p^{-\sigma}\cos(r\eta_p) +p^{-2\sigma})
-\log(1-\delta_{r,\text{even}}p^{-\sigma})
\Big)
d^{\rm ST}\eta_p,
\end{align*}
we find that the right-hand side of \eqref{st_p} is equal to
\begin{align}\label{haar_p}
\frac{1}{2}
&\int_0^{2\pi}
\Psi\Big(-\log(1-2p^{-\sigma}\cos(r\theta_p) +p^{-2\sigma})
-\log(1-\delta_{r,\text{even}}p^{-\sigma})
\Big)
d^{\rm ST} \theta_p
\nonumber\\
=&
\frac{1}{2}
\int_0^{2\pi}
\Psi\big(\mathscr{G}_{\sigma, p}(e^{ir\theta_p})
\big)
\frac{2}{\pi}\big(\frac{e^{i \theta_p}-e^{-i \theta_p}}{2i}\big)^2
d\theta_p
\nonumber\\
=&
\int_0^{2\pi}
\Psi\big(
\mathscr{G}_{\sigma, p}(e^{ir\theta_p})
\big)
\frac{1}{\pi}\big(\frac{e^{2i\theta_p}-2+e^{-2i\theta_p}}{-4}\big)
d\theta_p
\nonumber\\
=&
\int_{\mathcal{T}}
\Psi\big(\mathscr{G}_{\sigma, p}(t_p^r)\big)
\frac{1}{\pi}\big(\frac{t_p^2-2+t_p^{-2}}{-4}\big)\frac{dt_p}{it_p}
\nonumber\\
=&
\int_{\mathcal{T}}
\Psi\big(\mathscr{G}_{\sigma, p} (t_p^r)\big)
\frac{t_p^2-2+t_p^{-2}}{-2} d^Ht_p.
\end{align}
Moreover, choosing $\Psi\equiv 1$ in \eqref{haar_p}, we see that
\begin{align}\label{sosuuhitotunobaai}
\int_{\mathbb{R}}
\mathcal{M}_{\sigma, p}(\Sym^r, u) \frac{du}{\sqrt{2\pi}}
=\int_{\mathcal{T}}\frac{t_p^2-2+t_p^{-2}}{-2} d^Ht_p =1.
\end{align}
Therefore all the assertions of the proposition in the case $|\mathcal{P}_q|=1$
are now established.
\par
In the case $\infty > |\mathcal{P}_q|>1$, we construct the function ${\mathcal{M}}_{\sigma, \mathcal{P}_q}(\Sym^r, u)$ by the convolution product of ${\mathcal{M}}_{\sigma, \mathcal{P}'_q}(\Sym^r, u)$ and ${\mathcal{M}}_{\sigma, p}(\Sym^r, u)$ for $\mathcal{P}_q=\mathcal{P}'_q\cup\{p\} \subset\mathbb{P}_q$ inductively, that is
\[
\mathcal{M}_{\sigma, \mathcal{P}_q}(\Sym^r, u)
=
\int_{\mathbb{R}}
\mathcal{M}_{\sigma, \mathcal{P}'_q}(\Sym^r, u')
\mathcal{M}_{\sigma, p}(\Sym^r, u-u')
\frac{du'}{\sqrt{2\pi}}.
\]
It is easy to show that this function satisfies the statements of Proposition~\ref{M_P}.
\end{proof}
%
%
\par
Next, for the purpose of considering $\displaystyle \lim_{|\mathcal{P}_q|\to \infty} \mathcal{M}_{\sigma, \mathcal{P}_q}(\Sym^r, u)$, we consider the Fourier transform of $\mathcal{M}_{\sigma, \mathcal{P}_q}(\Sym^r, u)$ for $\sigma >1/2$. 
When $\mathcal{P}_q=\{p\}$ and $r=1, 2$, we define $\widetilde{\mathcal{M}}_{\sigma, \mathcal{P}_q}(\Sym^r, u)=\widetilde{\mathcal{M}}_{\sigma, p}(\Sym^r, u)$ by
\[
\widetilde{\mathcal{M}}_{\sigma, p}(\Sym^r, x)
=
\int_{\mathbb{R}} \mathcal{M}_{\sigma, p}(\Sym^r, u) \psi_x(u)
\frac{du}{\sqrt{2\pi}},
\]
where $\psi_x(u)=e^{ixu}$ for $x\in\mathbb{R}$.
From \eqref{st_p0}, we see that
\begin{align*} 
\widetilde{\mathcal{M}}_{\sigma, p}(\textrm{Sym}^r, x)
=
\frac{2}{\pi} \int_0^{\pi}
e^{ixF(\theta)}
\sin^2(\theta/r)d\theta,
\end{align*}
where
$F(\theta)=-\log(1-2p^{-\sigma}\cos\theta+p^{-2\sigma})-\log(1-\delta_{r,\text{even}}p^{-\sigma})$.
%
%
\begin{lem}\label{lemma-M-est}
Let $r=1, 2$.
The following two estimates hold:
\begin{align}\label{M-est-1}
|\widetilde{\mathcal{M}}_{\sigma, p}(\Sym^r, x)|\leq 1
\end{align}
for any prime $p$, and there exists a large $p_0$ such that
\begin{align}\label{M-est-2}
\widetilde{\mathcal{M}}_{\sigma, p}(\Sym^r, x)\ll
\frac{p^{\sigma}}{\sqrt{1+|x|}}
\end{align}
for any $p>p_0$.
\end{lem}
%
%
\begin{proof}
The first inequality is easy, because
\begin{align*} 
|\widetilde{\mathcal{M}}_{\sigma, p}(\Sym^r, x)|
\leq&
\frac{2}{\pi }
\int_0^{\pi}
\sin^2(\theta/r)
d\theta
\nonumber\\
=&
  \frac{2}{\pi}
  \int_0^{\pi}
  \frac{1-\cos(2\theta/r)}{2}
d\theta
=1,
\end{align*}
for $r=1, 2$.
To prove the second inequality, we may assume $|x|\geq 1$, because if $|x|<1$,
then \eqref{M-est-2} follows immediately from \eqref{M-est-1}.
We first observe
\begin{align*}
&
\widetilde{\mathcal{M}}_{\sigma, p}(\Sym^r, x)
=
\frac{2}{\pi}\int_0^{\pi} e^{ixF(\theta)}
\sin^2(\theta/r)
d\theta
\\
=&
\begin{cases}
\displaystyle
\frac{2}{\pi}\int_0^{\pi} (e^{ixF(\theta)})'
\frac{(1-2p^{-\sigma}\cos\theta +p^{-2\sigma})}{-2ix p^{-\sigma}}\sin\theta
d\theta & r=1 \\
\displaystyle
\frac{1}{\pi}\int_{0}^{\pi} e^{ixF(\theta)}(1-\cos\theta)d\theta
& r=2,
\end{cases}
\end{align*}
because
\begin{align*}
F^{\prime}(\theta)=\frac{-2p^{-\sigma}\sin\theta}
{1-2p^{-\sigma}\cos\theta+p^{-2\sigma}}.
\end{align*}
When $r=1$, the above is equal to
\[
\frac{2}{\pi}\int_0^{\pi} e^{ixF(\theta)}
\frac{(\cos\theta-2p^{-\sigma}\cos(2\theta) +p^{-2\sigma}\cos\theta)}{2ix p^{-\sigma}}
d\theta
\]
by integration by parts, and hence it is $O(p^{\sigma}|x|^{-1})$ which is sufficient.
When $r=2$, the above is
\[
=\frac{1}{\pi}\int_0^{\pi}e^{ixF(\theta)}d\theta-
\frac{1}{\pi}\int_0^{\pi}e^{ixF(\theta)}\cos\theta d\theta
=\frac{1}{\pi}I_1(x)-\frac{1}{\pi}I_2(x),
\]
say.
The estimate
\begin{align}\label{est-I-1}
I_1(x)\ll p^{\sigma/2}|x|^{-1/2}
\end{align}
follows from the Jessen-Wintner inequality (see \cite[Theorem 13]{jw} or \cite[Proposition 7.1]{mu-JNT}) for large $p$.
To evaluate $I_2(x)$, we divide
\begin{align*}
I_2(x)
&=
\frac{2}{\pi}\int_{\substack{0\leq \theta \leq \sqrt{|x|}^{-1}\\\pi-\sqrt{|x|}^{-1}\leq \theta \leq \pi}} e^{ixF(\theta)}\cos\theta d\theta
+
\frac{2}{\pi}\int_{\sqrt{|x|}^{-1}}^{\pi-\sqrt{|x|}^{-1}} e^{ixF(\theta)}\cos\theta d\theta\\
&= I_{21}(x)+I_{22}(x),
\end{align*}
say.
The inequalities
\begin{align*}
\bigg|\frac{xF'(\theta)}{\cos\theta}\bigg|
=
\bigg|\frac{x2p^{-\sigma}\tan\theta}{1-2p^{-\sigma}\cos\theta +p^{-2\sigma}}\bigg|
\geq
\frac{|x|2p^{-\sigma}\theta}{1+2p^{-\sigma}+p^{-2\sigma}}
\\
\bigg|\frac{xF'(\pi-\theta)}{\cos\theta}\bigg|
=
\bigg|\frac{-x2p^{-\sigma}\tan\theta}{1+2p^{-\sigma}\cos\theta +p^{-2\sigma}}\bigg|
\geq
\frac{|x|2p^{-\sigma}\theta}{1+2p^{-\sigma}+p^{-2\sigma}}
\end{align*}
for $0\leq \theta\leq \pi/2$ yield
\begin{align*}
I_{22}(x)
=&
\frac{2}{\pi}\int_{\sqrt{|x|}^{-1}}^{\pi/2} e^{ixF(\theta)}\cos\theta d\theta
+
\frac{2}{\pi}\int_{\pi/2}^{\pi-\sqrt{|x|}^{-1}} e^{ixF(\theta)}\cos\theta d\theta
\\
=&
\frac{2}{\pi}\int_{\sqrt{|x|}^{-1}}^{\pi/2} e^{ixF(\theta)}\cos\theta d\theta
-
\frac{2}{\pi}\int_{\sqrt{|x|}^{-1}}^{\pi/2} e^{ixF(\pi-\theta)}\cos\theta d\theta
\ll
\frac{p^{\sigma}}{\sqrt{|x|}}
\end{align*}
for large $p$ by the first derivative test (Titchmarsh \cite[Lemma 4.3]{tit}), while trivially
\[
|I_{21}(x)|
\leq
\frac{2}{\pi}\int_{\substack{0\leq \theta \leq \sqrt{|x|}^{-1}\\\pi-\sqrt{|x|}^{-1}\leq \theta \leq \pi}} d\theta
\ll
\frac{1}{\sqrt{|x|}}.
\]
Therefore 
\begin{equation}\label{est-I-2}
I_2(x)\ll p^{\sigma}|x|^{-1/2}.
\end{equation}
The desired estimate follows from \eqref{est-I-1} and \eqref{est-I-2}.
\end{proof}
\par
Now we define $\widetilde{\mathcal{M}}_{\sigma, \mathcal{P}_q}(\Sym^r, x)$ for $r=1, 2$ by
\begin{align}\label{FFourier-transf}
\widetilde{\mathcal{M}}_{\sigma, \mathcal{P}_q}(\Sym^r, x)
=
\prod_{p \in \mathcal{P}_q}
\widetilde{\mathcal{M}}_{\sigma, p}(\Sym^r, x).
\end{align}
The estimate \eqref{M-est-1} gives
\begin{equation}\label{trivial_P}
\left|\widetilde{\mathcal{M}}_{\sigma, \mathcal{P}_q}(\Sym^r, x)\right|
\leq
\prod_{p\in\mathcal{P}_q}
\frac{2}{\pi}\int_0^{\pi}\sin^2\theta_p d\theta_p
=1
\end{equation}
for any finite $\mathcal{P}_q$.
On the other hand, let $n\geq 3$, and let $\mathcal{P}_q^{0,n}\subset\mathbb{P}_q$ be a fixed finite set, $|\mathcal{P}_q^{0,n}|=n$, and of which all elements are $>p_0$.
If $\mathcal{P}_q\supset\mathcal{P}_q^{0,n}$, then we apply \eqref{M-est-2} to all $p\in\mathcal{P}_q^{0,n}$ and apply \eqref{M-est-1} to all the other members of $\mathcal{P}_q$ to obtain
\begin{align}\label{combined}
\widetilde{\mathcal{M}}_{\sigma, \mathcal{P}_q}(\Sym^r, x)
=O_n\left(\frac{\prod_{p\in\mathcal{P}_q^{0,n}}p^{\sigma}}{(1+|x|)^{n/2}}\right).
\end{align}
In what follows we mainly use the case $n=3$, so we write $\mathcal{P}_q^{0,3}=\mathcal{P}_q^0$.
Denote the largest element of $\mathcal{P}_q^0$ by $y_0$.
%
%
\par
The next lemma is an analogue of (a), (b) and (c) given in \cite[Section 3, pp. 645--646]{im_2011}.
Here we introduce a parameter $y$.    
We understand the symbol $\displaystyle{\lim_{y\to\infty}}$ in the
following sense:
\begin{align}\label{cases-I-II}
\begin{cases}
y>q\; \text{and}\; y\to\infty \quad \text{(in Case I)}\\
y=y(q)<q \;\text{and}\; y\to\infty \;\text{as} \;q\to\infty \quad 
\text{(in Case II)}.
\end{cases}
\end{align}
%
%
\begin{lem}\label{lem-abc}
Let $r=1, 2$.
The following assertions hold for any $\sigma >1/2$.
\begin{itemize}
\item[{\rm(a)}.]
Let $\mathcal{P}_q \subset \mathbb{P}_q$ be a finite subset with $\mathcal{P}_q\supset\mathcal{P}_q^0$.
Then $\widetilde{\mathcal{M}}_{\sigma, \mathcal{P}_q}(\Sym^r, x) \in L^t$ $(t\in [1,\infty])$.
\item[{\rm(b)}.]
For any finite subsets $\mathcal{P}''_q$ and $\mathcal{P}'_q$ of $\mathbb{P}_q$ with $\mathcal{P}''_q\subset \mathcal{P}'_q$, we have
\[
\left|\widetilde{\mathcal{M}}_{\sigma, \mathcal{P}'_q}(\Sym^r, x)\right|
\leq
\left|\widetilde{\mathcal{M}}_{\sigma, \mathcal{P}''_q}(\Sym^r, x)\right|.
\]
\item[{\rm(c)}.]
Put $\mathcal{P}_q(y)=\{ p \in \mathbb{P}_q \mid  p \leq y \} $.
The limit $\displaystyle \lim_{y\to\infty} \widetilde{\mathcal{M}}_{\sigma, \mathcal{P}_q(y)}(\Sym^r, x)$ exists, which we denote by 
$\widetilde{\mathcal{M}}_{\sigma}(\Sym^r, x)=
\widetilde{\mathcal{M}}_{\sigma,\mathbb{P}_*}(\Sym^r, x)$.
For any $a >0$, this convergence is uniform on $|x|\leq a$.
Moreover the infinite product  
$$
\prod_{p\in\mathbb{P}_*}\widetilde{\mathcal{M}}_{\sigma, p}(\Sym^r, x)
=\lim_{y\to\infty}
\prod_{p\in\mathcal{P}_q(y)}\widetilde{\mathcal{M}}_{\sigma, p}(\Sym^r, x)
$$
is absolutely convergent, and is equal to
$\widetilde{\mathcal{M}}_{\sigma}(\Sym^r, x)$.
\end{itemize}
\end{lem}
%
%
\begin{proof}
First prove (a):
By \eqref{trivial_P} and \eqref{combined} (with $n=3$), we find
\begin{align}\label{lem2-a-pr}
\int_{\mathbb{R}}|\widetilde{\mathcal{M}}_{\sigma, \mathcal{P}_q}(\textrm{Sym}^r, x)|^t|dx|
\ll 
\int_{|x|< 1} |dx|
+
\int_{|x|\geq 1}
\frac{\prod_{p\in\mathcal{P}_q^0}p^{t\sigma}}{|x|^{3t/2}} |dx|
< 
\infty
\end{align}
for $t\in [1, \infty)$.
Note that the implied constant here depends only on $\sigma,t$, and $\mathcal{P}_q^0$.
Also from \eqref{trivial_P}, we see that $\sup_x |\widetilde{\mathcal{M}}_{\sigma, \mathcal{P}_q}(\Sym^r, x)| < \infty $.
The assertion (a) is proved.
\par
The property (b) is easy by \eqref{M-est-1}. 
\par
We proceed to the proof of (c).
Let $y_1$ and $y_2$ be integers satisfying the same condition
\eqref{cases-I-II} as $y$, and $y_2\geq y_1$.
Let $p(i)$ $(1\leq i \leq d)$ be all prime numbers in $(y_1, y_2]$ with
\[
y_2\geq p(d) >p(d-1) > \ldots > p(1) > y_1.
\]
Set the following subsets of $\mathbb{P}_q$:
\begin{align*}
\mathcal{P}_q(y_1, 0)=&\mathcal{P}_q(y_1),
\\
\mathcal{P}_q(y_1, 1)=&\mathcal{P}_q(y_1, 0)\cup \{p(1)\},
\\
\vdots &
\\
\mathcal{P}_q(y_1, d)=&\mathcal{P}_q(y_1, d-1)\cup \{p(d)\}.
\end{align*}
First we note that
\begin{align}\label{wp_00}
\mathscr{G}_{\sigma, p}(e^{i\theta})
&=
-\log(1-e^{ir\theta}p^{-\sigma})-\log(1-e^{-ir\theta}p^{-\sigma})
-\log(1-\delta_{r,\text{even}}p^{-\sigma})\nonumber\\
&\ll
\frac{1}{p^{\sigma}}.
\end{align}
Next we show that
\begin{align}\label{wp_0}
-\frac{2}{\pi}\int_0^{\pi}
\mathscr{G}_{\sigma,p}(e^{ir\theta})\sin^2\theta
d\theta
\ll 
\frac{1}{p^{2\sigma}}.
\end{align}
In fact, the left-hand side is
\begin{align}
=&
-\frac{2}{\pi}\int_0^{\pi}
\Big(-\log(1-2p^{-\sigma}\cos(r\theta) +p^{-2\sigma})
-\log(1-\delta_{r,\text{even}}p^{-\sigma})
\Big)
\sin^2{\theta}
d\theta
\nonumber\\
=&
-\frac{1}{\pi}\int_0^{2\pi}
\Big(-\log(1-2p^{-\sigma}\cos(r\theta) +p^{-2\sigma})
-\log(1-\delta_{r,\text{even}}p^{-\sigma})
\Big)
\sin^2{\theta}
d\theta
\nonumber\\
=&
-\frac{1}{\pi}
\int_0^{2\pi}
(-\log(1-e^{ir\theta}p^{-\sigma})-\log(1-e^{-ir\theta}p^{-\sigma})
-\log(1-\delta_{r,\text{even}}p^{-\sigma}))
\nonumber\\
&\times
\bigg(\frac{e^{2i\theta}-2+e^{-2i\theta}}{-4}\bigg)
d\theta
\nonumber\\
=&
\frac{1}{\pi}
\int_0^{2\pi}
\bigg(
\sum_{h=1}^{\infty}\frac{e^{hir\theta}}{hp^{h\sigma}}
+\sum_{k=1}^{\infty}\frac{e^{-k ir\theta}}{k p^{k\sigma}}
+\sum_{\ell=1}^{\infty}\frac{\delta_{r,\text{even}}}{\ell p^{\ell\sigma}}
\bigg)
\bigg(\frac{e^{2i\theta}-2+e^{-2i\theta}}{+4}\bigg)
d\theta.
\nonumber
\end{align}
We compute the right-hand side of the above termwisely.
When $r=1$, only the contributions of the terms corresponding to $h=2$ and $k=2$
remain, and the right-hand side is
\[
=\frac{1}{4\pi}\int_0^{2\pi}\left(\frac{1}{2p^{2\sigma}}+\frac{1}{2p^{2\sigma}}
\right) d\theta=\frac{1}{2p^{2\sigma}}.
\]
When $r=2$, the remaining terms are those corresponding to $h=1$, $k=1$ and $\ell=1,2,\ldots$, and hence the right-hand side is
\begin{align}\label{kokogadaiji}
=\frac{1}{4\pi}\int_0^{2\pi}
\Big(\frac{1}{p^{\sigma}}+\frac{1}{p^{\sigma}}
-2\sum_{\ell=1}^{\infty}\frac{1}{\ell p^{\ell\sigma}}
\Big)
d\theta
=-\sum_{\ell=2}^{\infty}\frac{1}{\ell p^{\ell\sigma}}
\ll \frac{1}{p^{2\sigma}}.
\end{align}
Therefore we obtain \eqref{wp_0}.
(Here, it is an important point in the proof that the terms of the form $1/p^{\sigma}$ are finally cancelled with each other in the computations \eqref{kokogadaiji}.) 
\par
Now we evaluate the following difference:
\begin{align}\label{wp_1}
&
\left|
\widetilde{\mathcal{M}}_{\sigma, \mathcal{P}_q(y_1, d)}(\Sym^r, x)
-
\widetilde{\mathcal{M}}_{\sigma, \mathcal{P}_q(y_1, 0)}(\Sym^r, x)
\right|
\nonumber\\
\leq &
\sum_{i=1}^{d}
\left|
\widetilde{\mathcal{M}}_{\sigma, \mathcal{P}_q(y_1, i)}(\Sym^r, x)
-
\widetilde{\mathcal{M}}_{\sigma, \mathcal{P}_q(y_1, i-1)}(\Sym^r, x)
\right|
\nonumber\\
\leq&
\sum_{i=1}^{d}
\left|
\widetilde{\mathcal{M}}_{\sigma, \mathcal{P}_q(y_1, i-1)}(\Sym^r, x)
\right|
\left|
\widetilde{\mathcal{M}}_{\sigma, p(i)}(\Sym^r, x)
-1
\right|
\nonumber\\
\leq&
\sum_{i=1}^{d}
\left|
\widetilde{\mathcal{M}}_{\sigma, p(i)}(\Sym^r, x)
-1
\right|,
\end{align}
where on the last inequality we used \eqref{M-est-1}.
Here we show
\begin{align}\label{1+error}
\widetilde{\mathcal{M}}_{\sigma, p}(\Sym^r, x)
=1+O\bigg(\frac{|x|+|x|^2}{p^{2\sigma}}\bigg).
\end{align}
In fact, applying \eqref{st_p} (with $\Psi=\psi_x$), we have
\begin{align*}
&\widetilde{\mathcal{M}}_{\sigma, p}(\Sym^r, x)
=\frac{2}{\pi}\int_0^{\pi}
\psi_x\Big(\mathscr{G}_{\sigma, p}(e^{ir\theta})
\Big)\sin^2\theta d\theta\\
&=\frac{2}{\pi}\int_0^{\pi}
(1+ix\mathscr{G}_{\sigma, p}(e^{ir\theta})
+O(|x|^2|\mathscr{G}_{\sigma, p}(e^{ir\theta})|^2)
)\sin^2\theta d\theta,
\end{align*}
and hence, applying \eqref{wp_00}, \eqref{wp_0} and the fact
\[
\frac{2}{\pi}\int_0^{\pi}\sin^2\theta d\theta=1,
\]
we obtain \eqref{1+error}.
From \eqref{1+error} we find that the right-hand side of \eqref{wp_1} is
\begin{align*}
\ll&
\sum_{i=1}^{d}\frac{|x|+|x|^2}{p(i)^{2\sigma}}.
\end{align*}
Therefore, for any $\varepsilon>0$, there exists large number $y=y(x,\sigma,\varepsilon)$, satisfying \eqref{cases-I-II} and $y_2 >y_1 > y$, such that
\[
\left|
\widetilde{\mathcal{M}}_{\sigma, \mathcal{P}_q(y_2)}(\Sym^r, x)
-
\widetilde{\mathcal{M}}_{\sigma, \mathcal{P}_q(y_1)}(\Sym^r, x)
\right|
<\varepsilon
\]
for any $\sigma>1/2$.
This implies that the limit
\[
\lim_{y\to\infty} \widetilde{\mathcal{M}}_{\sigma, \mathcal{P}_q(y)}(\Sym^r, x)
=\lim_{y\to\infty} \prod_{p\in \mathcal{P}_q(y)}\widetilde{\mathcal{M}}_{\sigma, p}
(\Sym^r, x)
\]
exists, 
and its convergence is uniform in $|x|\leq a$ for any real number $a\in\mathbb{R}$.
Note that as $y\to\infty$, $\mathcal{P}_q(y) \to \mathbb{P}_q$ in Case I,
and $\mathcal{P}_q(y) \to \mathbb{P}$ in Case II.
Moreover, because of \eqref{1+error}, it follows that the infinite product $\prod_{p\in\mathbb{P}_*} \widetilde{\mathcal{M}}_{\sigma, p}(\Sym^r, x)$ is absolutely convergent and it is uniform in $|x|\leq a$ for any $\sigma > 1/2$.
\end{proof}
%
%
\begin{rem}\label{M-p-conti}
For $r=1, 2$, the Fourier inverse transform gives
\begin{align}\label{p-Fourierinverse}
\int_{\mathbb{R}}
\widetilde{\mathcal{M}}_{\sigma, \mathcal{P}_q}(\Sym^r, x) \psi_{-u}(x)
\frac{dx}{\sqrt{2\pi}}
=
\mathcal{M}_{\sigma, \mathcal{P}_q}(\Sym^r, u).
\end{align}
The case $t=1$ of Lemma~\ref{lem-abc} (a) implies that $\mathcal{M}_{\sigma, \mathcal{P}_q}(\Sym^r, u)$ is continuous if $\mathcal{P}_q\supset\mathcal{P}_q^0$.
\end{rem}
%
%
\begin{rem}\label{M-p-supp}
We have shown in Proposition~\ref{M_P} that the support of the function $\mathcal{M}_{\sigma, \mathcal{P}_q}(\Sym^r, u)$, which we denote by ${\rm Supp}(\mathcal{M}_{\sigma, \mathcal{P}_q})$, is compact.
Here we show that when $\mathcal{P}\supset\mathcal{P}_q^0$, then ${\rm Supp}(\mathcal{M}_{\sigma, \mathcal{P}_q})\subset C_{\sigma}(\mathcal{P}_q)$, where $C_{\sigma}(\mathcal{P}_q)$ denotes the closure of the image of $\mathscr{G}_{\sigma, \mathcal{P}_q}$.
In fact, for any $x\notin C_{\sigma}(\mathcal{P}_q)$, we can choose a non-negative
continuous function $\Psi_x$ such that $\Psi_x(x)>0$ and ${\rm Supp}(\Psi_x)\cap C_{\sigma}(\mathcal{P}_q)=\emptyset$.
Then the right-hand side of \eqref{prop-1-1} with $\Psi=\Psi_x$ is equal to $0$, and so is the left-hand side.
However, if $x\in {\rm Supp}(\mathcal{M}_{\sigma, \mathcal{P}_q})$, then the left-hand side of \eqref{prop-1-1} with $\Psi=\Psi_x$ is positive, because $\mathcal{M}_{\sigma, \mathcal{P}_q}(\Sym^r, u)$ is non-negative and continuous by Remark~\ref{M-p-conti}.
Therefore $x\notin {\rm Supp}(\mathcal{M}_{\sigma, \mathcal{P}_q})$, and hence the assertion.
\end{rem}
Lemma~\ref{lem-abc} yields the next proposition which is the analogue of \cite[Proposition~3.4]{im_2011} and \cite[Proposition~3.2]{mu} (whose idea goes back to \cite{ihara}).
Since the proof is omitted in \cite{im_2011} and \cite{mu}, here we describe the detailed proof.
%
%
\begin{prop}\label{tildeM}
Let $r=1, 2$.
For any $\sigma >1/2$, the convergence of the limit
\[
\widetilde{\mathcal{M}}_{\sigma}(\Sym^r, x)
=
\lim_{y\to\infty}
\widetilde{\mathcal{M}}_{\sigma, \mathcal{P}_q(y)}(\Sym^r, x)
\]
is uniform in $x\in \mathbb{R}$.
Furthermore, this convergence is $L^t$-convergence and the function $\widetilde{\mathcal{M}}_{\sigma}(\Sym^r, x)$ belongs to $ L^t$ $(1\leq t \leq \infty)$.
\end{prop}
%
%
\begin{proof}
By \eqref{lem2-a-pr}, for any $\mathcal{P}_q\supset\mathcal{P}_q^0$, for $t\in [1,\infty)$ we see that
\[
\int_{\mathbb{R}}
\left|
\widetilde{\mathcal{M}}_{\sigma,\mathcal{P}_q}(\Sym^r, x)
\right|^t |dx|
\]
is bounded by a constant depending only on $\sigma,t$, and $\mathcal{P}_q^0$.
Therefore, for any positive real number $\varepsilon>0$, there exists $R_t=R_t(\varepsilon, \sigma, \mathcal{P}_q^0)>1$ such that
\begin{equation}\label{R}
\int_{|x|>R_t}
\left|
\widetilde{\mathcal{M}}_{\sigma,\mathcal{P}_q}(\Sym^r, x)
\right|^t |dx|
<
\frac{\varepsilon}{2^{t+1}}.
\end{equation}
Using \eqref{combined} we can choose $R_{\infty}$ sufficiently large so that
\begin{equation}\label{R1}
\sup_{|x|>R_{\infty}}
\left|\widetilde{\mathcal{M}}_{\sigma,\mathcal{P}_q}(\Sym^r, x)\right|
<
\frac{\varepsilon}{2}
\end{equation}
also holds.
Let $y$ be a positive number satisfying \eqref{cases-I-II} and 
$\mathbb{P}_q \supset \mathcal{P}_q(y) \supset \mathcal{P}_q^0$.   Choose $Y$  
which satisfies $q<y<Y$ (in Case I) or $y=y(q)<Y<q$ (in Case II), and hence
 $\mathbb{P}_q \supset \mathcal{P}_q(Y) \supset \mathcal{P}_q(y)$.
Since \eqref{M-est-1} in Lemma~\ref{lemma-M-est} and (c) of Lemma~\ref{lem-abc} yield that
\begin{align}\label{prop2}
&
\left|
\widetilde{\mathcal{M}}_{\sigma}(\Sym^r, x)
-
\widetilde{\mathcal{M}}_{\sigma, \mathcal{P}_q(y)}(\Sym^r, x)
\right|^t
\nonumber\\
=&
\lim_{Y\to\infty}
\Big|
\prod_{p\in \mathcal{P}_q(Y)}
\widetilde{\mathcal{M}}_{\sigma, p}(\Sym^r, x)
-
\widetilde{\mathcal{M}}_{\sigma, \mathcal{P}_q(y)}(\Sym^r, x)
\Big|^t
\nonumber\\
=&
\lim_{Y\to\infty}
\Big|
\prod_{p\in \mathcal{P}_q(Y)\setminus\mathcal{P}_q(y)}
\widetilde{\mathcal{M}}_{\sigma, p}(\Sym^r, x)
- 1
\Big|^t
\Big|
\widetilde{\mathcal{M}}_{\sigma, \mathcal{P}_q(y)}(\Sym^r, x)
\Big|^t
\nonumber\\
\leq&
\lim_{Y\to\infty}
2^t
\Big|
\widetilde{\mathcal{M}}_{\sigma, \mathcal{P}_q(y)}(\Sym^r, x)
\Big|^t
=
2^t
\Big|
\widetilde{\mathcal{M}}_{\sigma, \mathcal{P}_q(y)}(\Sym^r, x)
\Big|^t,
\end{align}
from \eqref{R} (with $\mathcal{P}_q=\mathcal{P}_q^0$) we see that
\begin{align}\label{p2-proof-1}
&
\int_{|x|>R_t}
\left|
\widetilde{\mathcal{M}}_{\sigma}(\Sym^r, x)
-
\widetilde{\mathcal{M}}_{\sigma, \mathcal{P}_q(y)}(\Sym^r, x)
\right|^t |dx|\nonumber
\\
\leq&
2^t\int_{|x|>R_t}
\Big|\widetilde{\mathcal{M}}_{\sigma, \mathcal{P}_q(y)}(\Sym^r, x)
\Big|^t|dx|
\leq
\frac{\varepsilon}{2}.
\end{align}
From (c), we see that there exists $y'=y'(\varepsilon,\sigma,t,R_t)=y'(\varepsilon,\sigma,t,\mathcal{P}_q^0)>0$ such that
\[
\left|
\widetilde{\mathcal{M}}_{\sigma}(\Sym^r, x)
-
\widetilde{\mathcal{M}}_{\sigma, \mathcal{P}_q(y)}(\Sym^r, x)
\right|
<
\bigg(\frac{\varepsilon}{2}\bigg)^{1/t}
\bigg(\frac{1}{2R_t}\bigg)^{1/t}
\]
for $|x|\leq R_t$ and any $y>y'$.
Then we have
\begin{align}\label{p2-proof-2}
&
\int_{|x|\leq R_t}
\left|
\widetilde{\mathcal{M}}_{\sigma}(\Sym^r, x)
-
\widetilde{\mathcal{M}}_{\sigma, \mathcal{P}_q(y)}(\Sym^r, x)
\right|^t|dx|
\leq
\int_{|x|\leq R_t}\frac{\varepsilon}{4R_t}|dx|
=
\frac{\varepsilon}{2}.
\end{align}
From \eqref{p2-proof-1} and \eqref{p2-proof-2} we have
\[
\int_{\mathbb{R}} \left|
\widetilde{\mathcal{M}}_{\sigma}(\Sym^r, x)
-
\widetilde{\mathcal{M}}_{\sigma, \mathcal{P}_q(y)}(\Sym^r, x)
\right|^t
|dx|
<
\varepsilon,
\]
for any $y>y'$, and so $\widetilde{\mathcal{M}}_{\sigma}(\Sym^r, x)\in L^t$
for $t\in [1, \infty)$.
\par
As for the case $t=\infty$, first notice that there exists $y''=y''(\varepsilon,\sigma,R_{\infty})=y''(\varepsilon,\sigma,\mathcal{P}_q^0)$ with $\mathbb{P}_q \supset \mathcal{P}_q(y'') \supset \mathcal{P}_q^0$ such that
\[
\sup_{|x|\leq R_{\infty}}
\left|
\widetilde{\mathcal{M}}_{\sigma}(\Sym^r, x)
-
\widetilde{\mathcal{M}}_{\sigma,\mathcal{P}_q(y'')}(\Sym^r, x)
\right|
<
\varepsilon
\]
by (c).
From \eqref{R1} and \eqref{prop2} (with $t=1$), we see that
\begin{align*}
&
\sup_{|x|> R_{\infty}}
\left|
\widetilde{\mathcal{M}}_{\sigma}(\Sym^r, x)
-
\widetilde{\mathcal{M}}_{\sigma,\mathcal{P}_q(y'')}(\Sym^r, x)
\right|
\\
<&
2\sup_{|x|> R_{\infty}}
|\widetilde{\mathcal{M}}_{\sigma, \mathcal{P}_q(y'')}(\Sym^r, x)|
<
\varepsilon,
\end{align*}
and hence the case $t=\infty$ follows.
\end{proof}
%
%
\begin{rem}\label{unif-in-sigma}
The convergence in the above proposition is also uniform in $\sigma$, in the region $\sigma\geq 1/2+\varepsilon$ (for any $\varepsilon>0$).
\end{rem}
Let $n\geq 3$.
Recall that \eqref{combined} holds for any $\mathcal{P}_q(y)\supset\mathcal{P}_q^{0,n}$.
Therefore Proposition~\ref{tildeM} implies
\begin{equation}\label{tildeJW}
\widetilde{\mathcal{M}}_{\sigma}(\Sym^r, x)
=O_n\left( (1+|x|)^{-n/2}\right),
\end{equation}
where the implied constant depends on $\sigma$ and $\mathcal{P}_q^{0,n}$.
We also have
\begin{equation}\label{tildeTrivial}
|\widetilde{\mathcal{M}}_{\sigma}(\Sym^r, x)|\leq 1
\end{equation}
from \eqref{trivial_P}.
\par
Finally, we define the function $\mathcal{M}_{\sigma}(\Sym^r, u)
=\mathcal{M}_{\sigma,\mathbb{P}_*}(\Sym^r, u)$ for $r=1, 2$ by
\begin{align}\label{Fourierinverse}
\mathcal{M}_{\sigma}(\Sym^r, u)
=
\int_{\mathbb{R}}\widetilde{\mathcal{M}}_{\sigma}(\Sym^r, x)\psi_{-u}(x)
\frac{dx}{\sqrt{2\pi}}
\end{align}
(analogous to \eqref{p-Fourierinverse}), where we can see that the right-hand side of this equation is absolutely convergent since $\widetilde{\mathcal{M}}_{\sigma}(\Sym^r, x)\in L^1$ (or, more quantitatively, by \eqref{tildeJW}).
%
%
\begin{prop}\label{M}
For $\sigma>1/2$,
the function $\mathcal{M}_{\sigma}(\Sym^r, u)
=\mathcal{M}_{\sigma,\mathbb{P}_*}(\Sym^r, u)$ satisfies the following five properties.
\begin{itemize}
\item[\rm{(1)}] $\displaystyle \lim_{y\to\infty} \mathcal{M}_{\sigma, \mathcal{P}_q(y)}(\Sym^r, u)= \mathcal{M}_{\sigma}(\Sym^r, u)$ and this convergence is uniform in $u$.
\item[\rm{(2)}] The function $\mathcal{M}_{\sigma}(\Sym^r, u)$ is continuous in $u$ and non-negative.
\item[\rm{(3)}] $\displaystyle \lim_{|u| \to\infty} \mathcal{M}_{\sigma}(\Sym^r, u)=0$.
\item[\rm{(4)}] The functions $\mathcal{M}_{\sigma}(\Sym^r, u)$ and $\widetilde{\mathcal{M}}_{\sigma}(\Sym^r, x)$ are Fourier duals of each other.
\item[\rm{(5)}] $\displaystyle \int_{\mathbb{R}} \mathcal{M}_{\sigma}(\Sym^r, u) \frac{du}{\sqrt{2\pi}}=1$.
\end{itemize}
\end{prop}
%
%
This is the analogue of \cite[Proposition 3.5]{im_2011} and
the proof is similar.
\begin{proof}
\par
(1).
From \eqref{p2-proof-1} with $t=1$ we have
\begin{align*} 
&
\int_{|x|>R_1}
\left|
\widetilde{\mathcal{M}}_{\sigma}(\Sym^r, x)
-
\widetilde{\mathcal{M}}_{\sigma, \mathcal{P}_q(y)}(\Sym^r, x)
\right| |dx|\nonumber
\\
\leq&
2\int_{|x|>R_1}
\Big|\widetilde{\mathcal{M}}_{\sigma, \mathcal{P}_q(y)}(\Sym^r, x)
\Big||dx|
\leq
\frac{\varepsilon}{2}
\end{align*}
for $\mathcal{P}_q^0\subset\mathcal{P}_q(y)$.
On the other hand, from (c), there exists $y'''=y'''(\varepsilon,\sigma,R_1)>0$ such that
\[
\left|
\widetilde{\mathcal{M}}_{\sigma}(\Sym^r, x)
-
\widetilde{\mathcal{M}}_{\sigma, \mathcal{P}_q(y)}(\Sym^r, x)
\right|
<
\frac{\varepsilon}{4R_1}
\]
for any $x$ with $|x|\leq R_1$ and for any $y>y'''$.
Then we have
\[
\int_{|x|< R_1}
\left|
\widetilde{\mathcal{M}}_{\sigma}(\Sym^r, x)
-
\widetilde{\mathcal{M}}_{\sigma, \mathcal{P}_q(y)}(\Sym^r, x)
\right|
dx
<
\frac{\varepsilon}{2}.
\]
Therefore for any $y>y'''$ with $\mathbb{P}_q\supset \mathcal{P}_q(y)\supset\mathcal{P}_q^0$ we have 
\begin{align*}
&
\left|
\mathcal{M}_{\sigma}(\Sym^r, u)
-
\mathcal{M}_{\sigma, \mathcal{P}_q(y)}(\Sym^r, u)
\right|
\\
\leq& 
\int_{\mathbb{R}}
\left|
\widetilde{\mathcal{M}}_{\sigma}(\Sym^r, x)
-
\widetilde{\mathcal{M}}_{\sigma, \mathcal{P}_q(y)}(\Sym^r, x)
\right|
dx
<
\varepsilon.
\end{align*}
\par
(2).
Since $\widetilde{\mathcal{M}}_{\sigma}(\Sym^r, x)\in L^1$ by Proposition~\ref{tildeM}, $\mathcal{M}_{\sigma}(\Sym^r, u)$ is continuous.
Since $\mathcal{M}_{\sigma, \mathcal{P}_q(y)}(\Sym^r, u)$ is a non-negative function for any $y > 1$, $\mathcal{M}_{\sigma}(\Sym^r, u)$ is also non-negative by (1).
\par  
(3).
For any  $\varepsilon>0$, there exists $y>0$ such that
\[
\left|
\mathcal{M}_{\sigma}(\Sym^r, u)
-
\mathcal{M}_{\sigma,\mathcal{P}_q(y)}(\Sym^r, u)
\right|
<
\varepsilon
\]
for any $u$ by (1).
Therefore we see that
\[
\left|
\mathcal{M}_{\sigma}(\Sym^r, u)
\right|
<
\left|
\mathcal{M}_{\sigma,\mathcal{P}_q(y)}(\Sym^r, u)
\right|
+
\varepsilon.
\]
Since the support of $\mathcal{M}_{\sigma,\mathcal{P}_q(y)}(\Sym^r, u)$ is compact, we see that
\[
\lim_{|u|\to\infty} \mathcal{M}_{\sigma}(\Sym^r, u)\leq\varepsilon,
\]
and since $\varepsilon$ is arbitrary, the limit should be $0$.
\par
(4).
We know $\widetilde{\mathcal{M}}_{\sigma}(\Sym^r, x)$ is in $L^2$.
By the inversion formula, we have
\begin{align}\label{prop-3-pr-4}
\int_{\mathbb{R}} \mathcal{M}_{\sigma}(\Sym^r, u)
\psi_x(u)\frac{du}{\sqrt{2\pi}}
=
\widetilde{\mathcal{M}}_{\sigma}(\Sym^r, x).
\end{align}
\par
(5). 
Putting $x=0$ in \eqref{prop-3-pr-4}, we have
\[
\widetilde{\mathcal{M}}_{\sigma}(\Sym^r, 0)
=
\int_{\mathbb{R}} \mathcal{M}_{\sigma}(\Sym^r, u) \frac{du}{\sqrt{2\pi}}.
\]
On the other hand, from (c) we know
\[
\widetilde{\mathcal{M}}_{\sigma}(\Sym^r, 0)
=
\lim_{y\to\infty}
\prod_{p\in \mathbb{P}_q(y)}
\widetilde{\mathcal{M}}_{\sigma, p}(\Sym^r, 0)
\]
and
\[
\widetilde{\mathcal{M}}_{\sigma, p}(\Sym^r, 0)
=
\int_{\mathbb{R}}\mathcal{M}_{\sigma, p}(\Sym^r, u) \frac{du}{\sqrt{2\pi}}
=1
\]
by Proposition~\ref{M_P}.
\end{proof}

%
%
\section{The key lemma and the deduction of Theorem~\ref{main1}}\label{sec4}
\par
Let $r=1$ or $2$ (i.e. $\rho=0$). 
For a fixed $\sigma>1/2$ and a fixed finite set $\mathcal{P}_q$ satisfying $\mathcal{P}_q^0\subset\mathcal{P}_q\subset \mathbb{P}_q$, from \eqref{g_sigma} we see that
\[
\psi_x\circ \mathscr{G}_{\sigma, \mathcal{P}_q}(\alpha_f^r(\mathcal{P}_q))
=
\psi_x(\log L_{\mathcal{P}_q}(\Sym_f^r, \sigma))
\qquad
(r=1,2).
\]
Therefore, to prove our Theorem~\ref{main1}, it is important to consider the average 
\begin{align*}
\Avg(\psi_x\circ \mathscr{G}_{\sigma, \mathcal{P}_q})
=&
\lim_{q^m\to\infty}
\sideset{}{'}\sum_{f \in \mathscr{P}_k(q^m)}
\psi_x\circ \mathscr{G}_{\sigma, \mathcal{P}_q}(\alpha_f^r(\mathcal{P}_q)).
\end{align*}
Here, we remark that $L(\Sym_f^r, s)$ is entire due to $r=1,2$. 
%
%
\begin{rem}\label{rem-P-q-fixed}
We may understand that in Case II, $\mathcal{P}_q^0$ and $\mathcal{P}_q$ are 
actually independent of $q$.
In fact, since we consider the situation $q\to\infty$, we may assume that $q$ is sufficiently large, so is larger than any element of $\mathcal{P}_q^0$ and $\mathcal{P}_q$.
\end{rem}
%
%
\par
Our first aim in this section is to show the following
\begin{lem}\label{key}
Let $\mathcal{P}_q$ be a fixed finite subset of $\mathbb{P}_q$. 
Let $2\leq k < 12$ or $k=14$, and $r=1,2$ (i.e. $\rho=0$).
We have
\begin{align}\label{lem-3-formula}
\Avg(\psi_x \circ \mathscr{G}_{\sigma, \mathcal{P}_q})
=&
\int_{\Theta_{\mathcal{P}_q}}
\psi_x(
\mathscr{G}_{\sigma, \mathcal{P}_q} (e^{i\theta_{\mathcal{P}_q}r})
)
d^{\rm{ST}} \theta_{\mathcal{P}_q}
\notag\\
=&
\int_{\mathcal{T}_{\mathcal{P}_q}}
\psi_x
\left(
\mathscr{G}_{\sigma, \mathcal{P}_q}(t_{\mathcal{P}_q}^r)
\right)
\prod_{p\in\mathcal{P}_q}
\bigg(\frac{t_p^2 -2 + t_p^{-2}}{-2}\bigg)
d^H t_{\mathcal{P}_q}.
\end{align}
The convergence on the left-hand side (as $q^m\to\infty$) is uniform in 
$|x|\leq R$ for any $R>0$.
\end{lem}

In Case I, $q$ is fixed on \eqref{lem-3-formula}, and the limit on the
left-hand side is $m\to\infty$.
In Case II, in view of Remark \ref{rem-P-q-fixed}, the right-hand side 
of \eqref{lem-3-formula} is independent of $q$, and the limit on the
left-hand side is $q\to\infty$.

%
%
Lemma \ref{key} is an analogue of \cite[Lemma 4.1]{mu}, and the structure of the proof
is similar.
\begin{proof}
Let $1>\varepsilon'>0$ and $p\in\mathcal{P}_q$.
Considering the Taylor expansion of $g_{\sigma, p}(t)=-\log (1-t p^{-\sigma})$, we find that there exist an $M_p=M_p(\varepsilon',R)\in\mathbb{N}$ and $c_{m_p}=c_{m_p}(x,p,\sigma)\in\mathbb{C}$ such that $\psi_x\circ g_{\sigma, p}$ can be approximated by a polynomial as
\begin{equation}\label{app-uni}
\bigg|
\psi_x \circ g_{\sigma, p}(t)
- 
\sum_{m_p=0}^{M_p} c_{m_p} t^{m_p}
\bigg|
<{\varepsilon'},
\end{equation}
uniformly on $T$ with respect to $t$ and also on $|x|\leq R$ with respect to $x$.
Here $c_0=1$.
Write
\[
\phi_{\sigma, p}(t_p ; M_p)=
\sum_{m_p=0}^{M_p} c_{m_p} t_p^{m_p}
\]
and define
\[
\Phi_{\sigma, \mathcal{P}_q}(t_{\mathcal{P}_q};  M_{\mathcal{P}_q})
=
\prod_{p \in \mathcal{P}_q}
\phi_{\sigma, p}(t_p ; M_p)\phi_{\sigma, p}(t_p^{-1} ; M_p)
\phi_{\sigma, p}(\delta_{r,\text{even}} ; M_p),
\]
where $M_{\mathcal{P}_q}=(M_p)_{p\in\mathcal{P}_q}$.
Let $\varepsilon''>0$.
Choosing $\varepsilon'$ (depending on $|\mathcal{P}_q|$ and $\varepsilon''$) sufficiently small, we obtain $M_{\mathcal{P}_q}$ such that
\begin{equation}\label{app}
|
\psi_x \circ \mathscr{G}_{\sigma, \mathcal{P}_q}(t_{\mathcal{P}_q})
-
\Phi_{\sigma, \mathcal{P}_q}(t_{\mathcal{P}_q} ; M_{\mathcal{P}_q})| 
< \varepsilon'',
\end{equation}
again uniformly on $T_{\mathcal{P}_q}$ with respect to $t_{\mathcal{P}_q}$ and also on $|x|\leq R$ with respect to $x$.
In fact, since
\begin{align*}
&
\psi_x \circ \mathscr{G}_{\sigma, \mathcal{P}_q}(t_{\mathcal{P}_q})
\\
&=
\prod_{p\in\mathcal{P}_q}
\psi_x(g_{\sigma, p}(t_p))\psi_x(g_{\sigma, p}(t_p^{-1}))
\psi_x(g_{\sigma, p}(\delta_{r,\text{even}}))
\\
&=
\prod_{p\in\mathcal{P}_q}
(\phi_{\sigma, p}(t_p ; M_p) +O(\varepsilon'))
(\phi_{\sigma, p}(t_p^{-1} ; M_p)+O(\varepsilon'))
(\phi_{\sigma, p}(\delta_{r,\text{even}} ; M_p)+O(\varepsilon'))
\\
&=\Phi_{\sigma, \mathcal{P}_q}(t_{\mathcal{P}_q} ; M_{\mathcal{P}_q})
+(\text{remainder terms}),
\end{align*}
we obtain \eqref{app}.
Here, in the final stage, we use the fact $\phi_{\sigma, p}=O(1)$ which follows from \eqref{app-uni}.
%
%
\par
First we express the average of the value of $\psi_x\circ \mathscr{G}_{\sigma, \mathcal{P}_q}$ by using $\Phi_{\sigma, \mathcal{P}_q}$.
Let $\varepsilon>0$.
From \eqref{app} with $t_{\mathcal{P}_q}=\alpha_f(\mathcal{P}_q)$, $\varepsilon''=\varepsilon/2$ and \eqref{P1}, we have 
\begin{align*}
&
\left|
\sideset{}{'}\sum_{f \in \mathscr{P}_k(q^m)}
\psi_x \circ \mathscr{G}_{\sigma, \mathcal{P}_q}(\alpha_f^r(\mathcal{P}_q))
-
\sideset{}{'}\sum_{f \in \mathscr{P}_k(q^m)}
\Phi_{\sigma, \mathcal{P}_q}(\alpha_f^r(\mathcal{P}_q) ; M_{\mathcal{P}_q})
\right|
\nonumber\\
< &
\sideset{}{'}\sum_{f \in \mathscr{P}_k(q^m)}
\varepsilon''
=\frac{\varepsilon}{2}(1 +O(E(q^m))).
\end{align*}
Therefore, if $q^m$ is sufficiently large, from \eqref{E2} we see that
\begin{align}\label{step-1}
\left|
\sideset{}{'}\sum_{f \in \mathscr{P}_k(q^m)}
\psi_x \circ \mathscr{G}_{\sigma, \mathcal{P}_q}(\alpha_f^r(\mathcal{P}_q))
-
\sideset{}{'}\sum_{f \in \mathscr{P}_k(q^m)}
\Phi_{\sigma, \mathcal{P}_q}(\alpha_f^r(\mathcal{P}_q); M_{\mathcal{P}_q})
\right|<\varepsilon.
\end{align}
%
%
\par
Next we calculate $\Phi_{\sigma, \mathcal{P}_q}$ as follows;
\begin{align*}
&
\Phi_{\sigma, \mathcal{P}_q}(\alpha_f^r(\mathcal{P}_q) ;  M_{\mathcal{P}_q})
\\
=&
\prod_{ p\in \mathcal{P}_q}
\Big(\sum_{m_p=0}^{M_p}c_{m_p}e^{ im_pr\theta_f(p)}\Big)
\Big(\sum_{n_p=0}^{M_p}c_{n_p}e^{- in_pr\theta_f(p)}\Big)
\Big(\sum_{\ell_p=0}^{M_p}c_{\ell_p}\delta_{r,\text{even}}^{\ell_p}\Big)
\\
=&
\prod_{p\in \mathcal{P}}
\Big(\sum_{\ell_p=0}^{M_p}c_{\ell_p}\delta_{r,\text{even}}^{\ell_p}\Big)
\Big(
\sum_{m_p=0}^{M_p}c_{m_p}^2
+
\sum_{m_p=0}^{M_p}
\sum_{\substack{n_p=0\\m_p \neq n_p}}^{M_p}
c_{m_p}e^{ im_p r\theta_f(p)}
c_{n_p}e^{- in_p r\theta_f(p)}
\Big)
\\
=&
\prod_{p \in \mathcal{P}}
\Big(\sum_{\ell_p=0}^{M_p}c_{\ell_p}\delta_{r,\text{even}}^{\ell_p}\Big)
\Big(
\sum_{m_p=0}^{M_p}c_{m_p}^2
+
\sum_{m_p=0}^{M_p}
\sum_{\substack{n_p=0 \\ m_p<n_p}}^{M_p}
c_{m_p}c_{n_p}
e^{ i(m_p-n_p) r\theta_f(p)}
\\
&+
\sum_{m_p=0}^{M_p}
\sum_{\substack{n_p=0 \\ m_p<n_p}}^{M_p}
c_{m_p}c_{n_p}
e^{ i(n_p-m_p) r\theta_f(p)}
\Big).
\end{align*}
Putting $n_p-m_p=\nu_p$, the right-hand side is equal to
\begin{align*}
&
\prod_{ p \in \mathcal{P}_q}
\Big(\sum_{\ell_p=0}^{M_p}c_{\ell_p}\delta_{r,\text{even}}^{\ell_p}\Big)
\\
&\times
\Big(
\sum_{m_p=0}^{M_p}c_{m_p}^2
+
\sum_{\nu_p=1}^{M_p}
\sum_{n_p=\nu_p}^{M_p}
c_{n_p-\nu_p}c_{n_p}
\big(
e^{ i\nu_p r\theta_f(p)}
+
e^{- i\nu_p r\theta_f(p)}
\big)
\Big)
\\
=&
\prod_{ p \in \mathcal{P}_q}
\Big(\sum_{\ell_p=0}^{M_p}c_{\ell_p}\delta_{r,\text{even}}^{\ell_p}\Big)
\\
&\times
\bigg(
\sum_{m_p=0}^{M_p}c_{m_p}^2
+
\sum_{\nu_p=3}^{M_p}
\sum_{n_p=\nu_p}^{M_p}
c_{n_p-\nu_p}c_{n_p}
\big(
e^{ i\nu_p r\theta_f(p)}
+
e^{- i\nu_p r\theta_f(p)}
\big)
\\
&+
\sum_{n_p=2}^{M_p}
c_{n_p-2}c_{n_p}
\big(
e^{ 2i r\theta_f(p)}
+
e^{- 2i r\theta_f(p)}
\big)
\\
&+
\sum_{n_p=1}^{M_p}
c_{n_p-1}c_{n_p}
\big(
e^{ ir\theta_f(p)}
+
e^{- ir\theta_f(p)}
\big)
\bigg).
\end{align*}
Therefore, in the case $r=1$, using \eqref{euler} we see that
\begin{align}\label{odd}
&
\Phi_{\sigma, \mathcal{P}_q}(\alpha_f(\mathcal{P}_q); M_{\mathcal{P}_q})
\nonumber\\
=&
\prod_{p\in \mathcal{P}_q}
\bigg(
\sum_{m_p=0}^{M_p}c_{m_p}^2
+
\sum_{r_p=3}^{M_p}
\sum_{n_p=\nu_p}^{M_p}
c_{n_p-\nu_p}c_{n_p}
\big(
\lambda_f(p^{\nu_p})
-
\lambda_f(p^{\nu_p-2})
\big)
\nonumber\\
&
+
\sum_{n_p=2}^{M_p}
c_{n_p-2}c_{n_p}
\big(
\lambda_f(p^2)-1
\big)
+
\sum_{n_p=1}^{M_p}
c_{n_p-1}c_{n_p}
\lambda_f(p)
\bigg).
\end{align}
In the case $r=2$, similarly we see that
\begin{align}\label{even}
&
\Phi_{\sigma, \mathcal{P}_q}(\alpha_f^2(\mathcal{P}_q)); M_{\mathcal{P}_q})
\nonumber\\
=&
\prod_{ p \in \mathcal{P}_q}
\Big(\sum_{\ell_p=0}^{M_p}c_{\ell_p}\Big)
\bigg(
\sum_{m_p=0}^{M_p}c_{m_p}^2
+
\sum_{\nu_p=3}^{M_p}
\sum_{n_p=\nu_p}^{M_p}
c_{n_p-r_p}c_{n_p}
\big(
\lambda_f(p^{2\nu_p})
-
\lambda_f(p^{2\nu_p-2})
\big)
\nonumber\\
&+
\sum_{n_p=2}^{M_p}
c_{n_p-2}c_{n_p}
\big(
\lambda_f(p^4)-\lambda_f(p^2)
\big)
+
\sum_{n_p=1}^{M_p}
c_{n_p-1}c_{n_p}
\big(
\lambda_f(p^2)-1
\big)
\bigg)
\end{align}
From \eqref{odd} and \eqref{even}, by using \eqref{P} and the multiplicity of $\lambda_f$, we obtain
\begin{align*}
&
\sideset{}{'}\sum_{f \in \mathscr{P}_k(q^m)}
\Phi_{\sigma, \mathcal{P}_q}(\alpha_f^r(\mathcal{P}_q) ; M_{\mathcal{P}_q})
\\
=&
\prod_{ p\in \mathcal{P}_q}
\Big(\sum_{\ell_p=0}^{M_p}c_{\ell_p}\delta_{r,\text{even}}^{\ell_p}\Big)
\bigg(
\sum_{m_p=0}^{M_p}c_{m_p}^2
-
\sum_{m_p=0}^{M_p-\diamond}
c_{m_p}c_{m_p+\diamond}
\bigg)\\
&\qquad+ O\Big(\prod_{ p\in \mathcal{P}_q}
\Big(\sum_{\ell_p=0}^{M_p}c_{\ell_p}\delta_{r,\text{even}}^{\ell_p}\Big)
E(q^m)
\Big),
\end{align*}
where $\diamond=2$ if $r=1$ and $\diamond=1$ if $r=2$.
In this estimate, the implied constant of the error term depends on $\mathcal{P}_q$, $x$, $\sigma$ and $ M_{\mathcal{P}_q}= M_{\mathcal{P}_q}(\varepsilon',R)$ (hence depends on $\varepsilon$ under the above choice of $\varepsilon'$).
But still, this error term is smaller than $\varepsilon$ for sufficiently large $q^m$.
Combining this with \eqref{step-1}, we obtain
\begin{align}\label{step-2}
&
\Bigg|
\sideset{}{'}\sum_{f \in \mathscr{P}_k(q^m)}
\psi_x \circ \mathscr{G}_{\sigma, \mathcal{P}_q}(\alpha_f^r(\mathcal{P}_q))
\nonumber\\
&\hspace{2cm}
-
\prod_{p \in \mathcal{P}_q}
\Big(\sum_{\ell_p=0}^{M_p}c_{\ell_p}\delta_{r,\text{even}}^{\ell_p}\Big)
\bigg(
\sum_{m_p}^{M_p}c_{m_p}^2
-
\sum_{m_p=0}^{M_p-\diamond}
c_{m_p}c_{m_p+\diamond}
\bigg)
\Bigg|
<
2\varepsilon.
\end{align}
%
%
\par
Lastly we calculate the integral in the statement of Lemma~\ref{key}.
Putting $\Psi=\psi_x$ in \eqref{prop-1-1} of Proposition~\ref{M_P}, we already know that
\begin{align*}
& 
\int_{\Theta_{\mathcal{P}_q}}
\psi_x\big(\mathscr{G}_{\sigma, \mathcal{P}_q}(e^{i \theta_{\mathcal{P}_q}r})\big)
d^{\rm ST}\theta_{\mathcal{P}_q}
\\
=&
\int_{\mathcal{T}_{\mathcal{P}_q}}
\psi_x(\mathscr{G}_{\sigma, \mathcal{P}_q}( t_{\mathcal{P}_q}^r))
\prod_{p\in\mathcal{P}_q}
\bigg(\frac{t_p^2-2+t_p^{-2}}{-2}\bigg)
d^H t_{\mathcal{P}_q},
\end{align*}
so our remaining task is to prove that these integrals are equal to $\Avg(\psi_x\circ \mathscr{G}_{\sigma,\mathcal{P}_q})$.
For any $\varepsilon>0$, using \eqref{app}, we have
\begin{align*}
&
\int_{\mathcal{T}_{\mathcal{P}_q}}
\psi_x(\mathscr{G}_{\sigma, \mathcal{P}_q}(t_{\mathcal{P}_q}^r))
\prod_{p\in\mathcal{P}_q}
\bigg(\frac{t_p^2-2+t_p^{-2}}{-2}\bigg)
d^H t_{\mathcal{P}_q}
\\
=&
\int_{\mathcal{T}_{\mathcal{P}_q}}
\bigg(
\psi_x(\mathscr{G}_{\sigma, \mathcal{P}_q}(t_{\sigma, \mathcal{P}_q}^r))
-
\Phi_{\sigma, \mathcal{P}_q}(t_{\mathcal{P}_q}^r ; M_{\mathcal{P}_q})
\bigg)
\prod_{p\in\mathcal{P}_q}
\bigg(\frac{t_p^2-2+t_p^{-2}}{-2}\bigg)
d^H t_{\mathcal{P}_q}
\\
&+
\int_{\mathcal{T}_{\mathcal{P}_q}}
\Phi_{\sigma, \mathcal{P}_q}(t_{\mathcal{P}_q}^r ; M_{\mathcal{P}_q })
\prod_{p\in\mathcal{P}_q}
\bigg(\frac{t_p^2-2+t_p^{-2}}{-2}\bigg)
d^H t_{\mathcal{P}_q}
\\
=&
\prod_{p\in\mathcal{P}_q}
\int_{\mathcal{T}}
\phi_{\sigma, p}(t_p^r; M_p)\phi_{\sigma, p}(t_p^{-r}; M_p)
\phi_{\sigma, p}(\delta_{r,\text{even}}; M_p)
\bigg(\frac{t_p^2-2+t_p^{-2}}{-2}\bigg)
\frac{dt_p}{2\pi it_p}
\\
&
+
O(\varepsilon).
\end{align*}
Since
\begin{align*}
&
\int_{\mathcal{T}}
\phi_{\sigma, p}(t_p^r; M_p)\phi_{\sigma, p}(t_p^{-r}; M_p)
\phi_{\sigma, p}(\delta_{r,\text{even}}; M_p)
\bigg(\frac{t_p^2-2+t_p^{-2}}{-2}\bigg)
\frac{dt_p}{2\pi it_p}
\\
=&
\int_{\mathcal{T}}
\bigg(
\sum_{m=0}^{M_p}c_mt_p^{mr}
\bigg)
\bigg(
\sum_{n=0}^{M_p}c_nt_p^{-nr}
\bigg)
\bigg(
\sum_{\ell=0}^{M_p}c_n\delta_{r,\text{even}}^{\ell}
\bigg)
\bigg(\frac{t_p^2-2+t_p^{-2}}{-2}\bigg)
\frac{dt_p}{2\pi it_p}
\\
=&
\bigg(
\sum_{\ell=0}^{M_p}c_n\delta_{r,\text{even}}^{\ell}
\bigg)
\int_0^{2\pi}
\bigg(
\sum_{m=0}^{M_p}c_me^{imr\theta}
\bigg)
\bigg(
\sum_{n=0}^{M_p}c_ne^{-inr\theta}
\bigg)
\bigg(\frac{e^{2i\theta}-2+e^{-2i\theta}}{-2}\bigg)
\frac{d\theta}{2\pi}
\\
=&
\bigg(
\sum_{\ell=0}^{M_p}c_n\delta_{r,\text{even}}^{\ell}
\bigg)
\int_0^{2\pi}
\bigg(
\sum_{m=0}^{M_p}c_m^2
+
\sum_{m=0}^{M_p}
\sum_{\substack{n=0\\ m\neq n}}^{M_p}
c_mc_ne^{i(m -n)r\theta}
\bigg)
\\
&\times
\bigg(\frac{e^{2i\theta}-2+e^{-2i\theta}}{-2}\bigg)
\frac{d\theta}{2\pi}
\\
=&
\bigg(
\sum_{\ell=0}^{M_p}c_n\delta_{r,\text{even}}^{\ell}
\bigg)
\begin{cases}
\displaystyle
\sum_{m=0}^{M_p}c_m^2 -\sum_{m=0}^{M_p-2} c_mc_{m+2}
& r=1, \\
\displaystyle
\sum_{m=0}^{M_p}c_m^2 -\sum_{m=0}^{M_p-1} c_mc_{m+1}
& r=2,
\end{cases}
\end{align*}
we have
\begin{align}\label{step-3}
&
\int_{\mathcal{T}_{\mathcal{P}_q}}
\psi_x(\mathscr{G}_{\sigma, \mathcal{P}_q}(t_{\mathcal{P}_q}^r)
\prod_{p\in\mathcal{P}_q}
\bigg(\frac{t_p^2-2+t_p^{-2}}{-2}\bigg)
d^H t_{\mathcal{P}_q}
\nonumber\\
=&
\prod_{p\in\mathcal{P}_q}
\bigg(
\sum_{\ell=0}^{M_p}c_n\delta_{r,\text{even}}^{\ell}
\bigg)
\bigg(
\sum_{m_p=0}^{M_p}c_{m_p}^2
-
\sum_{m_p=0}^{M_p-\diamond}
c_{m_p}c_{m_p+\diamond}
\bigg)
+
O(\varepsilon).
\end{align}
From \eqref{step-2} and \eqref{step-3} we find that the identity in the statement 
of Lemma~\ref{key} holds with the error $O(\varepsilon)$, but this error can be arbitrarily small if we choose sufficiently large $q^m$, so the assertion of Lemma~\ref{key} follows.
\end{proof}
%
%
\par
The next lemma is an analogue of \cite[Lemma 4.3]{mu}.
\begin{lem}\label{keylemma}
Suppose Assumption (GRH), in the case  $2\leq k <12$ or $k=14$, for $\sigma>1/2$ and $r=1,2$, we have
\[
\Avg\psi_x (\log L_{\mathbb{P}_q}(\Sym_f^r, \sigma))
=
\int_{\mathbb{R}}
\mathcal{M}_{\sigma}(\Sym^r, u)
\psi_x(u)\frac{du}{\sqrt{2\pi}},
\]
where $\psi_x(u)=\exp (ixu)$. 
The above convergence is uniform in $|x|\leq R$ for any $R>0$.
\end{lem}
%
%
We will prove this lemma in Section \ref{sec5}.
We note that this lemma is actually the special case $\Psi=\psi_x$ of our main
Theorem~\ref{main1}.
Showing this lemma is the core of the proof of Theorem~\ref{main1}.
\begin{proof}[Proof of Theorem~\ref{main1} and Remark~\ref{abs-conv-case}]
From Lemma~\ref{keylemma}, we can obtain Theorem~\ref{main1} by the same argument as in \cite[Section~7]{mu} (or in \cite[Section~9]{im_2011}), with corrections presented in \cite{matsumoto}.
We omit the details here, and just outline the argument very briefly.

First, let
\[
\Lambda=\{\Psi\in L^1\cap L^{\infty}\mid \Psi^{\vee}\in L^1\cap L^{\infty},
(\Psi^{\vee})^{\wedge}=\Psi\},
\]
where $\Psi^{\vee}$ means the Fourier transform of $\Psi$ and $\Psi^{\wedge}$
means the Fourier inverse transform of $\Psi$.
Using properties of $\widetilde{\mathcal{M}}_{\sigma}$
(Proposition \ref{tildeM}, \eqref{tildeTrivial}, Proposition \ref{M}, and
Lemma \ref{keylemma}), we can show the assertion of Theorem \ref{main1}
for $\Psi\in\Lambda$.    For bounded continuous $\Psi$, we approximate it
by the elements of $\Lambda$ to obtain the result.
As for the case of compactly supported Riemann integrable $\Psi$, see
\cite{matsumoto}.
\par
The assertion in Remark~\ref{abs-conv-case} can be deduced easily from Remark~\ref{M-p-supp} as follows. 
First note that, by \eqref{g_sigma}, the support of $\log L_{\mathcal{P}_q}(\Sym_f^r,\sigma)$ is included in $C_{\sigma}(\mathcal{P}_q)$.
If $\sigma>1$, $C_{\sigma}(\mathcal{P}_q)$ remains bounded when $|\mathcal{P}_q|\to\infty$.
Let $B_{\sigma}$ be the closure of the union of all $C_{\sigma}(\mathcal{P}_q)$, where $\mathcal{P}_q$ runs over all possible finite subset of $\mathbb{P}_q$.
This is surely a compact set.
By Remark~\ref{abs-conv-case} we see that $\mathrm{Supp}(\mathcal{M}_{\sigma,\mathcal{P}_q})\subset B_{\sigma}$ for any $\mathcal{P}_q$, hence $\mathrm{Supp}(\mathcal{M}_{\sigma})\subset B_{\sigma}$.
Therefore $\mathrm{Supp}(\mathcal{M}_{\sigma})$ is compact.
\par
Let $\Psi:\mathbb{R}\to\mathbb{C}$ be a continuous function.
Define a bounded continuous $\Psi_0$ satisfying $\Psi_0(x)=\Psi(x)$ for $x\in B_{\sigma}$ and $\Psi_0(x)=0$ if $|x|$ is sufficiently large.
Then by Theorem~\ref{main1}, formula \eqref{main-formula} is valid for $\Psi=\Psi_0$.
However, since $\Psi_0=\Psi$ on $B_{\sigma}$, we can replace $\Psi_0$ by $\Psi$ on the both sides of \eqref{main-formula}.
Therefore the condition ``bounded continuous'' in Theorem~\ref{main1} can be relaxed to ``continuous'' for $\sigma>1$.
\end{proof}

%
%
\section{Bounds for symmetric power $L$-functions}\label{sec:sym}
Let $f$ be a primitive form in $S_k(N)$ with $k \geq2$. 
For any prime $p \nmid N$, there exist complex numbers $\alpha_f(p)$ and $\beta_f(p)$ satisfying 
\begin{gather*}
\alpha_f(p)+\beta_f(p)=\lambda_f(p)
\quad\text{and}\quad
\alpha_f(p) \beta_f(p)=1. 
\end{gather*}
By the work of Deligne, we know that $|\alpha_f(p)|=|\beta_f(p)|=1$, and thus they are the complex conjugate of each other.
For $r \geq1$, we define
\begin{gather*}
L_p(\Sym_f^r, s)
= \prod_{h=0}^{r} (1-\alpha_f^{r-h}(p) \beta_f^h(p) p^{-s})^{-1}. 
\end{gather*}
Recently, Newton and Thorne \cite{NT2021} established the Langlands functoriality of $\Sym_f^r$ for all $r \geq1$, that is, there exists an automorphic representation $\pi_r$ of $\mathrm{GL}_{r+1}(\mathbb{A}_\mathbb{Q})$ whose standard $L$-function
\begin{gather*}
L(\pi_r, s)
= \prod_{p \leq \infty} L_p(\pi_r, s)
\end{gather*}
satisfies $L_p(\pi_r, s)=L_p(\Sym_f^r, s)$ for all primes $p \nmid N$. 
In this section, we will collect several preliminary results on $L(\pi_r, s)$ to prove Lemma \ref{keylemma}. 

First, it should be remarked that $\pi_r$ is not necessarily cuspidal, and $L(\pi_r, s)$ has sometimes poles at $s=0,1$. 
If $f$ is not of CM-type, then $\pi_r$ is a cuspidal representation. 
Therefore, $L(\pi_r, s)$ has no poles in that case. 
On the other hand, if $f$ has CM, then $L(\pi_r, s)$ is decomposed as the product of $L$-functions of degree 1 or 2. 
More precisely, we have the following formula. 

\begin{lem}\label{lem:CM}
Let $f$ be a primitive form in $S_k(N)$ with $k \geq2$. 
Assume that $f$ has CM by the quadratic extension $F/\mathbb{Q}$, and denote by $\chi_F$ the quadratic Dirichlet character attached to $F$. 
Then, there exists a Hecke character $\phi$ on $F$ such that 
\begin{gather}\label{eq:decomp}
L(\pi_r, s)
= \prod_{h=0}^{[r/2]} L(\phi^{r-2h}, s), 
\end{gather}
where $L(\phi^j, s)$ denotes the Hecke $L$-function of the character $\phi^j$ for $j \geq1$, and 
\begin{gather*}
L(\phi^0, s)
= 
\begin{cases}
\zeta(s)
& \text{if $r \equiv 0 \pmod4$}, 
\\
L(s, \chi_F)
& \text{if $r \equiv 2 \pmod4$}. 
\end{cases}
\end{gather*}
Furthermore, the character $\phi^j$ is nontrivial for $j \geq1$. 
Hence $L(\pi_r, s)$ has poles if and only if $f$ has CM and $r \equiv 0 \pmod4$, and such possible poles are simple and located only at $s=0,1$. 
\end{lem}

\begin{proof}
This result seems to be well-known to experts of the theory of automorphic forms. 
For example, formula \eqref{eq:decomp} can be seen in \cite[Lemma 4.2]{RS2008} but the authors did not give the proof. 
Here, we sketch a proof of \eqref{eq:decomp} due to lack of a good reference. 
The proof is based on calculations of local factors of the $L$-functions. 
The idea is essentially the same as that given in the arXiv version of \cite{RS2008_2}, which was described in terms of representation theory. 

For a place $v$ of $F$, let $\varpi_v$ be a prime element of the local field $F_v$, and denote by $\phi_v$ the $v$-component of $\phi$. 
By the assumption, $L(\pi_1, s)$ is equal to the $L$-function $L(\phi,s)$ with some Hecke character $\phi$ on $F$. 
Hence, for any prime number $p \nmid N$, we obtain
\begin{align*}
L_p(\pi_1, s)
&=
\begin{cases}
(1-\phi_{v_1}(\varpi_{v_1}) p^{-s})^{-1} (1-\phi_{v_2}(\varpi_{v_2}) p^{-s})^{-1}
& \text{if $p$ splits in $F$}, 
\\
(1+p^{-2s})^{-1}
& \text{if $p$ is inert in $F$,} 
\end{cases}
\end{align*}
where $v_1$ and $v_2$ are the primes of $F$ lying above $p$ in the former case. 
Note that $\phi_{v_1}(\varpi_{v_1})$ and $\phi_{v_2}(\varpi_{v_2})$ are complex conjugates.
Then we calculate $L_p(\pi_r, s)$ for any $r \geq1$ as follows. 
If $p$ splits in $F$, we have 
\begin{align*}
&L_p(\pi_r, s)
= \prod_{h=0}^{r} (1-\phi_{v_1}^{r-h}(\varpi_{v_1}) \phi_{v_2}^{h}(\varpi_{v_2}) p^{-s})^{-1} \\
&= 
\begin{cases}
\displaystyle{
(1-p^{-s})^{-1}
\prod_{h=0}^{[r/2]-1} (1-\phi_{v_1}^{r-2h}(\varpi_{v_1}) p^{-s})^{-1} (1-\phi_{v_2}^{r-2h}(\varpi_{v_2}) p^{-s})^{-1}
}
\\ \hspace{80.5mm} \text{if $r$ is even}, 
\\
\displaystyle{
\prod_{h=0}^{[r/2]} (1-\phi_{v_1}^{r-2h}(\varpi_{v_1}) p^{-s})^{-1} (1-\phi_{v_2}^{r-2h}(\varpi_{v_2}) p^{-s})^{-1}
}
\quad \text{if $r$ is odd}
\end{cases}
\end{align*}
for any $r \geq1$ by straight computations. 
If $p$ is inert, we also obtain
\begin{align*}
&L_p(\pi_r, s)
= \prod_{h=0}^{r} (1-i^{r-h} (-i)^h p^{-s})^{-1} \\
&= 
\begin{cases}
\displaystyle{
(1-p^{-s})^{-1}
\prod_{h=0}^{[r/2]-1} (1-(-1)^{r-2h} p^{-2s})^{-1} 
}
&\text{if $r \equiv 0 \pmod{4}$}, 
\\
\displaystyle{
(1+p^{-s})^{-1}
\prod_{h=0}^{[r/2]-1} (1-(-1)^{r-2h} p^{-2s})^{-1} 
}
&\text{if $r \equiv 2 \pmod{4}$}, 
\\
\displaystyle{
\prod_{h=0}^{[r/2]} (1-(-1)^{r-2h} p^{-2s})^{-1} 
}
&\text{if $r$ is odd}. 
\end{cases}
\end{align*}
Denote by $\pi'_r$ the isobaric sum $\mathrm{AI}(\phi^r) \boxplus \mathrm{AI}(\phi^{r-2}) \boxplus \cdots \boxplus \mathrm{AI}(\phi^{r-2[r/2]})$, where $\mathrm{AI}(\phi^j)$ is the automorphic representation over $\mathbb{Q}$ induced from $\phi^j$. 
Then the above calculations yield $L_p(\pi_r,s)=L_p(\pi'_r,s)$ for any prime $p \nmid N$. 
(Recall that $\chi_F(p)=1$ if $p$ splits, and $\chi_F(p)=-1$ if $p$ is inert.) 
By the strong multiplicity one theorem, we conclude that \eqref{eq:decomp} holds. 

To show that $\phi^j$ is nontrivial for $j \geq1$, we recall that $L(\phi,s)$ is equal to the $L$-function of $f \in S_k(N)$. 
In particular, the infinite component $L(\phi_\infty,s)$ is equal to $2(2\pi)^{-s-(k-1)/2} \Gamma(s+(k-1)/2)$. 
It implies that $\phi_\infty$ is given by $\phi_\infty(z)=(z/|z|)^{k-1}$ for $z \in \mathbb{C}^{\times}$. 
Therefore, the character $(\phi_\infty)^j$ is nontrivial as $k \geq2$. 
From the above, the proof of the lemma is completed.  
\end{proof}

Since $\pi_r$ is an automorphic representation of $\mathrm{GL}_{r+1}(\mathbb{A}_\mathbb{Q})$ which is self contragredient, the standard $L$-function $L(\pi_r, s)$ satisfies the functional equation
\begin{gather}\label{eq:FE}
L(\pi_r, s)
= W(\pi_r) N(\pi_r)^{1/2-s} L(\pi_r, 1-s), 
\end{gather}
where $W(\pi_r)$ and $N(\pi_r)$ are the root number and conductor of $\pi_r$, respectively. 
Note that $W(\pi_r)$ satisfies $|W(\pi_r)|=1$. 
If $f \in S_k(N)$ with $N$ square-free, then we can confirm that $N(\pi_r)$ is equal to $N^r$. 
In the general case, such a formula may not hold, but we have the following upper bound which is sufficient for our purpose. 

\begin{lem}\label{lem:cond}
Let $f$ be a primitive form in $S_k(N)$. 
For any $r \geq1$, we denote by $\pi_r$ the automorphic representation of $\mathrm{GL}_{r+1}(\mathbb{A}_\mathbb{Q})$ attached to $\Sym_f^r$. 
Then 
\begin{gather}\label{cond}
N(\pi_r) 
\ll N^{c(r)}, 
\end{gather}
where $c(r)$ is a positive constant depending only on $r$. 
\end{lem}

\begin{proof}
If $f$ is not of CM-type, then $\pi_r$ is cuspidal for any $r \geq1$ as mentioned before. 
Therefore \eqref{cond} is due to \cite[Lemma 2.1]{rouse07}. 
In the CM case, we see that it is also satisfied by applying \eqref{eq:decomp} and evaluating the conductors of $\phi^j$. 
\end{proof}

Since the automorphic representation $\pi_1$ is obtained from a primitive form $f \in S_k(N)$ with $k \geq2$, the infinite component of $\pi_1$ is isomorphic to the discrete series representation of $\mathrm{GL}_2(\mathbb{R})$ of weight $k$. 
Then the infinite component of $\pi_r$ for $r \geq1$ is determined by Cogdell and Michel \cite[Corollary 3.2]{CM04}. 
Furthermore, the $L$-function $L_\infty(\pi_r, s)$ is given by
\begin{gather*}
L_\infty(\pi_{2m+1}, s)
= \prod_{a=0}^{m} \Gamma_\mathbb{C}\left(s+\frac{(2a+1)(k-1)}{2}\right), \\
L_\infty(\pi_{2m}, s)
= 
\begin{cases}
\displaystyle{
\Gamma_\mathbb{R}(s) 
\prod_{a=1}^{m} \Gamma_\mathbb{C}(s+a(k-1)) 
}
& \text{if $m(k-1)$ is even}, 
\\
\displaystyle{
\Gamma_\mathbb{R}(s+1) 
\prod_{a=1}^{m} \Gamma_\mathbb{C}(s+a(k-1)) 
}
& \text{if $m(k-1)$ is odd}, 
\end{cases}
\end{gather*}
where $\Gamma_\mathbb{R}(s)=\pi^{-s/2}\Gamma(s/2)$ and $\Gamma_\mathbb{C}(s)=2(2\pi)^{-s}\Gamma(s)$ (see \cite[Section 3.1.3]{CM04}). 
Then, we obtain an upper bound on the size of 
\begin{gather*}
L_{\mathrm{fin}}(\pi_r, s)
= \prod_{p<\infty} L_p(\pi_r, s)
\end{gather*}
by the Phragm\'{e}n-Lindel\"{o}f convexity principle. 

\begin{prop}\label{prop:convbd}
Let $f$ be a primitive form in $S_k(N)$. 
For any $r \geq1$, we denote by $\pi_r$ the automorphic representation of $\mathrm{GL}_{r+1}(\mathbb{A}_\mathbb{Q})$ attached to $\Sym_f^r$. 
Suppose that $L(\pi_r, s)$ is continued to $\mathbb{C}$ as an entire function. 
Then we have 
\begin{gather*}
L_{\mathrm{fin}}(\pi_r, s)
\ll_{k,r} \left( N^{c(r)} (|t|+1)^{r+1} \right)^{1/2}
\end{gather*}
for $s=\sigma+it$ with $1/2 \leq \sigma \leq4$, where $c(r)$ is the constant of Lemma \ref{lem:cond}. 
\end{prop}

\begin{proof}
Suppose that $f$ is not of CM-type. 
In this case $\pi_r$ is cuspidal, and the local factors of $L_{\mathrm{fin}}(\pi_r, s)$ are represented as 
\begin{gather}\label{eq:local}
L_p(\pi_r, s)
= \prod_{j=1}^{r+1} (1-\gamma_j(p) p^{-s})^{-1}
\end{gather}
for all prime number $p$, where $\gamma_j(p)$ are complex numbers satisfying
\begin{gather*}
|\gamma_j(p)|
\leq p^{1/2-\delta(r)}, 
\qquad
\delta(r)
= \frac{1}{(r+1)^2+1}. 
\end{gather*}
See \cite[Appendix]{RS}. 
Using these bounds, we derive
\begin{gather*}
|L_{\mathrm{fin}}(\pi_r, s)| 
\leq \prod_{p<\infty} \prod_{j=1}^{r+1} (1-p^{1/2-\delta(r)-\sigma})^{-1}
= \zeta(\sigma-1/2+\delta(r))^{r+1}
\end{gather*}
for $\sigma >3/2-\delta(r)$. 
In particular, we have $|L_{\mathrm{fin}}(\pi_r, s)| \ll_r 1$ for $3/2 \leq \sigma \leq 4$. 
On the other hand, let $s=\sigma+it$ with $\sigma \leq-1/2$. 
Then \eqref{eq:FE} gives
\begin{gather*}
|L_{\mathrm{fin}}(\pi_r, s)|
\leq N(\pi_r)^{1/2-\sigma} \left|\frac{L_\infty(\pi_r, 1-s)}{L_\infty(\pi_r, s)}\right| |L_{\mathrm{fin}}(\pi_r, 1-s)|. 
\end{gather*}
A simple calculation using Stirling's formula yields 
\begin{gather*}
\left|\frac{L_\infty(\pi_r, 1-s)}{L_\infty(\pi_r, s)}\right|
\ll_{k,r} (|t|+1)^{(r+1)(1/2-\sigma)}. 
\end{gather*}
Hence we find that $|L_{\mathrm{fin}}(\pi_r, s)| \ll_{k,r} \{N^{c(r)} (|t|+1)^{r+1}\}^{1/2-\sigma}$ for $\sigma \leq -1/2$ by applying Lemma \ref{lem:cond}. 
Finally, by the Phragm\'{e}n-Lindel\"{o}f convexity principle, we obtain the desired bound of $|L_{\mathrm{fin}}(\pi_r, s)|$ even for $1/2 \leq \sigma \leq 3/2$. 

In the case where $f$ is of CM-type, we deduce from the decomposition \eqref{eq:decomp} that the local factor $L_p(\pi_r, s)$ is represented as \eqref{eq:local} with $|\gamma_j(p)| \leq1$. 
Hence, by the same argument as in the case of non-CM type, we obtain the suitable upper bound of $|L_{\mathrm{fin}}(\pi_r, s)|$ for $1/2 \leq \sigma \leq4$, provided that $L(\pi_r, s)$ is entire. 
\end{proof}

\begin{cor}\label{cor:dlog}
Let $f$ be a primitive form in $S_k(N)$. 
For any $r \geq1$, we denote by $\pi_r$ the automorphic representation of $\mathrm{GL}_{r+1}(\mathbb{A}_\mathbb{Q})$ attached to $\Sym_f^r$. 
Suppose that $L(\pi_r, s)$ is continued to $\mathbb{C}$ as an entire function. 
If $L(\pi_r, s)$ has no zeros for $\mathrm{Re}(s)>1/2$, then we have 
\begin{gather*}
\frac{L'_{\mathrm{fin}}}{L_{\mathrm{fin}}}(\pi_r, s)
\ll_{k,r} \log \left( N(|t|+1) \right)
\end{gather*}
for $s=\sigma+it$ with $1/2 \leq \sigma \leq4$. 
\end{cor}

\begin{proof}
This corollary can be proved in the same line as \cite[Lemma	3]{d} by using Proposition \ref{prop:convbd}. 
\end{proof}

Let $f$ be a primitive form in $S_k(q^m)$, where $q$ is a prime number, and $2\leq k < 12$ or $k=14$. 
We consider the case $r=1$ or $2$. 
Now we recall that the partial $L$-function $L_{\mathbb{P}_q}(\Sym_f^r, s)$ is represented as
\begin{gather}\label{eq:sym_Pq}
L_{\mathbb{P}_q}(\Sym_f^r, s)
= \prod_{\substack{p<\infty \\ p \neq q}} L_p(\pi_r, s)
= L_{\mathrm{fin}}(\pi_r, s) \cdot L_q(\pi_r, s)^{-1}. 
\end{gather}
Thus we obtain
\begin{gather}\label{eq:ac}
\frac{L'_{\mathbb{P}_q}}{L_{\mathbb{P}_q}}(\Sym_f^r, s)
= \frac{L'_{\mathrm{fin}}}{L_{\mathrm{fin}}}(\pi_r, s)
- \frac{L'_q}{L_q}(\pi_r, s). 
\end{gather}
Since $L(\pi_r, s)$ is an entire function due to $r=1,2$, we can apply Corollary \ref{cor:dlog} under Assumption (GRH). 
Therefore the first term of the right-hand side of \eqref{eq:ac} is 
\begin{gather*}
\frac{L'_{\mathrm{fin}}}{L_{\mathrm{fin}}}(\pi_r, s)
\ll_{r} \log \left( q^m(|t|+1) \right)
\end{gather*}
for $s=\sigma+it$ with $1/2 \leq \sigma \leq 4$ if we suppose Assumption (GRH). 
Furthermore, by using \eqref{eq:local}, we calculate the second term as 
\begin{gather*}
\frac{L'_q}{L_q}(\pi_r, s)
= -(\log{q}) \sum_{m=1}^{\infty} \sum_{j=1}^{r+1} \gamma_j(q)^m q^{-ms}. 
\end{gather*}
The complex numbers $\gamma_j(q)$ satisfy $|\gamma_j(q)| \leq q^{1/2-\delta(r)}$ if $f$ is not of CM-type, and $|\gamma_j(q)| \leq 1$ otherwise. 
Therefore, we obtain 
\begin{gather}\label{eq:L_q}
\frac{L'_q}{L_q}(\pi_r, s)
\ll_r (\log{q}) \sum_{m=1}^{\infty} q^{-m(\sigma-1/2+\delta(r))}
\ll_r \log{q}
\end{gather}
for $\sigma \geq1/2$ in each case. 
As a result, the upper bound
\begin{gather}\label{eq:AC}
\frac{L'_{\mathbb{P}_q}}{L_{\mathbb{P}_q}}(\Sym_f^r, s)
\ll_{r} \log \left( q^m(|t|+1) \right)
\end{gather}
is available for $1/2 \leq \sigma \leq 4$ if we suppose Assumption (GRH). 

%
%
\section{Proof of Lemma~\ref{keylemma}}\label{sec5}
We begin with the case $\sigma>1$.
\begin{proof}[Proof of Lemma~\ref{keylemma} in the case $\sigma >1$]
Since $\sigma >1$, we can find a sufficiently large 
$y>0$ for which it holds that
\begin{align}\label{5-1-1}
|\log L_{\mathbb{P}_q}(\Sym_f^r, \sigma)
-
\log L_{\mathcal{P}_q(y)}(\Sym_f^r, \sigma)
|<\varepsilon
\end{align}
and 
\begin{align}\label{5-1-2}
|\widetilde{\mathcal{M}}_{\sigma, \mathcal{P}_q(y)}(\Sym^r, x)-\widetilde{\mathcal{M}}_{\sigma}(\Sym^r, x)|<\varepsilon
\end{align} 
for any $x\in\mathbb{R}$ and any $\varepsilon>0$.
The last inequality is provided by Proposition~\ref{tildeM}.
Using this $\mathcal{P}_q(y)$, we have
\begin{align*}
&
\bigg|\sideset{}{'}\sum_{f \in \mathscr{P}_k(q^m)}
\psi_x (\log L_{\mathbb{P}_q}(\Sym_f^r, \sigma))
-
\int_{\mathbb{R}}\mathcal{M}_{\sigma}(\Sym^r, u)
\psi_x(u) \frac{du}{\sqrt{2\pi}}
\bigg|
\\
\leq &
\bigg|\sideset{}{'}\sum_{f \in \mathscr{P}_k(q^m)}
\bigg(\psi_x (\log L_{\mathbb{P}_q}(\Sym_f^r, \sigma))
-
\psi_x (\log L_{\mathcal{P}_q(y)}(\Sym_f^r, \sigma))
\bigg)\bigg|
\\
&+
\bigg|\sideset{}{'}\sum_{f \in \mathscr{P}_k(q^m)}
\psi_x (\log L_{\mathcal{P}_q(y)}(\Sym_f^r, \sigma))
-
\int_{\mathbb{R}}\mathcal{M}_{\sigma, \mathcal{P}_q(y)}(\Sym^r, u)
\psi_x(u) \frac{du}{\sqrt{2\pi}}
\bigg|
\\
&
+
\bigg|
\int_{\mathbb{R}}\mathcal{M}_{\sigma, \mathcal{P}_q(y)}(\Sym^r, u)
\psi_x(u) \frac{du}{\sqrt{2\pi}}
-
\int_{\mathbb{R}}\mathcal{M}_{\sigma}(\Sym^r, u)
\psi_x(u) \frac{du}{\sqrt{2\pi}}
\bigg|
\\
=
&S_1+S_2+S_3,
\end{align*}
say.
We remind the relation 
\begin{equation}\label{ihara}
|\psi_x(u)-\psi_x(u')|
\ll |x|\cdot |u-u'|
\end{equation}
for $u\in \mathbb{R}$ (see Ihara~\cite[(6.5.19)]{ihara} or Ihara-Matsumoto~\cite{im_2011}).
We see that
\begin{align*}
S_1 \ll \; & |x| 
\sideset{}{'}\sum_{f \in \mathscr{P}_k(q^m)}
\bigg(
|\log L_{\mathbb{P}_q}(\Sym_f^r, \sigma)
-
\log L_{\mathcal{P}_q(y)}(\Sym_f^rf, \sigma)|
\bigg)
\end{align*}
and
\[
S_3
=
\left|
\widetilde{\mathcal{M}}_{\sigma,\mathcal{P}_q(y)}(\Sym^r, x)
-
\widetilde{\mathcal{M}}_{\sigma}(\Sym^r, x)
\right|.
\]
Therefore \eqref{5-1-1} and \eqref{5-1-2} yield that
$|S_1|$ and $|S_3|$ are $O(\varepsilon)$ for large $y$, with the implied constant depending on $R$.
As for the estimate on $|S_2|$, we use Proposition~\ref{M_P} and Lemma~\ref{key},
whose convergence is uniform on $|x|\leq R$.
This completes the proof.
\end{proof}
\par
Now we proceed to the proof in the more difficult case $1\geq\sigma>1/2$.
Let
\[
F(\Sym_f^r,s)=\frac{L_{\mathbb{P}_q}(\Sym_f^r, s)}
{L_{\mathcal{P}_q(\log q^m)}(\Sym_f^r, s)}.
\]
It is to be noted that $y=\log q^m$ satisfies \eqref{cases-I-II} for
sufficiently large $m$ in Case I, and for sufficiently large $q$ in
Case II.   Therefore we may apply the results in Section \ref{sec3}
in the following argument.

Using \eqref{eq:AC}
for $2\geq \sigma>1/2$, we have
\begin{align}\label{10-5-7-anal}
\frac{F^{\prime}}{F}(\Sym_f^r,s)
&=
\frac{L'_{\mathbb{P}_q}}{L_{\mathbb{P}_q}}
-
\frac{L'_{\mathbb{P}_q(\log q^m)}}{L_{\mathbb{P}_q(\log q^m)}}
\nonumber\\
&\ll
\log q^m+(\log(1+|t|))-
\sum_{\substack{p\neq q\\p\leq \log q^m}}\sum_{h=0}^r\frac{\log p}{p^{\sigma}-1}
\nonumber\\
&\ll_{r}
\log q^m+(\log(1+|t|))
\end{align}
analogously to \cite[(5.7)]{mu}.
Here the implied constant depens on $r$ and $k$, but we restrict $k$ to $0<k<12$ or $14$.
In \cite{mu}, we used the assumption $q>Q(\mu)$ (where $\mu$ is a positive integer
and $Q(\mu)$ is the smallest prime satisfying $2^{\mu}/\sqrt{Q(\mu)}<1$), but this type of assumption is not necessary here because of
\eqref{eq:L_q}, where we do not consider the $q$-factor in the case $r=2$ and $m\neq 1$ since $\lambda_f(q)=0$.
From \eqref{10-5-7-anal} we can deduce the following lemma, which is an analogue of \cite[Lemma~5.1 and Remark~5.2]{mu}, whose basic idea is due to Duke \cite{d} (Assumption (GRH) is necessary here).
%
%
\begin{lem}\label{appSymL}
  Suppose Assumption (GRH).
Let $f$ be a primitive form in $S_k(q^m)$, where $q$ is a prime number.
For fixed $r$ $(r= 1, 2)$ and $\sigma=1/2 + \delta$ $(0< \delta \leq 1/2)$, we have
\begin{align}\label{F}
&
\log L_{\mathbb{P}_q}(\Sym_f^r, \sigma)
- 
\log L_{\mathcal{P}_q(\log q^m)}(\Sym_f^r, \sigma)
-
\mathcal{S}_r
\nonumber\\
\ll_r&
\frac{1}{\delta (\log q^m)^{2\delta}}+(\log q^m)^{-1/2}
+(q^{m/4(k-1)r})^{-\delta/2},
\end{align}
where
\[
\mathcal{S}_r=
\sum_{ p\in \mathbb{P}_q\setminus \mathcal{P}_q(\log q^m)}
\frac{\lambda_f( p^r)}{ p^{\sigma}} e^{- p/q^{m/(k-1)r}}.
\]
\end{lem}

%
%
\begin{proof}[Proof of Lemma~\ref{keylemma} in the case $1\geq\sigma>1/2$]
Now let $r=1, 2$ (i.e. $\rho=0$).
Because of \eqref{prop-3-pr-4}, the desired assertion is the claim that
\begin{equation*}
\bigg|
\sideset{}{'}\sum_{f \in \mathscr{P}_k(q^m)}
\psi_x(\log L_{\mathbb{P}_q}(\Sym_f^r, \sigma))
-
\widetilde{\mathcal{M}}_{\sigma}(\Sym^r, x)
\bigg|
\end{equation*}
tends to $0$ as $q^m\to \infty$, 
when $1\geq \sigma > 1/2$.
\par
Using \eqref{ihara}, we can see the following inequality:
\begin{align}\label{basic}
&
\bigg|
\sideset{}{'}\sum_{f \in \mathscr{P}_k(q^m)}
\psi_x(\log L_{\mathbb{P}_q}(\Sym_f^r, \sigma)) 
-
\widetilde{\mathcal{M}}_{\sigma}(\Sym^r, x)
\bigg|
\nonumber\\
\leq&
\bigg|
\sideset{}{'}\sum_{f \in \mathscr{P}_k(q^m)}
\bigg(
\psi_x(\log L_{\mathbb{P}_q}(\Sym_f^r, \sigma)) 
-
\psi_x(\log L_{\mathcal{P}_q(\log q^m)}(\Sym_f^r, \sigma))
\bigg)
\bigg|
\nonumber\\
&+
\bigg|
\sideset{}{'}\sum_{f \in \mathscr{P}_k(q^m)}
\psi_x(\log L_{\mathcal{P}_q(\log q^m)}(\Sym_f^r, \sigma))
-
\widetilde{\mathcal{M}}_{\sigma, \mathcal{P}_q(\log q^m)}(\Sym^r, x)
\bigg|
\nonumber\\
&+
\left|
\widetilde{\mathcal{M}}_{\sigma, \mathcal{P}_q(\log q^m)}(\Sym^r, x)
-
\widetilde{\mathcal{M}}_{\sigma}(\Sym^r, x)
\right|
\nonumber\\
\ll &
\sideset{}{'}\sum_{f \in \mathscr{P}_k(q^m)}
|x| \bigg(\big|\log L_{\mathbb{P}_q}(\Sym_f^r, \sigma)
- \log L_{\mathcal{P}_q(\log q^m)}(\Sym_f^r, \sigma)
-\mathcal{S}_r\big| +\big|\mathcal{S}_r\big|
\nonumber\\
&+
\bigg|
\sideset{}{'}\sum_{f \in \mathscr{P}_k(q^m)}
\psi_x(\log L_{\mathcal{P}_q(\log q^m)}(\Sym_f^r, \sigma))
-
\widetilde{\mathcal{M}}_{\sigma, \mathcal{P}_q(\log q^m)}(\Sym^r, x)
\bigg|
\nonumber\\
&+
\left|
\widetilde{\mathcal{M}}_{\sigma, \mathcal{P}_q(\log q^m)}(\Sym^r, x)
-
\widetilde{\mathcal{M}}_{\sigma}(\Sym^r, x)
\right|,
\end{align}
this sum being denoted by
\begin{align*}
\mathcal{X}_{\log q^m}+\mathcal{Y}_{\log q^m}
+\mathcal{Z}_{\log q^m},
\end{align*}
say.
From Proposition~\ref{tildeM}, for any $\varepsilon>0$, there exists a number $N_0=N_0(\varepsilon)$ for which
\[
\left|
\widetilde{\mathcal{M}}_{\sigma, \mathcal{P}_q(\log q^m)}(\Sym^r, x)
-
\widetilde{\mathcal{M}}_{\sigma}(\Sym^r, x)
\right|<\varepsilon
\]
holds for any $q^m>N_0$, uniformly in $x\in\mathbb{R}$.
Therefore
\begin{align}\label{lim-Z}
\lim_{q^m\to\infty}\mathcal{Z}_{\log q^m}
=0.
\end{align}
On the estimate of $\mathcal{X}_{\log q^m}$, by using \eqref{P1} and \eqref{F}, we find that
\begin{align*}
&
\sideset{}{'}\sum_{f \in \mathscr{P}_k(q^m)}
|x| \bigg(
\big|
\log L_{\mathbb{P}_q}(\Sym_f^r, \sigma)
- \log L_{\mathcal{P}_q(\log q^m)}(\Sym_f^r, \sigma)
-\mathcal{S}_r\big|
\bigg)
\to 0
\end{align*}
as $q^m$ tends to $\infty$, uniformly in $|x|\leq R$.
Next, from \cite[(6.4)]{mu}, we have
\begin{equation*}
\sideset{}{'}\sum_{f \in \mathscr{P}_k(q^m)}
\big|\mathcal{S}_r\big|
\ll 
(\log q^m)^{-\delta/2}.
\end{equation*}
Hence we see that
\begin{align}\label{lim-X}
\lim_{q^m\to\infty} \mathcal{X}_{\log q^m}
=
0
\end{align}
uniformly in $|x|\leq R$.
\par
The remaining part of this section is devoted to the estimate of $\mathcal{Y}_{\log q^m}$.
According to the method \cite{im_2011} and \cite{mu}, we begin with the Taylor expansion
\[
\psi_x(g_{\sigma,  p}(t_p))=\exp(ixg_{\sigma,  p}(t_p))
=1+\sum_{n=1}^{\infty}\frac{(ix)^n}{n!}g_{\sigma,  p}^n(t_p),
\]
where
\begin{align*}
g_{\sigma,  p}^n(t_p)
=&
\left(-\log(1-t_p p^{-\sigma})\right)^n
= 
\left(\sum_{j=1}^{\infty}\frac{1}{j}\left(\frac{t_p}{ p^{\sigma}}\right)^j\right)^n
\\
=&
\sum_{a=1}^{\infty} \bigg(\sum_{\substack{a=j_1+\ldots+j_n\\ j_{\ell} \geq 1}}\frac{1}{j_1j_2\cdots j_n}\bigg)
\left(\frac{t_p}{ p^{\sigma}}\right)^a.
\end{align*}
Hence
\begin{align*}
\psi_x(g_{\sigma,  p}(t_p))
= &
1+\sum_{n=1}^{\infty}\frac{(ix)^n}{n!}
\sum_{a=1}^{\infty} \bigg(\sum_{\substack{a=j_1+\ldots+j_n\\ j_{\ell} \geq 1}}\frac{1}{j_1j_2\cdots j_n}\bigg)
\left(\frac{t_p}{ p^{\sigma}}\right)^a\\
=&1+
\sum_{a=1}^{\infty} 
\sum_{n=1}^a\frac{(ix)^n}{n!}\bigg(\sum_{\substack{a=j_1+\ldots+j_n\\ j_{\ell} \geq 1}}\frac{1}{j_1j_2\cdots j_n}\bigg)
\left(\frac{t_p}{ p^{\sigma}}\right)^a,
\end{align*}
which we can write as
\begin{align}\label{G-exp}
\psi_x(g_{\sigma,  p}(t_p))=\sum_{a=0}^{\infty}G_a( p, x) t_p^a
\end{align}
with
\begin{align*}
G_a( p, x)
=& \begin{cases}
       1 & a=0,\\
       \displaystyle{\frac{1}{ p^{a\sigma}}\sum_{n=1}^a\frac{(ix)^n}{n!}\bigg(\sum_{\substack{a=j_1+\ldots+j_n\\ j_{\ell} \geq 1}}\frac{1}{j_1j_2\cdots j_n}\bigg)} 
         & a\geq 1.      
       \end{cases}
\end{align*}
Define
\begin{align*}
G_a(x)
=&
 \begin{cases}
       1 & a=0,\\
       \displaystyle\sum_{n=1}^a\frac{x^n}{n!}
       \binom{a-1}{n-1}
         & a\geq 1.      
       \end{cases}
\end{align*}
This notation is the same as \cite[(63)]{im_2011}.
We have
\begin{equation*}
|G_a( p, x)|
\leq
\frac{1}{ p^{a\sigma}}G_a(|x|)
\end{equation*}
(see \cite[(65)]{im_2011}).
From \eqref{g_sigma} and \eqref{G-exp}, we find (using the convention $\lambda_f(p^n)=0$ for $n<0$) that
\begin{align}\label{daiji}
&
\sideset{}{'}\sum_{f \in \mathscr{P}_k(q^m)}
\psi_x(\log L_{\mathcal{P}_q(\log q^m)}(\Sym_f^r, \sigma))
\nonumber\\
=&
\sideset{}{'}\sum_{f \in \mathscr{P}_k(q^m)}
\prod_{ p \in \mathcal{P}_q(\log q^m)}
\psi_x (\mathscr{G}_{\sigma,  p}(\alpha_f^r( p)))
\nonumber\\
=&
\sideset{}{'}\sum_{f \in \mathscr{P}_k(q^m)}
\prod_{ p \in \mathcal{P}_q(\log q^m)}
\psi_x (g_{\sigma,  p}(\alpha_f^r(p)))\psi_x(g_{\sigma,  p}(\beta_f^r( p)))
\psi_x(g_{\sigma,  p}(\delta_{r,\text{even}}))
\nonumber\\
=&
\sideset{}{'}\sum_{f \in \mathscr{P}_k(q^m)}
\prod_{ p \in \mathcal{P}_q(\log q^m)}
\Big(\sum_{a_p=0}^{\infty}G_{a_p}( p, x) 
\alpha_f^{a_p r}( p)
\Big)
\Big(\sum_{b_p=0}^{\infty}G_{b_p}( p, x) 
\beta_f^{b_p r}( p)
\Big)
\nonumber\\
&\times
\Big(\sum_{c_p=0}^{\infty}G_{c_p}( p, x) 
\delta_{r,\text{even}}^{c_p}
\Big)
\nonumber\\      
=&
\sideset{}{'}\sum_{f \in \mathscr{P}_k(q^m)}
\prod_{ p \in \mathcal{P}_q(\log q^m)}
\Big(\sum_{c_p=0}^{\infty}G_{c_p}( p, x) 
\delta_{r,\text{even}}^{c_p}
\Big)
\Big(
\sum_{a_p=0}^{\infty} G_{a_p}^2( p, x) 
\nonumber\\
&
+ 
\sum_{0\leq b_p <  a_p}
G_{a_p}( p, x) G_{b_p}( p, x) 
\alpha_f^{a_p r}( p) \beta_f^{b_p r}( p)
\nonumber\\
&
+ 
\sum_{0\leq a_p < b_p} 
G_{a_p}( p, x) G_{b_p}( p, x) 
\alpha_f^{a_p r}( p) \beta_f^{b_p r}( p)
\Big)
\nonumber\\      
=&
\sideset{}{'}\sum_{f \in \mathscr{P}_k(q^m)}
\prod_{ p \in \mathcal{P}_q(\log q^m)}
\Big(\sum_{c_p=0}^{\infty}G_{c_p}( p, x) 
\delta_{r,\text{even}}^{c_p}
\Big)
\Big(
\sum_{a_p=0}^{\infty} G_{a_p}^2( p, x) 
\nonumber\\
&
+ 
\sum_{0\leq b_p}
\sum_{\nu_p=1}^{\infty}
G_{b_p+\nu_p}( p, x) G_{b_p}( p, x) 
\alpha_f^{\nu_p r}( p)
\nonumber\\
&
+ 
\sum_{0\leq a_p}
\sum_{\nu_p=1}^{\infty}
G_{a_p}( p, x) G_{a_p+\nu_p}( p, x) 
\alpha_f^{-\nu_p r}( p) 
\Big)
\nonumber\\      
=&
\sideset{}{'}\sum_{f \in \mathscr{P}_k(q^m)}
\prod_{ p \in \mathcal{P}_q(\log q^m)}
\Big(\sum_{c_p=0}^{\infty}G_{c_p}( p, x) 
\delta_{r,\text{even}}^{c_p}
\Big)
\Big(
\sum_{a_p=0}^{\infty} G_{a_p}^2( p, x)
\nonumber\\
&
+ 
\sum_{0\leq b_p}
\sum_{\nu_p=1}^{\infty}
G_{b_p+\nu_p}( p, x) G_{b_p}( p, x) 
(\alpha_f^{\nu_p r}( p)+ \alpha_f^{-\nu_p r}( p))
\Big)
\nonumber\\
=&
\sideset{}{'}\sum_{f \in \mathscr{P}_k(q^m)}
\prod_{ p \in \mathcal{P}_q(\log q^m)}
\Big(\sum_{c_p=0}^{\infty}G_{c_p}( p, x) 
\delta_{r,\text{even}}^{c_p}
\Big)
\nonumber\\
&
\times\bigg(
\sum_{\nu_p=0}^{\infty}
(\lambda_f(p^{\nu_p r})- \lambda_f(p^{\nu_p r-2}))
\Big(
\sum_{0\leq b_p}G_{b_p+\nu_p}( p, x) G_{b_p}( p, x)
\Big)
\bigg)
\nonumber\\
=:&
\sideset{}{'}\sum_{f \in \mathscr{P}_k(q^m)}
\prod_{ p \in \mathcal{P}_q(\log q^m)}
\Big(\sum_{c_p=0}^{\infty}G_{c_p}( p, x) 
\delta_{r,\text{even}}^{c_p}
\Big)
\nonumber\\
&
\times\bigg(
\sum_{\nu_p=0}^{\infty}
(\lambda_f(p^{\nu_p r})- \lambda_f(p^{\nu_p r-2}))
G_{p, x}(\nu_p)
\bigg),
\end{align}
where
\[
G_{p, x}(\nu_p)
=\sum_{0\leq b_p}G_{b_p+\nu_p}(p,x)G_{b_p}(p,x).
\]
For $r=1$, the right-hand side of \eqref{daiji} is
\begin{align}\label{daiji_r=1}
=&
\sideset{}{'}\sum_{f \in \mathscr{P}_k(q^m)}
\prod_{ p \in \mathcal{P}_q(\log q^m)}
\nonumber\\
&
\bigg(
G_{p, x}(0)
+
\lambda_f(p)
G_{p, x}(1)
+
(\lambda_f(p^{2})- 1)
G_{p, x}(2)
\nonumber\\
&+
\sum_{\nu_p=3}^{\infty}
(\lambda_f(p^{\nu_p})- \lambda_f(p^{\nu_p-2}))
G_{p, x}(\nu_p)
\bigg)
\nonumber\\
=&
\sideset{}{'}\sum_{f \in \mathscr{P}_k(q^m)}
\prod_{ p \in \mathcal{P}_q(\log q^m)}
\nonumber\\
&
\bigg(
(G_{p, x}(0)-G_{p, x}(2))
+
\sum_{\nu_p=1}^{\infty}
\lambda_f(p^{\nu_p})
G_{p, x}(\nu_p)
-
\sum_{\nu_p=1}^{\infty}
\lambda_f(p^{\nu_p})
G_{p, x}(\nu_p+2)
\bigg)
\nonumber\\
=&
\sideset{}{'}\sum_{f \in \mathscr{P}_k(q^m)}
\prod_{ p \in \mathcal{P}_q(\log q^m)}
\bigg(
\sum_{\nu_p=0}^{\infty}
\lambda_f(p^{\nu_p})
\Big(
G_{p, x}(\nu_p)
-
G_{p, x}(\nu_p+2)
\Big)
\bigg).
\end{align}
For $r=2$, the right-hand side of \eqref{daiji} is
\begin{align}\label{daiji_r=2}
=&
\sideset{}{'}\sum_{f \in \mathscr{P}_k(q^m)}
\prod_{ p \in \mathcal{P}_q(\log q^m)}
\Big(\sum_{c_p=0}^{\infty}G_{c_p}( p, x) 
\Big)
\nonumber\\
&
\bigg(
G_{p, x}(0)
+
(\lambda_f(p^{2})- 1)
G_{p, x}(1)
+
\sum_{\nu_p=2}^{\infty}
(\lambda_f(p^{2\nu_p})- \lambda_f(p^{2\nu_p-2}))
G_{p, x}(\nu_p)
\bigg)
\nonumber\\
=&
\sideset{}{'}\sum_{f \in \mathscr{P}_k(q^m)}
\prod_{ p \in \mathcal{P}_q(\log q^m)}
\Big(\sum_{c_p=0}^{\infty}G_{c_p}( p, x) 
\Big)
\nonumber\\
&
\bigg(
\Big(
G_{p, x}(0)-G_{p, x}(1)
\Big)
+
\sum_{\nu_p=1}^{\infty}
\lambda_f(p^{2\nu_p})
G_{p, x}(\nu_p)
-
\sum_{\nu_p=1}^{\infty}
\lambda_f(p^{2\nu_p}))
G_{p, x}(\nu_p+1)
\bigg)
\nonumber\\
=&
\sideset{}{'}\sum_{f \in \mathscr{P}_k(q^m)}
\prod_{ p \in \mathcal{P}_q(\log q^m)}
\Big(\sum_{c_p=0}^{\infty}G_{c_p}( p, x) 
\Big)
\nonumber\\
&\times
\bigg(
\sum_{\nu_p=0}^{\infty}
\lambda_f(p^{2\nu_p})
\Big(
G_{p, x}(\nu_p)
-
G_{p, x}(\nu_p+1)
\Big)
\bigg).
\end{align}
Let
\begin{align*}
  G_p(\text{main}, x)
  =&
  G_{p, x}(0)-G_{p, x}(\diamond)
  =
\begin{cases}
  G_{p, x}(0)-G_{p, x}(2)& r=1,\\
  G_{p, x}(0)-G_{p, x}(1)& r=2,
\end{cases}
\\
G_{f,p}(\text{error}, x)
=&
\sum_{\nu_p=1}^{\infty}\lambda_f(p^{\nu_p r})
(G_{p, x}(\nu_p)-G_{p, x}(\nu_p+\diamond))
\\
=&
\begin{cases}
  \displaystyle \sum_{\nu_p=1}^{\infty}\lambda_f(p^{\nu_p})(G_{p, x}(\nu_p)-G_{p, x}(\nu_p+2)) & r=1,\\
  \displaystyle \sum_{\nu_p=1}^{\infty}\lambda_f(p^{2\nu_p})(G_{p, x}(\nu_p)-G_{p, x}(\nu_p+1)) & r=2.
\end{cases}
\end{align*}
Write $\mathcal{P}_q(\log q^m)=\{p_1, p_2, \ldots, p_L\}$, where $p_{\ell}$ means the $\ell$th prime number.
From \eqref{daiji}, \eqref{daiji_r=1}, \eqref{daiji_r=2} and \eqref{P1} we have
\begin{align}\label{daiji12}
&
\sideset{}{'}\sum_{f \in \mathscr{P}_k(q^m)}
  \psi_x(\log L_{\mathcal{P}_q(\log q^m)}(\Sym_f^r, \sigma))
\nonumber\\
=&
\sideset{}{'}\sum_{f \in \mathscr{P}_k(q^m)}
\prod_{ p \in \mathcal{P}_q(\log q^m)}
\Big(\sum_{c_p=0}^{\infty}G_{c_p}( p, x) 
\delta_{r,\text{even}}^{c_p}
\Big)
(G_p(\text{main},x) + G_{f,p}(\text{error},x))
\nonumber\\
=&
\bigg(
\prod_{ p \in \mathcal{P}_q(\log q^m)}
\Big(\sum_{c_p=0}^{\infty}G_{c_p}( p, x) 
\delta_{r,\text{even}}^{c_p}
\Big)
\bigg)
\sideset{}{'}\sum_{f \in \mathscr{P}_k(q^m)}
\bigg(
\prod_{ p \in \mathcal{P}_q(\log q^m)}
G_p(\text{main},x)
\nonumber\\
&
+
\sum_{\substack{(j_0,\ldots,j_L)\neq (0,\ldots,0)\\j_{\ell}\in\{0,1\}}}
\prod_{\ell=1}^L
G_{p_{\ell}}^{1-j_{\ell}}(\text{main},x)
\prod_{\ell=1}^L
G_{f,p_{\ell}}^{j_{\ell}}(\text{error},x)
\bigg)
\nonumber\\
=&
\bigg(
\prod_{ p \in \mathcal{P}_q(\log q^m)}
\Big(\sum_{c_p=0}^{\infty}G_{c_p}( p, x) 
\delta_{r,\text{even}}^{c_p}
\Big)
\bigg)
\nonumber\\
&
\times
\bigg(
\prod_{ p \in \mathcal{P}_q(\log q^m)}
G_p(\text{main},x)
+
E(q^m)
\prod_{ p \in \mathcal{P}_q(\log q^m)}
G_p(\text{main},x)
\nonumber\\
&
+
\sideset{}{'}\sum_{f \in \mathscr{P}_k(q^m)}
\sum_{\substack{(j_0,\ldots,j_L)\neq (0,\ldots,0)\\j_{\ell}\in\{0,1\}}}
\prod_{\ell=1}^L
G_{p_{\ell}}^{1-j_{\ell}}(\text{main},x)
\prod_{\ell=1}^L
G_{f,p_{\ell}}^{j_{\ell}}(\text{error},x)
\bigg).
\end{align}
On the other hand, from \eqref{FFourier-transf}, Proposition~\ref{M_P} and \eqref{G-exp}, 
\begin{align}\label{Daiji}
&
\widetilde{\mathcal{M}}_{\sigma, \mathcal{P}_q(\log q^m)}(\Sym^r, x)
=
\prod_{p \in \mathcal{P}_q(\log q^m)}
\widetilde{\mathcal{M}}_{\sigma, p}(\Sym^r, x)
\nonumber\\
=&
\prod_{p \in \mathcal{P}_q(\log q^m)}
\int_{\mathbb{R}}
\mathcal{M}_{\sigma, p}(\Sym^r, u)
\psi_x(u) du
\nonumber\\
=&
\prod_{p \in \mathcal{P}_q(\log q^m)}
\int_{\mathcal{T}}
\psi_x\big(\mathscr{G}_{\sigma, p}(t_p^r))\big)
\bigg(\frac{t_p^2 -2 + t_p^{-2}}{-2}\bigg)
\frac{dt_p}{2\pi i t_p}
\nonumber\\
=&
\prod_{ p\in \mathcal{P}_q(\log q^m)}
\bigg(
\int_{\mathcal{T}}
\psi_x\big(g_{\sigma, p}( t_p^r)\big)
\psi_x\big(g_{\sigma, p}( t_p^{-r})\big)
\psi_x\big(g_{\sigma, p}(\delta_{r,\text{even}})\big)
\nonumber\\
&\times
\bigg(\frac{t_p^2 -2 + t_p^{-2}}{-2}\bigg)
\frac{dt_p}{2\pi i t_p}
\bigg)
\nonumber\\
=&
\prod_{ p\in \mathcal{P}_q(\log q^m)}
\int_{\mathcal{T}}
\Big(\sum_{a_p=0}^{\infty}G_{a_p}( p, x) t_p^{a_p r} \Big)
\Big(\sum_{b_p=0}^{\infty}G_{b_p}( p, x) t_p^{-b_p r}\Big)
\nonumber\\
&\times
\Big(\sum_{c_p=0}^{\infty}G_{c_p}( p, x)\delta_{r,\text{even}}^{c_p}\Big)
\bigg(\frac{t_p^2 -2 + t_p^{-2}}{-2}\bigg)
\frac{dt_p}{2\pi i t_p}
\bigg)
\nonumber\\
=&
\prod_{ p\in \mathcal{P}_q(\log q^m)}
\Big(\sum_{c_p=0}^{\infty}G_{c_p}( p, x)\delta_{r,\text{even}}^{c_p}\Big)
\nonumber\\
&\times
\bigg(
\int_{\mathcal{T}}
\Big(\sum_{a_p=0}^{\infty}G_{a_p}^2( p, x) 
+
\sum_{a_p=0}^{\infty}
\sum_{\substack{b_p=0\\ a_p\neq b_p}}^{\infty}
G_{a_p}( p, x)G_{b_p}( p, x) t_p^{(a_p-b_p)r}
\Big)
\nonumber\\
&\times
\bigg(\frac{t_p^2 -2 + t_p^{-2}}{-2}\bigg)
\frac{dt_p}{2\pi i t_p}
\bigg)
\nonumber\\
=&
\prod_{ p\in \mathcal{P}_q(\log q^m)}
\Big(\sum_{c_p=0}^{\infty}G_{c_p}( p, x)\delta_{r,\text{even}}^{c_p}\Big)
\nonumber\\
&\times
\bigg(
\int_{\mathcal{T}}
\Big(\sum_{a_p=0}^{\infty}G_{a_p}^2( p, x) 
+
\sum_{a_p=0}^{\infty}
\sum_{\nu_p=1}^{\infty}
G_{a_p}( p, x)G_{a_p+\nu_p}( p, x) t_p^{-\nu_p r}
\nonumber\\
&+
\sum_{b_p=0}^{\infty}
\sum_{\nu_p=1}^{\infty}
G_{b_p+\nu_p}( p, x)G_{b_p}( p, x) t_p^{\nu_p r}
\Big)
\bigg(\frac{t_p^2 -2 + t_p^{-2}}{-2}\bigg)
\frac{dt_p}{2\pi i t_p}
\bigg).
\end{align}
For $r=1$, the right-hand side of \eqref{Daiji} is
\begin{align}\label{Daiji_r=1}
=&
\prod_{ p\in \mathcal{P}_q(\log q^m)}
\int_{\mathcal{T}}
\bigg(
\Big(\sum_{a_p=0}^{\infty}G_{a_p}^2( p, x) 
+
\sum_{a_p=0}^{\infty}
\sum_{\nu_p=1}^{\infty}
G_{a_p}( p, x)G_{a_p+\nu_p}( p, x) t_p^{-\nu_p}
\nonumber\\
&+
\sum_{b_p=0}^{\infty}
\sum_{\nu_p=1}^{\infty}
G_{b_p+\nu_p}( p, x)G_{b_p}( p, x) t_p^{\nu_p}
\Big)
\bigg(\frac{t_p^2 -2 + t_p^{-2}}{-2}\bigg)
\frac{dt_p}{2\pi i t_p}
\bigg)
\nonumber\\
=&
\prod_{ p\in \mathcal{P}_q(\log q^m)}
\bigg(
\sum_{a_p=0}^{\infty}G_{a_p}^2( p, x) 
-
\sum_{a_p=0}^{\infty}
G_{a_p}( p, x)G_{a_p+2}( p, x) 
\bigg)
\nonumber\\
=&
\prod_{ p\in \mathcal{P}_q(\log q^m)}
(G_{p, x}(0)-G_{p, x}(2))
=
\prod_{ p\in \mathcal{P}_q(\log q^m)}
G_p(\text{main},x).
\end{align}
For $r=2$, the right-hand side of \eqref{Daiji} is
\begin{align}\label{Daiji_r=2}
=&
\prod_{ p\in \mathcal{P}_q(\log q^m)}
\Big(\sum_{c_p=0}^{\infty}G_{c_p}( p, x)\Big)
\int_T
\bigg(
\Big(\sum_{a_p=0}^{\infty}G_{a_p}^2( p, x)
\nonumber\\
&+
\sum_{a_p=0}^{\infty}
\sum_{\nu_p=1}^{\infty}
G_{a_p}( p, x)G_{a_p+\nu_p}( p, x) t_p^{-2\nu_p}
\nonumber\\
&+
\sum_{b_p=0}^{\infty}
\sum_{\nu_p=1}^{\infty}
G_{b_p+\nu_p}( p, x)G_{b_p}( p, x) t_p^{2\nu_p}
\Big)
\bigg(\frac{t_p^2 -2 + t_p^{-2}}{-2}\bigg)
\frac{dt_p}{2\pi i t_p}
\bigg)
\nonumber\\
=&
\prod_{ p\in \mathcal{P}_q(\log q^m)}
\Big(\sum_{c_p=0}^{\infty}G_{c_p}( p, x)\Big)
\bigg(
\sum_{a_p=0}^{\infty}G_{a_p}^2( p, x) 
-
\sum_{a_p=0}^{\infty}
G_{a_p}( p, x)G_{a_p+1}( p, x)\bigg)
\nonumber\\
=&
\prod_{ p\in \mathcal{P}_q(\log q^m)}
\Big(\sum_{c_p=0}^{\infty}G_{c_p}( p, x)\Big)
(G_{p, x}(0)-G_{p, x}(1))
\nonumber\\
=&
\prod_{ p\in \mathcal{P}_q(\log q^m)}
\Big(\sum_{c_p=0}^{\infty}G_{c_p}( p, x)\Big)
G_p(\text{main}, x).
\end{align}
From \eqref{Daiji}, \eqref{Daiji_r=1} and \eqref{Daiji_r=2}, we obtain
\begin{equation}\label{Daiji12}
\widetilde{\mathcal{M}}_{\sigma, \mathcal{P}_q(\log q^m)}(\Sym^r, x)
=
\prod_{ p\in \mathcal{P}_q(\log q^m)}
\Big(\sum_{c_p=0}^{\infty}G_{c_p}( p, x)\delta_{r,\text{even}}^{c_p}\Big)
G_p(\text{main}, x).
\end{equation}
Since
\[
\bigg|\sum_{c_p=0}^{\infty} G_{c_p}(p, x)\bigg|
=
|\psi_x(g_{\sigma, p}(1))|
=|e^{ix\log(1-p^{-\sigma})}|=1
\]
by \eqref{G-exp}, from
\eqref{daiji12} and \eqref{Daiji12} we now obtain
\begin{align*}
&
\mathcal{Y}_{\mathcal{P}_q(\log q^m)}
\nonumber\\
=&
\prod_{ p\in \mathcal{P}_q(\log q^m)}
\Big|\sum_{c_p=0}^{\infty}G_{c_p}( p, x)\delta_{r,\text{even}}^{c_p}\Big|
\bigg|
E(q^m)
\prod_{ p \in \mathcal{P}_q(\log q^m)}
G_p(\text{main},x)
\nonumber\\
&
+
\sideset{}{'}\sum_{f \in \mathscr{P}_k(q^m)}
\sum_{\substack{(j_0,\ldots,j_L)\neq (0,\ldots,0)\\j_{\ell}\in\{0,1\}}}
\prod_{\ell=1}^L
G_{p_{\ell}}^{1-j_{\ell}}(\text{main},x)
\prod_{\ell=1}^L
G_{f,p_{\ell}}^{j_{\ell}}(\text{error},x)
\bigg|
\nonumber\\
=&
\bigg|
E(q^m)
\prod_{ p \in \mathcal{P}_q(\log q^m)}
G_p(\text{main},x)
\nonumber\\
&
+
\sum_{\substack{(j_0,\ldots,j_L)\neq (0,\ldots,0)\\j_{\ell}\in\{0,1\}}}
\prod_{\ell=1}^L
G_{p_{\ell}}^{1-j_{\ell}}(\text{main},x)
\bigg(
\sideset{}{'}\sum_{f \in \mathscr{P}_k(q^m)}
\prod_{\ell=1}^L
G_{f,p_{\ell}}^{j_{\ell}}(\text{error},x)
\bigg)
\bigg|
\nonumber\\
\leq &
\bigg|
E(q^m)
\prod_{ p \in \mathcal{P}_q(\log q^m)}
G_p(\text{main},x)
\bigg|
\nonumber\\
&
+
\bigg|
\sum_{\substack{(j_0,\ldots,j_L)\neq (0,\ldots,0)\\j_{\ell}\in\{0,1\}}}
\prod_{\ell=1}^L
G_{p_{\ell}}^{1-j_{\ell}}(\text{main},x)
\bigg(
\sideset{}{'}\sum_{f \in \mathscr{P}_k(q^m)}
\prod_{\ell=1}^L
G_{f,p_{\ell}}^{j_{\ell}}(\text{error},x)
\bigg)
\bigg|
\nonumber\\
=: &
\mathcal{Y}''_{\mathcal{P}_q(\log q^m)}
+
\mathcal{Y}'_{\mathcal{P}_q(\log q^m)},
\end{align*}
say.
Here, the definition of $\mathcal{Y}''_{\mathcal{P}_q(\log q^m)}$ and $\mathcal{Y}'_{\mathcal{P}_q(\log q^m)}$ are different from those in \cite{mu}, but we will prove that they tend to $0$ by the same argument as in \cite[Section 6]{mu}.
We first consider the inner sum in the definition of $\mathcal{Y}'_{\mathcal{P}_q(\log q^m)}$.
Let
\[
\mathsf{G}_x(n)
=
\prod_{\substack{1\leq \ell \leq L \\ j_{\ell}=1}}
(G_{p_{\ell}, x}(\nu_{p_{\ell}})-G_{p_{\ell}, x}(\nu_{p_{\ell}}+\diamond)) 
\]
if $n$ is of the form
\[
n=\prod_{\substack{1\leq \ell \leq L \\ j_{\ell}=1}} p_{\ell}^{\nu_{p_{\ell}} r},
\quad \nu_{p_{\ell}}\geq 1,
\]
and $\mathsf{G}_x(n)=0$ otherwise.
Let $M$ be a positive number, and $\eta>0$ be an arbitrarily small positive number.
We see that
\begin{align*}
  &
  \sideset{}{'}\sum_{f \in \mathscr{P}_k(q^m)}
  \prod_{\ell=1}^L
  G_{f,p_{\ell}}^{j_{\ell}}(\text{error},x)
  \nonumber\\
  =&
  \sideset{}{'}\sum_{f \in \mathscr{P}_k(q^m)}
  \prod_{\ell=1}^L
  \bigg(
  \sum_{\nu_{p_{\ell}}=1}^{\infty}\lambda_f(p_{\ell}^{\nu_{p_{\ell}} r})
  (G_{p_{\ell}, x}(\nu_{p_{\ell}})-G_{p_{\ell}, x}(\nu_{p_{\ell}}+\diamond))
  \bigg)^{j_{\ell}}
  \nonumber\\
  =&
  \sideset{}{'}\sum_{f \in \mathscr{P}_k(q^m)}
  \sum_{n>1}\lambda_f(n)\mathsf{G}_x(n)
  \nonumber\\
  =&
  \sideset{}{'}\sum_{f \in \mathscr{P}_k(q^m)}
  \bigg(
  \sum_{M\geq n>1}\lambda_f(n)\mathsf{G}_x(n)
  +
  \sum_{n> M}\lambda_f(n)\mathsf{G}_x(n)
  \bigg)
  \nonumber\\
\ll&
  \bigg|
  \sideset{}{'}\sum_{f \in \mathscr{P}_k(q^m)}
\sum_{M\geq n>1}\lambda_f(n)\mathsf{G}_x(n)
\bigg|
  +
\bigg|  
  \sideset{}{'}\sum_{f \in \mathscr{P}_k(q^m)}
  \sum_{n>M}\lambda_f(n)\mathsf{G}_x(n)
\bigg|
\nonumber\\
\ll&
  \bigg|
  \sideset{}{'}\sum_{f \in \mathscr{P}_k(q^m)}
\sum_{M\geq n>1}\lambda_f(n)\mathsf{G}_x(n)
\bigg|
  +
  \sideset{}{'}\sum_{f \in \mathscr{P}_k(q^m)}
  \sum_{n>M} n^{\eta}|\mathsf{G}_x(n)|
\nonumber\\
\ll&
E(q^m) \sum_{M\geq n>1} n^{(k-1)/2} |\mathsf{G}_x(n)|
  +
\sum_{n>M} n^{\eta}|\mathsf{G}_x(n)|,
\end{align*}
where on the last inequality we used \eqref{P}.
Therefore
\begin{align}\label{Y-0-bis}
&\mathcal{Y}'_{\mathcal{P}_q(\log q^m)}\ll
\sum_{\substack{(j_0,\ldots,j_L)\neq (0,\ldots,0)\\j_{\ell}\in\{0,1\}}}
\prod_{\ell=1}^L
G_{p_{\ell}}^{1-j_{\ell}}(\text{main},x)\nonumber\\
&\times\bigg(E(q^m) \sum_{M\geq n>1} n^{(k-1)/2} |\mathsf{G}_x(n)|
  +
\sum_{n>M} n^{\eta}|\mathsf{G}_x(n)|\bigg).
\end{align}
This corresponds to \cite[(6.13)]{mu}.
By \cite[(6.14) and (6.15)]{mu},
we have the estimates of $G_{p,x}(\nu_p)$ as
\[
|G_{p, x}(\nu_p)|
\leq
\begin{cases}
  \displaystyle \bigg(\exp\bigg(\frac{|x|}{p^{\sigma}-1}\bigg)\bigg)^2 & \nu_p=0, \\
  \displaystyle \frac{1}{p^{\nu_p \sigma/2}}\bigg(\exp\bigg(\frac{|x|}{p^{\sigma}-1}\bigg)\bigg)^2\bigg(\exp\bigg(\frac{|x|}{p^{\sigma/2}-1}\bigg)\bigg)
  & \nu_p \neq 0.
  \end{cases}
\]
These estimates yield that for $\sigma>1/2$ and $|x|<R$, there exists a large
$p_0=p_0(\sigma,R)$ for which 
\begin{align*}
  &
  |G(\text{main}, x)|
  =
  |G_{p, x}(0)-G_{p, x}(\diamond)|
  \nonumber\\
  \leq&
  \bigg(\exp\bigg(\frac{|x|}{p^{\sigma}-1}\bigg)\bigg)^2\bigg(1+
  \frac{1}{p^{\nu_p\sigma/2}}\exp\bigg(\frac{|x|}{p^{\sigma/2}-1}\bigg)\bigg)
  \leq
  2
\end{align*}
holds for any $p>p_0$.   Therefore
\begin{align*}
\mathcal{Y}''_{\mathcal{P}_q(\log q^m)}
=
|E(q^m)|
\prod_{\ell=1}^L
|G_{p_{\ell}}(\text{main},x)|
\ll_{\sigma,R} 
|E(q^m)| 2^L \to 0
\end{align*}
as $q^m\to \infty$ by the same argument as in \cite[p.~2077]{mu}.
Further, 
since $\nu_p\geq 1$ we have
\begin{align}\label{Y'1}
|\mathsf{G}_x(n)|
\leq &
\prod_{\substack{1\leq \ell \leq L\\j_{\ell}=1}}
(2\max\{|G_{p_{\ell}, x}(\nu_p)|,\; |G_{p_{\ell}, x}(\nu_p+\diamond)|\})
\nonumber\\
\leq&
\frac{2^L}{n^{\sigma r/2}}
\prod_{\substack{1\leq \ell \leq L\\j_{\ell}=1}}
\bigg(\exp\bigg(\frac{|x|}{p^{\sigma}-1}\bigg)\bigg)^2\bigg(\exp\bigg(\frac{|x|}{p^{\sigma/2}-1}\bigg)\bigg).
\end{align}
From the estimates \eqref{Y-0-bis} and \eqref{Y'1}, we see that $\mathcal{Y}'_{\log q^m}\to 0$ as $q^m\to\infty$ by the same argument as in \cite[pp.~2075--2077]{mu}.
Therefore we conclude that
\begin{equation}\label{lim-Y}
 \lim_{q^m\to\infty} \mathcal{Y}_{\log q^m}\to 0.
\end{equation}
\par
Finally we see that Lemma~\ref{keylemma} is established, by substituting \eqref{lim-Z}, \eqref{lim-X} and \eqref{lim-Y} into \eqref{basic}.
\end{proof}
%
%
\section{The first moment of $\log{L}_{\mathbb{P}_q}(\Sym_f^r, \sigma)$}\label{sec-Mine}
In this and the next sections we will prove Theorem \ref{main2}.
Let $\Psi_1(x)=cx$ with a constant $c \neq 0$. 
Our aim in this section is to show that in both Cases I and II 
\begin{gather}\label{eq:1stMoment}
\Avg \Psi_1(\log{L}_{\mathbb{P}_q}(\Sym_f^r, \sigma)) 
= c
\begin{cases}
-2^{-1} \sum_{p \in \mathbb{P}_*} p^{-2\sigma}
& \text{$r$ is odd}, 
\\
\sum_{p \in \mathbb{P}_*} \sum_{j > 1} j^{-1} p^{-j \sigma}
& \text{$r$ is even}
\end{cases}
\end{gather}
holds for any $\sigma>1/2$ under the assumptions of Theorem \ref{main2}. 
The following two lemmas are analogous to Lemmas 8.1 and 8.2 of Granville and Soundararajan \cite{GranvilleSoundararajan2001}. 
Note that the original results in \cite{GranvilleSoundararajan2001} were proved under milder assumption on zero-free regions, but we suppose GRH in the following for convenience. 
%
%
\begin{lem}\label{lem:1}
Let $s=\sigma+it$ with $\sigma>1/2$ and $|t| \leq q^{dm}$, where $d$ is a constant which does not depend on $q$ and $m$. 
Suppose Assumption (GRH). 
Then 
\begin{gather}\label{eq:upperBD}
|\log{L}_{\mathbb{P}_q}(\Sym_f^r, s)|
\ll_{r,d} \frac{\log{q^m}}{\sigma-1/2}
\end{gather}
for any primitive form $f$ in $S_k(q^m)$ such that $L(\Sym_f^r, s)$ is an entire function. 
\end{lem}
\begin{proof}
Let $s=\sigma+it$ with $\sigma>1$. 
From \eqref{log-L-P-def} we see that $\log{L}_{\mathbb{P}_q}(\Sym_f^r, s)$ has the Dirichlet series representation
\begin{gather}\label{eq:Dseries}
\log{L}_{\mathbb{P}_q}(\Sym_f^r, s)
= \sum_{n=2}^{\infty} \frac{\gamma(n; \Sym_f^r) \Lambda(n)}{\log{n}} n^{-s}, 
\end{gather}
where $\Lambda(n)$ denotes the usual von Mangoldt function, and
$\gamma(n; \Sym_f^r)$ is determined by
\begin{align}\label{eq:gamma}
\gamma(p^j; \Sym_f^r)
&= \sum_{h=0}^{r} \alpha_f(p)^{j(r-h)} \beta_f(p)^{jh}
\nonumber\\
&= \sum_{h=0}^{\rho} 2 \cos(j(r-2h) \theta_f(p)) +\delta_{r, \mathrm{even}} 
\end{align}
if $p \in \mathbb{P}_q$ and $j \geq1$, and $\gamma(n; \Sym_f^r)=0$ otherwise. 
We have $|\gamma(n; \Sym_f^r)| \leq r+1$ for any $n$ by definition. 
Therefore formula \eqref{eq:Dseries} yields
\begin{gather}\label{eq:uppBD1}
|\log{L}_{\mathbb{P}_q}(\Sym_f^r, s)|
\leq \sum_{n=2}^{\infty} \frac{(r+1) \Lambda(n)}{\log{n}} n^{-\sigma}
=(r+1) \log{\zeta}(\sigma)
\end{gather}
for $s=\sigma+it$ with $\sigma>1$. 
Hence \eqref{eq:upperBD} follows for $\sigma \geq 2$, and we assume $1/2<\sigma<2$ in the following. 
Put $\sigma_0=\frac{1}{2}(\sigma+\frac{1}{2})$. 
We consider closed discs $\mathcal{R}_1$ and $\mathcal{R}_2$,
centered at $s_0=2+it$, and of radii $R_1=2-\sigma_0$ and $R_2=2-\sigma$,
respectively.
Then the GRH implies that $\log{L}_{\mathbb{P}_q}(\Sym_f^r, s)$ is analytic on 
$\mathcal{R}_1$. 
Since the point $s$ belongs to the smaller disc $\mathcal{R}_2$, we deduce 
\begin{align} \label{eq:BCthm}
|\log{L}_{\mathbb{P}_q}(\Sym_f^r, s)|
&\leq \frac{2 R_2}{R_1-R_2} \max_{|s-s_0| \leq R_1} \Re \log{L}_{\mathbb{P}_q}(\Sym_f^r, s)
\nonumber\\
&\qquad +\frac{R_1+R_2}{R_1-R_2} |\log{L}_{\mathbb{P}_q}(\Sym_f^r, s_0)| 
\end{align}
from the Borel-Carath\'{e}odory theorem. 
We have 
\begin{gather*}
\frac{2 R_2}{R_1-R_2}
\ll \frac{1}{\sigma-1/2}
\quad\text{and}\quad
\frac{R_1+R_2}{R_1-R_2}
\ll \frac{1}{\sigma-1/2}
\end{gather*}
with absolute implied constants. 
We estimate the first term of the right-hand side of \eqref{eq:BCthm}. 
As in Section \ref{sec:sym}, we denote by $\pi_r$ the automorphic representation attached to $\Sym_f^r$. 
By \eqref{eq:local} and the bounds for $|\gamma_j(p)|$, we have 
\begin{gather*}
L_q(\pi_r, s)^{-1}
= \prod_{j=1}^{r+1} (1-\gamma_j(p) p^{-s})
\ll_r 1. 
\end{gather*}
Therefore, using formula \eqref{eq:sym_Pq}, we obtain the bound $L_{\mathbb{P}_q}(\Sym_f^r, s) \ll_r |L(\pi_r, s)|$. 
Furthermore, we apply Proposition \ref{prop:convbd} to deduce 
\begin{gather*}
L_{\mathbb{P}_q}(\Sym_f^r, s)
\ll_r \left( q^{c(r) m} (|t|+1)^{r+1} \right)^{1/2}
\ll_r q^{(c(r)+dr+d)m/2}. 
\end{gather*}
Then we see that 
\begin{gather*}
\max_{|s-s_0| \leq R_1} \Re \log{L}_{\mathbb{P}_q}(\Sym_f^r, s)
= \max_{|s-s_0| \leq R_1} \log|{L}_{\mathbb{P}_q}(\Sym_f^r, s)|
\ll_{r,d} \log{q^m}
\end{gather*}
holds. 
By inequality \eqref{eq:uppBD1} we obtain 
\begin{gather*}
|\log{L}_{\mathbb{P}_q}(\Sym_f^r, s_0)|
\leq (r+1) \log{\zeta}(2) 
\ll_r 1
\end{gather*}
due to $\Re(s_0)=2$. 
Hence, we conclude from \eqref{eq:BCthm} and the above estimates that 
\begin{gather*}
|\log{L}_{\mathbb{P}_q}(\Sym_f^r, s)|
\ll_{r, d} \frac{1}{\sigma-1/2} \log{q^m} + \frac{1}{\sigma-1/2}
\ll \frac{\log{q^m}}{\sigma-1/2}
\end{gather*}
as desired. 
\end{proof}
%
%
\begin{lem}\label{lem:2}
Let $\sigma>1/2$ and $2 \leq y \leq q^{dm}$, where $d$ is a constant which does not depend on $q$ and $m$. 
Put $\sigma_1=\min \big\{ \frac{1}{2}(\sigma+\frac{1}{2}), \frac{1}{2}+\frac{1}{\log{y}} \big\}$. 
Suppose Assumption (GRH). 
Then 
\begin{gather*}
\log{L}_{\mathbb{P}_q}(\Sym_f^r, \sigma)
= \sum_{n<y} \frac{\gamma(n; \Sym_f^r) \Lambda(n)}{\log{n}} n^{-\sigma}
+ O_{r,d}\left( \frac{\log{q^m}}{(\sigma_1-1/2)^2} y^{\max\{\sigma_1-\sigma, -1\}} \right)
\end{gather*}
for any primitive form $f$ in $S_k(q^m)$ such that $L(\Sym_f^r, s)$ is an entire function. 
\end{lem}
\begin{proof}
We assume at first $y \in \mathbb{Z}+\frac{1}{2}$ and $2 \leq y \ll q^{dm}$. 
Put
\[
\mathfrak{c}=\mathfrak{c}(\sigma, y)=\max\{1-\sigma, 0\}+\frac{1}{\log{y}}.
\]
Then clearly $\sigma+\mathfrak{c}>1$.
By Perron's formula \cite[Chapter 17]{Davenport2000}, we have 
\begin{align*}
&\frac{1}{2\pi i} \int_{\mathfrak{c}-3iy}^{\mathfrak{c}+3iy} \log{L}_{\mathbb{P}_q}(\Sym_f^r, \sigma+s) y^{s} \,\frac{ds}{s} \\
&=\sum_{n<y} \frac{\gamma(n; \Sym_f^r) \Lambda(n)}{n^{\sigma}\log{n}}
+O\left( \sum_{n=1}^{\infty} \frac{\Lambda(n)}{n^{\sigma +\mathfrak{c}}\log{n}} y^{\mathfrak{c}} 
\min \left\{1, \frac{1}{3y |\log{(y/n)}|} \right\} \right). 
\end{align*}
Note that $3y |\log{(y/n)}| \geq1$ for any $n$. 
(In fact, since $|\log(y/n)|$ attains the minimal value at $n=y+1/2$,
we have 
$3y|\log(y/n)|\geq 3y\log(1+(2y)^{-1})\geq 1$.)
Thus the above error term is
\begin{align}\label{eq:series}
\ll
y^{\max\{-\sigma, -1\}} \sum_{n=1}^{\infty} \frac{\Lambda(n)}{n^{\sigma+\mathfrak{c}}(\log n)|\log{(y/n)}|}. 
\end{align}
We evaluate \eqref{eq:series} by dividing the sum into three subsums.
For the terms for which $n \leq \frac{3}{4} y$ or $n \geq \frac{5}{4} y$ holds, 
we remark that $|\log{(y/n)}|$ has a uniform positive lower bound. 
Hence their contribution is 
\begin{align*}
&\ll y^{\max\{-\sigma,-1\}} \sum_{\substack{n\leq 3y/4 \\ n\geq 5y/4}} \frac{\Lambda(n)}{n^{\sigma+\mathfrak{c}}\log n}
\leq y^{\max\{-\sigma,-1\}} \sum_{n=1}^{\infty} \frac{1}{n^{\sigma+\mathfrak{c}}}\\
& =y^{\max\{-\sigma,-1\}}\zeta(\sigma+\mathfrak{c})
\ll y^{\max\{-\sigma, -1\}} \log y
\end{align*}
by using $\zeta(1+t)\ll t^{-1}$ for $t>0$. 
Next, we consider the terms for which $\frac{3}{4} y<n<\frac{5}{4} y$ holds. 
For $\frac{3}{4} y<n<\frac{5}{4} y$, we know
\begin{align*}
&
|\log{(y/n)}|
\\
=&
\begin{cases}
  \log{(y/n)} = -\log (1-(y-n)/y) \geq (y-n)/y & 1<y/n<4/3,
  \\
  \log{(n/y)} = -\log (1-(n-y)/n) \geq (n-y)/n & 1>y/n>4/5. 
\end{cases}
\end{align*}
Then we see that
\begin{align*}
&
  y^{\max\{-\sigma, -1\}} \sum_{3y/4 < n < 5y/4}\frac{\Lambda(n)}{n^{\sigma+\mathfrak{c}}(\log n)|\log(y/n)|}
\\
\ll&
y^{\max\{-\sigma, -1\}} \sum_{3y/4 < n <y} \frac{y\Lambda(n)}{n^{\sigma+\mathfrak{c}}(y-n)\log n}
\\
&+
y^{\max\{-\sigma, -1\}} \sum_{y < n < 5y/4} \frac{n\Lambda(n)}{n^{\sigma+\mathfrak{c}}(n-y)\log n}
\\
\ll&
y^{\max\{-\sigma, -1\}} \sum_{3y/4 < n <y} \frac{1}{y-n}
+
y^{\max\{-\sigma, -1\}} \sum_{y < n < 5y/4} \frac{1}{n-y}
\end{align*}
since $\sigma+\mathfrak{c}>1$. 
Recalling $y \in \mathbb{Z}+\frac{1}{2}$, we see that the right-hand side is
\begin{align*}
\leq&
y^{\max\{-\sigma, -1\}} \sum_{0 \leq \mu < y/4} \frac{1}{\mu+1/2}
+
y^{\max\{-\sigma, -1\}} \sum_{0 \leq \nu < y/4} \frac{1}{\nu+1/2}
\\
\ll&
y^{\max\{-\sigma, -1\}} \log{y}. 
\end{align*}
Collecting the above results, we obtain the formula
\begin{align}\label{eq:int1}
&\frac{1}{2\pi i} \int_{\mathfrak{c}-3iy}^{\mathfrak{c}+3iy} \log{L}_{\mathbb{P}_q}(\Sym_f^r, \sigma+s) y^{s} \,\frac{ds}{s}
\nonumber\\
&= \sum_{n<y} \frac{\gamma(n; \Sym_f^r) \Lambda(n)}{n^{\sigma}\log{n}} 
+ O\left( y^{\max\{-\sigma, -1\}} \log{y} \right). 
\end{align}
On the other hand, letting $\mathfrak{c}_1=\max\{\sigma_1-\sigma, -1\}$, we obtain
\begin{align*}
&\frac{1}{2\pi i} \int_{\mathfrak{c}-3iy}^{\mathfrak{c}+3iy} \log{L}_{\mathbb{P}_q}(\Sym_f^r, \sigma+s) y^{s} \,\frac{ds}{s} 
= \log{L}_{\mathbb{P}_q}(\Sym_f^r, \sigma) \\
&\qquad+ \frac{1}{2\pi i} \left(\int_{\mathfrak{c}-3iy}^{\mathfrak{c}_1-3iy} 
+ \int_{\mathfrak{c}_1-3iy}^{\mathfrak{c}_1+3iy} 
+ \int_{\mathfrak{c}_1+3iy}^{\mathfrak{c}+3iy} \right) 
\log{L}_{\mathbb{P}_q}(\Sym_f^r, \sigma+s) y^{s} \,\frac{ds}{s} 
\end{align*}
since we do not encounter any poles of the integrand while changing the 
integration contours, except for a simple pole at $s=0$ with residue $\log{L}_{\mathbb{P}_q}(\Sym_f^r, \sigma)$. 
Lastly, we apply Lemma \ref{lem:1} to estimate the integral terms on
the right-hand side.
Since 
$\sigma+\mathfrak{c}_1-1/2=\max\{\sigma_1-1/2,\sigma-3/2\}\geq \sigma_1-1/2$, 
the first and the third integrals are
\[
\ll\int_{\mathfrak{c}_1}^{\mathfrak{c}}\frac{\log q^m}{\sigma+u-1/2}
y^{u-1}du
\ll\frac{\log q^m}{\sigma+\mathfrak{c}_1-1/2}y^{\mathfrak{c}-1}
\ll \frac{\log q^m}{\sigma_1-1/2}y^{\max\{-\sigma,-1\}},
\]
while the second integral is
\[
\ll\int_{-3y}^{3y}\frac{\log q^m}{\sigma+\mathfrak{c}_1-1/2}y^{\mathfrak{c}_1}
\frac{|dv|}{|\mathfrak{c}_1+iv|}
\ll \frac{\log q^m}{\sigma_1-1/2}y^{\max\{\sigma_1-\sigma,-1\}}
\int_{-3y}^{3y}\frac{|dv|}{|\mathfrak{c}_1+iv|},
\]
and further
\[
\int_{-3y}^{3y}\frac{|dv|}{|\mathfrak{c}_1+iv|}
=\int_{|v|\leq 1}+\int_{1<|v|\leq 3y}
\ll \frac{1}{|\mathfrak{c}_1|}+\log y \ll \frac{1}{\sigma_1-1/2}
\]
because of the definition of $\sigma_1$.
Therefore we now arrive at
\begin{align}\label{eq:int2}
&\frac{1}{2\pi i} \int_{\mathfrak{c}-3iy}^{\mathfrak{c}+3iy} \log{L}_{\mathbb{P}_q}(\Sym_f^r, \sigma+s) y^{s} \,\frac{ds}{s}
\nonumber\\
&= \log{L}_{\mathbb{P}_q}(\Sym_f^r, \sigma)
+ O_{r,d}\left( \frac{\log{q^m}}{(\sigma_1-1/2)^2} y^{\max\{\sigma_1-\sigma, -1\}} \right). 
\end{align}
The desired formula follows from \eqref{eq:int1} and \eqref{eq:int2} in the case $y \in \mathbb{Z}+\frac{1}{2}$ and $2 \leq y \ll q^{dm}$. 
If $y \notin \mathbb{Z}+\frac{1}{2}$ and $2 \leq y \leq q^{dm}$, then we take the smallest real number $\widetilde{y} \in \mathbb{Z}+\frac{1}{2}$ greater than $y$. 
Since $y< \widetilde{y} <y+1$, we have 
\begin{gather*}
\sum_{n<\widetilde{y}} \frac{\gamma(n; \Sym_f^r) \Lambda(n)}{n^{\sigma}\log{n}} 
= \sum_{n<y} \frac{\gamma(n; \Sym_f^r) \Lambda(n)}{n^{\sigma}\log{n}} 
+ O\left(y^{-\sigma}\right) 
\end{gather*}
with an absolute implied constant, and 
\begin{gather*}
\frac{\log{q^m}}{(\sigma_1-1/2)^2} \widetilde{y}^{\max\{\sigma_1-\sigma, -1\}}
\leq \frac{\log{q^m}}{(\sigma_1-1/2)^2} y^{\max\{\sigma_1-\sigma, -1\}}. 
\end{gather*}
Therefore the result is true in this case, and we complete the proof.  
\end{proof}
%
%
Then, we proceed to the proof of formula \eqref{eq:1stMoment}. 
By Lemma \ref{lem:2}, we have 
\begin{align}\label{eq:1stMoment1}
&
\sideset{}{'}\sum_{f \in \mathscr{P}_k^\mathrm{E}(q^m)}
\Psi_1(\log{L}_{\mathbb{P}_q}(\Sym_f^r, \sigma)) \nonumber\\
&= c \sum_{n<y} \bigg(
\sideset{}{'}\sum_{f \in \mathscr{P}_k(q^m)}
\gamma(n; \Sym_f^r) \bigg) \frac{\Lambda(n)}{n^{\sigma} \log{n}}
+ E, 
\end{align}
where $E$ is evaluated as
\begin{align*}
E
&= -c \sum_{n<y} \bigg(
\sideset{}{'}\sum_{f \in \mathscr{P}_k^\mathrm{P}(q^m)}
\gamma(n; \Sym_f^r) \bigg) \frac{\Lambda(n)}{n^{\sigma} \log{n}} \\
&\qquad\quad
+ O_{r,d}\left( \frac{\log{q^m}}{(\sigma_1-1/2)^2} y^{\max\{\sigma_1-\sigma, -1\}} 
\sideset{}{'} \sum_{f \in S_k(q^m)}  1\right) \\
&\ll_{r, c} y^{1/2}
\sideset{}{'}\sum_{f \in \mathscr{P}_k^\mathrm{P}(q^m)}
1 
+ \frac{\log{q^m}}{(\sigma_1-1/2)^2} y^{\max\{\sigma_1-\sigma, -1\}} 
\sideset{}{'}\sum_{f \in \mathscr{P}_k(q^m)}
1
\end{align*}
by the inequality $|\gamma(n; \Sym_f^r)| \leq r+1$. 
Furthermore, by estimates \eqref{hol} and \eqref{P1}, we obtain
\begin{gather}\label{eq:1sterror}
E 
\ll_{r,c,d} y^{1/2} 
q^{-\Delta m} 
+ \frac{\log{q^m}}{(\sigma_1-1/2)^2} y^{\max\{\sigma_1-\sigma, -1\}}. 
\end{gather}
Denote by $U_{\ell}(x)$ the Chebyshev polynomial of the second kind of degree $\ell$; that is, 
$U_{\ell}(\cos\theta)=\sin((\ell+1)\theta)/\sin\theta$. 
Then $\{U_{\ell}(\cos \xi)\}_{\ell=0}^{\infty}$ is an orthonormal basis for $L^2([0, \pi])$ with respect to the Sato-Tate measure $d^{\mathrm{ST}} \xi$. 
Hence we obtain the identity
\begin{gather}\label{eq:Cheby}
\sum_{h=0}^{\rho} 2 \cos(j(r-2h) \xi) +\delta_{r, \mathrm{even}}
= \sum_{\ell=0}^{\infty} c(\ell; j, r) U_{\ell}(\cos \xi)
\end{gather}
for any $\xi \in [0, \pi]$, where 
\begin{align}\label{eq:c}
&c(\ell; j, r)
=\int_{0}^{\pi} \left( \sum_{h=0}^{\rho} 2 \cos(j(r-2h) \xi) + \delta_{r, \mathrm{even}} \right) 
U_{\ell}(\cos \xi) \,d^{\mathrm{ST}} \xi
\nonumber\\
=&\sum_{h=0}^{\rho}\int_0^{\pi}2\cos(j(r-2h)\xi)\frac{2}{\pi}
\sin((\ell+1)\xi)\sin\xi d\xi+\delta_{r, \mathrm{even}}\delta_{\ell,0}
\nonumber\\
=&\sum_{h=0}^{\rho}\frac{2}{\pi}\int_0^{\pi}\left(
\cos(j(r-2h)\xi)\cos(\ell\xi)
-\cos(j(r-2h)\xi)\cos((\ell+2)\xi)\right)d\xi
\nonumber\\
&+ \delta_{r, \mathrm{even}} \delta_{\ell,0}
\nonumber\\
=& \sum_{h=0}^{\rho} \frac{1}{\pi} \int_{0}^{\pi} \left\{ \cos((j(r-2h) - \ell)\xi) - \cos((j(r-2h)-\ell-2)\xi) \right\} \,d \xi
\nonumber\\
&+ \delta_{r, \mathrm{even}} \delta_{\ell,0}. 
\end{align}
From this we deduce $|c(\ell; j, r)| \leq r+1$ and $c(\ell; j, r)=0$ for $\ell>jr$. 
Recall that $\gamma(n; \Sym_f^r)$ is determined by \eqref{eq:gamma} for $n=p^j$ with $p \in \mathbb{P}_q$ and $j \geq1$. 
Thus by \eqref{eq:Cheby} we can write
\begin{gather*}
\gamma(p^j; \Sym_f^r)
= \sum_{\ell=0}^{jr} c(\ell; j, r) \lambda_f(p^\ell),
\end{gather*}
because by \eqref{euler} we have
$$
\lambda_f(p^\ell)=\sum_{h=0}^{\ell}\cos((\ell-2h)\theta_f(p))
=U_{\ell}(\cos\theta_f(p)).
$$
Then we obtain
\begin{gather*}
\sideset{}{'}\sum_{f \in \mathscr{P}_k(q^m)}
\gamma(p^j; \Sym_f^r)
= c(0; j, r) 
+ \sum_{\ell=1}^{jr} c(\ell; j, r) p^{\ell(k-1)/2} E(q^m)
\end{gather*}
by \eqref{P}.
Therefore the first term of the right-hand side of \eqref{eq:1stMoment1} is calculated as 
\begin{align*}
&c \mathop{ \sum_{p \in \mathbb{P}_q} \sum_{j=1}^{\infty} } \limits_{p^j <y} 
\bigg(
\sideset{}{'}\sum_{f \in \mathscr{P}_k(q^m)}
\gamma(p^j; \Sym_f^r) \bigg)
\frac{1}{j} p^{-j \sigma} \\
&= c \mathop{ \sum_{p \in \mathbb{P}_q} \sum_{j=1}^{\infty} } \limits_{p^j <y} 
\frac{c(0; j, r)}{j} p^{-j \sigma}
+ c \mathop{ \sum_{p \in \mathbb{P}_q} \sum_{j=1}^{\infty} } \limits_{p^j <y}
\sum_{\ell=1}^{jr} c(\ell; j, r) p^{\ell(k-1)/2} E(q^m) 
\frac{1}{j p^{j \sigma}} \\
&= S_1+S_2, 
\end{align*}
say. 
By using $|c(\ell; j, r)| \leq r+1$ and \eqref{E2}, we estimate $S_2$ as 
\begin{align}\label{eq:S2}
S_2 
&\ll_c
q^{-m-1/2} \sum_{p \in \mathbb{P}_q} \sum_{j<\log y/\log p}
\sum_{\ell=1}^{jr} p^{\ell (k-1)/2} \frac{1}{j} p^{-j \sigma}
\nonumber\\
&\ll
q^{-m-1/2} \sum_{p \in \mathbb{P}_q} \sum_{j<\log y/\log p}
\frac{1}{j}\bigg(\frac{p^{(k-1)r/2}}{p^{\sigma}}\bigg)^j
\nonumber\\
&\ll
q^{-m-1/2} \sum_{\substack{p \in \mathbb{P}_q\\ 1<\log y/\log p}} y^{(k-1)r/2}
\sum_{j<\log y/\log p}
\bigg(\frac{1}{p^{\sigma}}\bigg)^j
\nonumber\\
&\ll
q^{-m-1/2} \sum_{p< y} y^{(k-1)r/2}p^{-\sigma}
\nonumber\\
&\ll
q^{-m-1/2} y^{(k-1)r/2+1/2}.
\end{align}
To estimate $S_1$, we calculate $c(0; j, r)$ by using formula \eqref{eq:c}. 
We obtain
\begin{gather}\label{eq:c0}
c(0; j, r)
=
\begin{cases}
-1
& \text{$j=2$ and $r$ is odd}, 
\\
1
& \text{$j\neq1$ and $r$ is even},
\\
0
& \text{otherwise}.
\end{cases}
\end{gather}
Therefore we observe that, as a function in $r$, $c(0;j,r)$ is determined only by
the parity of $r$.

The above equation \eqref{eq:c0} especially gives $c(0; j, r)=0$ for $j=1$. 
Therefore we obtain
\begin{align}\label{eq:S1}
S_1
&= c \sum_{\substack{p \in \mathbb{P}_q\\ p \leq \sqrt{y} }} \sum_{2 \leq j<\log y/\log p} \frac{c(0; j, r)}{j p^{j \sigma}}
\nonumber\\
&= c \sum_{p \in \mathbb{P}_q} \sum_{j=2}^{\infty} \frac{c(0; j, r)}{j p^{j \sigma}}
+ O\bigg(
\sum_{\substack{p \in \mathbb{P}_q \\ p \leq \sqrt{y}}} \sum_{j\geq \log y/\log p}\frac{1}{j p^{j \sigma}}
+
\sum_{\substack{p \in \mathbb{P}_q \\ p > \sqrt{y}}} \sum_{j=2}^{\infty} \frac{1}{j p^{j \sigma}}
\bigg)
\nonumber\\
&= c \sum_{p \in \mathbb{P}_q} \sum_{j=2}^{\infty} \frac{c(0; j, r)}{j p^{j \sigma}}
+ O\bigg(\frac{1}{y^{\sigma-1/2}\log y}+\frac{\sigma}{ y^{\sigma-1/2}(\sigma-1/2)} \bigg). 
\end{align}
By \eqref{eq:S2} and \eqref{eq:S1}, we derive 
\begin{align*}
&c \mathop{ \sum_{p \in \mathbb{P}_q} \sum_{j=1}^{\infty} } \limits_{p^j <y} 
\bigg(
\sideset{}{'}\sum_{f \in \mathscr{P}_k(q^m)}
\gamma(p^j; \Sym_f^r) \bigg)
\frac{1}{j p^{j \sigma}} 
\\
=& c \sum_{p \in \mathbb{P}_q} \sum_{j=2}^{\infty} \frac{c(0; j, r)}{j p^{j \sigma}}
+ O\bigg(q^{-m-1/2} y^{r(k-1)/2 + 1/2}
+ \frac{y^{-\sigma + 1/2}}{\min\{(\sigma-1/2)/\sigma, \log y\}}\bigg).
\end{align*}
Therefore, formula \eqref{eq:1stMoment1} and error estimate \eqref{eq:1sterror} yield
\begin{align}\label{1stmomentfinal}
&
\sideset{}{'}\sum_{f \in \mathscr{P}_k^\mathrm{E}(q^m)}
\Psi_1(\log{L}_{\mathbb{P}_q}(\Sym_f^r, \sigma)) 
\nonumber
\\
=& c \sum_{p \in \mathbb{P}_q} \sum_{j=2}^{\infty} \frac{c(0; j, r)}{j p^{j \sigma}}
+ O\bigg(q^{-m-1/2} y^{r(k-1)/2 + 1/2}
+ y^{1/2} q^{-\Delta m}
\nonumber
\\
&+ \frac{1}{y^{\sigma-1/2}\min\{(\sigma-1/2)/\sigma, \log y\}}
+\frac{\log q^m}{(\sigma_1-1/2)^2}y^{\max\{\sigma_1-\sigma, -1\}}\bigg). 
\end{align}
Choose the parameter $y$ as $y=q^{\theta m}$ with 
$0<\theta< \min\{ 2/(r(k-1)+1), 2\Delta\}$. 
Then the error terms in \eqref{1stmomentfinal} tend to zero in the cases where $m$ tends to $\infty$ while $q$ is fixed (Case I), and $q$ tends to $\infty$ while $m$ is fixed (Case II). 
Hence we conclude that the limit formula
\begin{gather*}
\Avg \Psi_1(\log{L}_{\mathbb{P}_q}(\Sym_f^r, \sigma)) 
= c \sum_{p \in \mathbb{P}_q} \sum_{j=2}^{\infty} \frac{c(0; j, r)}{j p^{j \sigma}} 
\end{gather*}
is valid in Case I. 
On the other hand, the main term in Case II requires a different treatment. 
Since we have 
\begin{gather*}
c \sum_{p \in \mathbb{P}_q} \sum_{j=2}^{\infty} \frac{c(0; j, r)}{j p^{j \sigma}}
= c \sum_{p \in \mathbb{P}} \sum_{j=2}^{\infty} \frac{c(0; j, r)}{j p^{j \sigma}}
+ O\left(q^{-2 \sigma}\right)
\end{gather*}
for $q$ large enough, it is deduced from formula \eqref{1stmomentfinal} that
\begin{gather*}
\Avg \Psi_1(\log{L}_{\mathbb{P}_q}(\Sym_f^r, \sigma)) 
= c \sum_{p \in \mathbb{P}} \sum_{j=2}^{\infty} \frac{c(0; j, r)}{j p^{j \sigma}} 
\end{gather*}
if $q$ tends to $\infty$ while $m$ is fixed. 
As a result we obtain \eqref{eq:1stMoment} in both Cases I and II by inserting \eqref{eq:c0} to this formula. 
%
\section{Calculations of integrals involving $M$-functions}\label{sec-Mine2}
\par
In this section we assume $r=1$ or $2$.
The aim of this section is to show that
\begin{gather}\label{eq:intM}
\int_{\mathbb{R}} \mathcal{M}_\sigma(\Sym^{r}, u) \Psi_1(u) \,\frac{du}{\sqrt{2\pi}}
= c
\begin{cases}
-2^{-1} \sum_{p \in \mathbb{P}_*} p^{-2\sigma}
& r=1,
\\
\sum_{p \in \mathbb{P}_*} \sum_{j > 1} j^{-1} p^{-j \sigma}
& r=2 
\end{cases}
\end{gather}
holds for any $\sigma>1/2$ unconditionally.
Then, comparing \eqref{eq:1stMoment} and \eqref{eq:intM}, we immediately
obtain Theorem \ref{main2}.

In fact, most of this section is devoted to the proof of \eqref{eq:intM} in Case I, where the $M$-function $\mathcal{M}_\sigma(\Sym^{r}, u)$ depends on a fixed prime number $q$ and is denoted by $\mathcal{M}_{\sigma, \mathbb{P}_q}(\Sym^{r}, u)$ (see \eqref{eq:M}). 
Using the result in Case I, we finally prove \eqref{eq:intM} in Case II at the end of this section. 

We first prepare certain propositions which are necessary to verify an
exchange procedure of a limit and an integral appearing in the course of
the proof.    Some ideas in the following proof is inspired by
\cite[Lemma A in Section 5]{im-moscow}.

\begin{prop}\label{prop:}
Let $r=1$ or $2$ and let $q$ be a fixed prime number. 
Suppose that $\phi_0$ is a continuous non-decreasing function on $[0, \infty)$ satisfying $\phi_0(t) >0$, $\lim_{t \to\infty} \phi_0(t) = \infty$, and for $\sigma>1/2$
\begin{gather}\label{eq:cond1} 
\sup_{y \geq 2} 
\int_{\mathbb{R}} \mathcal{M}_{\sigma, \mathcal{P}_q(y)}(\Sym^r, u) \phi_0(|u|)^2 \frac{du}{\sqrt{2\pi}}
< \infty. 
\end{gather}
Then we have the formula
\begin{gather}\label{eq:limformula}
\lim_{y \to\infty} 
\int_{\mathbb{R}} \mathcal{M}_{\sigma, \mathcal{P}_q(y)}(\Sym^r, u) \Psi(u) \frac{du}{\sqrt{2\pi}}
= \int_{\mathbb{R}} \mathcal{M}_{\sigma, \mathbb{P}_q}(\Sym^r, u) \Psi(u) \frac{du}{\sqrt{2\pi}}
\end{gather}
for any continuous function $\Psi$ on $\mathbb{R}$ such that $|\Psi(u)| \leq \phi_0(|u|)$, where the limit is understood in the sense of Case I 
in \eqref{cases-I-II}. 
\end{prop}

\begin{proof}
First, we show that \eqref{eq:limformula} holds for any bounded continuous function $\Psi$. 
Define two Borel measures on $\mathbb{R}$ by
\begin{align*}
P_y (A)
&= \int_{A} \mathcal{M}_{\sigma, \mathcal{P}_q(y)}(\Sym^r, u) \frac{du}{\sqrt{2\pi}}, 
\\
Q (A)
&= \int_{A} \mathcal{M}_{\sigma, \mathbb{P}_q}(\Sym^r, u) \frac{du}{\sqrt{2\pi}}. 
\end{align*}
They are probability measures by Propositions \ref{M_P} and \ref{M}. 
Furthermore, the characteristic function of $P_y$ is calculated as
\begin{align*}
\widehat{P}_y (x)
&:= \int_{\mathbb{R}} \psi_x(u) P_y(du)
= \int_{\mathbb{R}} \mathcal{M}_{\sigma, \mathcal{P}_q(y)}(\Sym^r, u) \psi_x(u) \frac{du}{\sqrt{2\pi}} 
\\
&= \widetilde{\mathcal{M}}_{\sigma, \mathcal{P}_q(y)}(\Sym^r, x). 
\end{align*}
In a similar way, we have 
$\widehat{Q} (x)
= \widetilde{\mathcal{M}}_{\sigma, \mathbb{P}_q}(\Sym^r, x)$. 
Recall that $\widetilde{\mathcal{M}}_{\sigma, \mathcal{P}_q(y)}(\Sym^r, x)$ converges to $\widetilde{\mathcal{M}}_{\sigma, \mathbb{P}_q}(\Sym^r, x)$ as $y \to\infty$ uniformly in $|x| \leq a$ for any $a >0$ (Lemma \ref{lem-abc}). 
Therefore L{\'e}vy's continuity theorem asserts that $P_y$ converges 
weakly to $Q$, that is, 
\begin{gather*}
\lim_{y \to\infty} 
\int_{\mathbb{R}} \Psi(u) P_y(du)
= \int_{\mathbb{R}} \Psi(u) Q(du)
\end{gather*}
for any bounded continuous function $\Psi$ on $\mathbb{R}$. 
This is nothing but formula \eqref{eq:limformula}. 

Next, we consider the general case. 
Take a continuous function $\Psi$ satisfying $|\Psi(u)| \leq \phi_0(|u|)$. 
Then the integral
\begin{gather*}
I_y (\Psi)
= \int_{\mathbb{R}} \mathcal{M}_{\sigma, \mathcal{P}_q(y)}(\Sym^r, u) \Psi(u) \frac{du}{\sqrt{2\pi}}
\end{gather*}
is bounded for $y \geq2$. 
Indeed, we have 
\begin{gather*}
|I_y (\Psi)| 
\leq \left(\int_{\mathbb{R}} \mathcal{M}_{\sigma, \mathcal{P}_q(y)}(\Sym^r, u) \phi_0(|u|)^2 \frac{du}{\sqrt{2\pi}}\right)^{1/2}
\end{gather*}
by the Cauchy-Schwarz inequality, and condition \eqref{eq:cond1} yields that the right-hand side is bounded for $y \geq2$. 
Hence there exists a sequence $\{y_\nu\}$ for which $I_{y_\nu} (\Psi)$ converges to a limit value $I (\Psi)$ as $\nu \to\infty$. 
If one can show that 
\begin{gather}\label{eq:I}
I (\Psi)
= \int_{\mathbb{R}} \mathcal{M}_{\sigma, \mathbb{P}_q}(\Sym^r, u) \Psi(u) \frac{du}{\sqrt{2\pi}}
\end{gather}
is valid for any choice of $\{y_\nu\}$, then limit formula \eqref{eq:limformula} holds. 
To show \eqref{eq:I}, we modify $\Psi$ as follows. 
Let $E$ be any compactly supported continuous function on $\mathbb{R}$ such that $0 \leq E(u) \leq 1$ everywhere, and $E(u) =1$ for $|u| \leq1$. 
Then we put 
\begin{gather*}
\Psi_R(u) 
= \Psi(u) E \left(\frac{u}{R}\right)
\end{gather*}
with $R>0$. 
This is a bounded continuous function on $\mathbb{R}$, and the inequality
\begin{gather*}
|\Psi(u)-\Psi_R(u)|
\leq |\Psi(u)| (1-\mathbf{1}_{[-R,R]}(u)) 
\end{gather*}
holds, where $\mathbf{1}_{[-R,R]}$ is the indicator function of the interval $[-R,R]$. 
Furthermore, we have 
\begin{gather*}
(1-\mathbf{1}_{[-R,R]}(u)) \phi_0(R)
\leq \phi_0(|u|)
\end{gather*}
for any $u \in \mathbb{R}$ since $\phi_0$ is non-decreasing. 
Thus we evaluate the difference between $\Psi(u)$ and $\Psi_R(u)$ as 
\begin{gather}\label{eq:Phi}
|\Psi(u)-\Psi_R(u)|
\leq |\Psi(u)| \frac{\phi_0(|u|)}{\phi_0(R)} 
\leq \frac{\phi_0(|u|)^2}{\phi_0(R)}. 
\end{gather}
Therefore, we obtain for any $y \geq2$ the upper bound
\begin{align}\label{eq:I_y}
|I_y (\Psi) - I_y (\Psi_R)|
&\leq \int_{\mathbb{R}} \mathcal{M}_{\sigma, \mathcal{P}_q(y)}(\Sym^r, u) |\Psi(u)-\Psi_R(u)| \frac{du}{\sqrt{2\pi}} 
\nonumber \\
&\leq \frac{1}{\phi_0(R)} \int_{\mathbb{R}} \mathcal{M}_{\sigma, \mathcal{P}_q(y)}(\Sym^r, u) 
\phi_0(|u|)^2 \frac{du}{\sqrt{2\pi}} 
\nonumber \\
&\ll \frac{1}{\phi_0(R)}
\end{align}
by condition \eqref{eq:cond1}. 
Note that $I_{y_\nu} (\Psi) \to I (\Psi)$ as $\nu \to\infty$. 
Since we have checked that \eqref{eq:limformula} is true for bounded continuous functions, we also obtain
\begin{gather*}
\lim_{\nu \to\infty} 
I_{y_\nu} (\Psi_R)
= \int_{\mathbb{R}} \mathcal{M}_{\sigma, \mathbb{P}_q}(\Sym^r, u) \Psi_R(u) \frac{du}{\sqrt{2\pi}}. 
\end{gather*}
Hence \eqref{eq:I_y} yields 
\begin{gather}\label{eq:uppBD}
\left|I (\Psi) - \int_{\mathbb{R}} \mathcal{M}_{\sigma, \mathbb{P}_q}(\Sym^r, u) \Psi_R(u) \frac{du}{\sqrt{2\pi}} \right|
\ll \frac{1}{\phi_0(R)}. 
\end{gather}
On the other hand, we again apply \eqref{eq:Phi} to deduce
\begin{align}\label{tuikatuika}
&\left|\int_{\mathbb{R}} \mathcal{M}_{\sigma, \mathbb{P}_q}(\Sym^r, u) \Psi_R(u) \frac{du}{\sqrt{2\pi}}
- \int_{\mathbb{R}} \mathcal{M}_{\sigma, \mathbb{P}_q}(\Sym^r, u) \Psi(u) \frac{du}{\sqrt{2\pi}} \right|
\notag\\
&\leq \frac{1}{\phi_0(R)} 
\int_{\mathbb{R}} \mathcal{M}_{\sigma, \mathbb{P}_q}(\Sym^r, u) \phi_0(|u|)^2 \frac{du}{\sqrt{2\pi}}. 
\end{align}
Recall that $\mathcal{M}_{\sigma, \mathcal{P}_q(y)}(\Sym^r, u)$ converges to $\mathcal{M}_{\sigma, \mathbb{P}_q}(\Sym^r, u)$ as $y \to\infty$ for any $u \in \mathbb{R}$. 
Thus, by Fatou's lemma and condition \eqref{eq:cond1}, we see that
\begin{align*}
\int_{\mathbb{R}} \mathcal{M}_{\sigma, \mathbb{P}_q}(\Sym^r, u) \phi_0(|u|)^2 \frac{du}{\sqrt{2\pi}}
&\leq \liminf_{y \to\infty} 
\int_{\mathbb{R}} \mathcal{M}_{\sigma, \mathcal{P}_q(y)}(\Sym^r, u) \phi_0(|u|)^2 \frac{du}{\sqrt{2\pi}}
\\
&\ll 1. 
\end{align*}
Hence from \eqref{tuikatuika} we obtain
\begin{gather*}
\lim_{R \to\infty} \int_{\mathbb{R}} \mathcal{M}_{\sigma, \mathbb{P}_q}(\Sym^r, u) \Psi_R(u) \frac{du}{\sqrt{2\pi}}
= \int_{\mathbb{R}} \mathcal{M}_{\sigma, \mathbb{P}_q}(\Sym^r, u) \Psi(u) \frac{du}{\sqrt{2\pi}} 
\end{gather*}
due to $\lim_{R \to\infty} \phi_0(R) = \infty$. 
Therefore, letting $R \to\infty$ in \eqref{eq:uppBD}, we see that equality \eqref{eq:I} is satisfied. 
This completes the proof of the desired result. 
\end{proof}


\begin{prop}\label{propprop}
For any $\sigma>1/2$, \eqref{eq:cond1} holds with $\phi_0(t)=e^{At}$
for any positive constant $A$.
\end{prop}

\begin{proof}
Put $\phi_0^+(u) = e^{2Au}$ and $\phi_0^-(u) = e^{-2Au}$. 
Then we have 
\begin{gather*}
\phi_0(|u|)^2 \leq \phi_0^+(u) + \phi_0^-(u), 
\end{gather*}
which implies
\begin{align*}
&\int_{\mathbb{R}} \mathcal{M}_{\sigma, \mathcal{P}_q(y)}(\Sym^r, u) \phi_0(|u|)^2 \frac{du}{\sqrt{2\pi}}
\\
&\leq \int_{\mathbb{R}} \mathcal{M}_{\sigma, \mathcal{P}_q(y)}(\Sym^r, u) \phi_0^+(u) \frac{du}{\sqrt{2\pi}}
+ \int_{\mathbb{R}} \mathcal{M}_{\sigma, \mathcal{P}_q(y)}(\Sym^r, u) \phi_0^-(u) \frac{du}{\sqrt{2\pi}}\\
&= I^+(y)+I^-(y), 
\end{align*}
say.
By Proposition \ref{M_P}, we have 
\begin{gather*}
I^+(y)
= \int_{\Theta_{\mathcal{P}_q(y)}} \exp\left( 2A \mathscr{G}_{\sigma, \mathcal{P}_q(y)} (e^{i \theta_{\mathcal{P}_q} r}) \right)
d^{\mathrm{ST}} \theta_{\mathcal{P}_q(y)},
\end{gather*}
and this is further
\begin{align}\label{eq:plus}
&= \int_{\Theta_{\mathcal{P}_q(y)}} \prod_{p \in \mathcal{P}_q(y)} 
\exp\left( 2A \mathscr{G}_{\sigma, p} (e^{i \theta_{\mathcal{P}_q} r}) \right)
d^{\mathrm{ST}} \theta_{\mathcal{P}_q(y)} 
\nonumber \\
&= \prod_{p \in \mathcal{P}_q(y)} 
\int_{0}^{\pi} \exp\left( 2A \mathscr{G}_{\sigma, p} (e^{i \theta_p r}) \right) d^{\mathrm{ST}} \theta_p 
\end{align}
by the definitions of the function $\mathscr{G}_{\sigma, \mathcal{P}_q(y)}$ and the Sato-Tate measure $d^{\mathrm{ST}} \theta_{\mathcal{P}_q(y)}$. 
The Taylor expansion of $\exp(z)$ gives the asymptotic formula
\begin{gather*}
\exp\left( 2A \mathscr{G}_{\sigma, p} (e^{i \theta_p r}) \right)
= 1 + 2A \mathscr{G}_{\sigma, p} (e^{i \theta_p r}) +O_{A}\left(p^{-2\sigma}\right)
\end{gather*}
since $\mathscr{G}_{\sigma, p} (e^{i \theta_p r}) \ll p^{-\sigma}$ is satisfied for any $\theta_p \in [0, \pi]$. 
Therefore
\begin{gather*}
\int_{0}^{\pi} \exp\left( 2A \mathscr{G}_{\sigma, p} (e^{i \theta_p r}) \right) d^{\mathrm{ST}} \theta_p 
= 1 + 2A \int_{0}^{\pi} \mathscr{G}_{\sigma, p} (e^{i \theta_p r}) d^{\mathrm{ST}} \theta_p 
+O_{A}\left(p^{-2\sigma}\right). 
\end{gather*}
Since the integral on the right-hand side is $O(p^{-2\sigma})$ by
\eqref{wp_0},
the infinite product 
\begin{gather*}
\prod_{p \in \mathbb{P}_q} 
\int_{0}^{\pi} \exp\left( 2A \mathscr{G}_{\sigma, p} (e^{i \theta_p r}) \right) d^{\mathrm{ST}} \theta_p 
\end{gather*}
converges absolutely. 
By \eqref{eq:plus}, we conclude that the integral
$I^+(y)$
is bounded for $y \geq2$. 
We can prove in the same line that
\begin{gather*}
I^-(y)
= \int_{\Theta_{\mathcal{P}_q(y)}} \exp\left( -2A \mathscr{G}_{\sigma, \mathcal{P}_q(y)} (e^{i \theta_{\mathcal{P}_q} r}) \right)
d^{\mathrm{ST}} \theta_{\mathcal{P}_q(y)}
\end{gather*}
is also bounded. 
Therefore we obtain
\begin{gather*}
\sup_{y \geq2} 
\int_{\mathbb{R}} \mathcal{M}_{\sigma, \mathcal{P}_q(y)}(\Sym^r, u) \phi_0(|u|)^2 \frac{du}{\sqrt{2\pi}}
\leq \sup_{y \geq2} I^+(y)
+ \sup_{y \geq2} I^-(y)
< \infty
\end{gather*}
as desired. 
\end{proof}

In particular, in view of Proposition \ref{propprop},
we may choose $\Psi(u)=\Psi_1(u)=cu$ in Proposition 
\ref{prop:}, which yields

\begin{cor}\label{cor:}
Let $r=1$ or $2$.
For any $\sigma>1/2$, we have 
\begin{gather}\label{cor:-formula}
\lim_{y \to\infty} 
\int_{\mathbb{R}} \mathcal{M}_{\sigma, \mathcal{P}_q(y)}(\Sym^r, u) \Psi_1(u) \frac{du}{\sqrt{2\pi}}
= \int_{\mathbb{R}} \mathcal{M}_{\sigma, \mathbb{P}_q}(\Sym^r, u) \Psi_1(u) \frac{du}{\sqrt{2\pi}},
\end{gather}
where the limit is understood in the sense of Case I in \eqref{cases-I-II}.
\end{cor}

Now we prove \eqref{eq:intM} in Case I.
In view of Corollary \ref{cor:}, it is enough to calculate the
left-hand side of \eqref{cor:-formula}.

Let $\mathcal{P}_q$ be any finite subset of $\mathbb{P}_q$.
From Proposition~\ref{M_P},
we have 
\begin{align}\label{eq:intM1}
\int_{\mathbb{R}} \mathcal{M}_{\sigma, \mathcal{P}_q}(\Sym^{r}, u) \Psi_1(u) \,\frac{du}{\sqrt{2\pi}}
&= \int_{\Theta_{\mathcal{P}_q}} \Psi_1 \left( \mathscr{G}_{\sigma, \mathcal{P}_q} (e^{i \theta_{\mathcal{P}_q} r}) \right) 
\,d^{\mathrm{ST}} \theta_{\mathcal{P}_q} \notag\\
&= c \sum_{p \in \mathcal{P}_q} 
\int_{\Theta_{\mathcal{P}_q}} \mathscr{G}_{\sigma, p} (e^{i \theta_p r}) 
\,d^{\mathrm{ST}} \theta_{\mathcal{P}_q} \nonumber\\
&= c \sum_{p \in \mathcal{P}_q} 
\int_{0}^{\pi} \mathscr{G}_{\sigma, p} (e^{i \theta_p r}) \,d^{\mathrm{ST}} \theta_p 
\end{align}
by the definition of the measure $d^{\mathrm{ST}} \theta_{\mathcal{P}_q}$. 
For any $p \in \mathcal{P}_q$, we have
\begin{align*}
&
\int_{0}^{\pi} \mathscr{G}_{\sigma, p} (e^{i \theta_p r}) \,d^{\mathrm{ST}} \theta_p
\\
&= \int_{0}^{\pi} g_{\sigma, p} (e^{i \theta_p r}) \,d^{\mathrm{ST}} \theta_p 
+ \int_{0}^{\pi} g_{\sigma, p} (e^{i \theta_p r}) \,d^{\mathrm{ST}} \theta_p 
+ g_{\sigma,p}(\delta_{r, \mathrm{even}}) \\
&= \sum_{j=1}^{\infty} \left( \int_{0}^{\pi} (2 \cos(\theta_p jr) + \delta_{r, \mathrm{even}}) \,d^{\mathrm{ST}} \theta_p \right)
\frac{1}{j p^{j \sigma}}. 
\end{align*}
By \eqref{eq:Cheby} with $\rho=0$, the identity
\begin{gather*}
2 \cos(\theta_p jr) + \delta_{r, \mathrm{even}}
= \sum_{\ell=0}^{\infty} c(\ell; j, r) U_{\ell}(\cos \theta_p)
\end{gather*}
holds for any $\theta_p \in [0, \pi]$. 
Since $\{U_{\ell}(\cos \xi)\}_{\ell=0}^{\infty}$ is an orthonormal basis for $L^2([0, \pi])$ with respect to $d^{\mathrm{ST}} \theta_p$ and $U_0(x)=1$, we obtain
\begin{gather*}
\int_{0}^{\pi} (2 \cos(\theta_p jr) + \delta_{r, \mathrm{even}}) \,d^{\mathrm{ST}} \theta_p
= c(0; j,r), 
\end{gather*}
which implies
\begin{gather}\label{eq:int_p}
\int_{0}^{\pi} \mathscr{G}_{\sigma, p} (e^{i \theta_p r}) \,d^{\mathrm{ST}} \theta_p 
= \sum_{j=1}^{\infty} \frac{c(0; j, r)}{j p^{j \sigma}}. 
\end{gather}
Inserting this into \eqref{eq:intM1} with $\mathcal{P}_q=\mathcal{P}_q(y)$, we derive
\begin{gather*}
\int_{\mathbb{R}} \mathcal{M}_{\sigma, \mathcal{P}_q(y)}(\Sym^r, u) \Psi_1(u) \,\frac{du}{\sqrt{2\pi}}
= c \sum_{p \in \mathcal{P}_q(y)} \sum_{j=1}^{\infty} \frac{c(0; j, r)}{j p^{j \sigma}} 
\end{gather*}
Then we obtain \eqref{eq:intM} in Case I by letting $y \to\infty$ and using \eqref{eq:c0}.

Finally, we confirm that \eqref{eq:intM} remains valid in Case II. 
The $M$-function in this case is denoted by $\mathcal{M}_{\sigma, \mathbb{P}}(\Sym^{r}, u)$ (see \eqref{eq:M}). 
Let $q$ be a fixed prime number. 
Then we will show the following identity
\begin{gather}\label{eq:M_Case_I_II}
\mathcal{M}_{\sigma, \mathbb{P}}(\Sym^{r}, u)
= \int_{\mathbb{R}}
\mathcal{M}_{\sigma, \mathbb{P}_{q}}(\Sym^r, u')
\mathcal{M}_{\sigma, q}(\Sym^r, u-u')
\frac{du'}{\sqrt{2\pi}}
\end{gather}
for $\sigma>1/2$ and $u \in \mathbb{R}$, where $\mathcal{M}_{\sigma, \mathbb{P}_{q}}(\Sym^r, u)$ is the $M$-function in Case I. 
By Lemma \ref{lem-abc} (c) and Proposition \ref{M} (4), the Fourier transform of $\mathcal{M}_{\sigma, \mathbb{P}}(\Sym^{r}, u)$ is represented as 
\begin{gather*}
\widetilde{\mathcal{M}}_{\sigma, \mathbb{P}}(\Sym^r, x)
= \lim_{y \to\infty}
\prod_{p\in\mathcal{P}_{q'}(y)}\widetilde{\mathcal{M}}_{\sigma, p}(\Sym^r, x), 
\end{gather*}
where the limit is understood in the sense that 
\begin{gather}\label{eq:lim_II}
y=y(q')<q'
\quad\text{and}\quad
y(q') \to\infty 
\quad\text{as}\quad 
q' \to\infty 
\end{gather}
as in Case II of \eqref{cases-I-II}. 
Note that $\mathcal{P}_{q'}(y)$ equals to $\mathcal{P}(y)$ due to $y=y(q')<q'$. 
Therefore, for any prime number $q'$ satisfying $q<y(q')$, we obtain
\begin{gather*}
\prod_{p\in\mathcal{P}_{q'}(y)} \widetilde{\mathcal{M}}_{\sigma, p}(\Sym^r, x)
= \prod_{p\in\mathcal{P}_{q}(y)} \widetilde{\mathcal{M}}_{\sigma, p}(\Sym^r, x)
\cdot \widetilde{\mathcal{M}}_{\sigma, q}(\Sym^r, x). 
\end{gather*}
In the same sense of the limit as in \eqref{eq:lim_II}, we have further
\begin{gather*}
\lim_{y \to\infty}
\prod_{p\in\mathcal{P}_{q}(y)} \widetilde{\mathcal{M}}_{\sigma, p}(\Sym^r, x) 
= \widetilde{\mathcal{M}}_{\sigma, \mathbb{P}_{q}}(\Sym^r, x)
\end{gather*}
by Lemma \ref{lem-abc} (c) since $y=y(q')$ satisfies $y>q$ and $y \to\infty$. 
As a result, we have the identity
\begin{gather*}
\widetilde{\mathcal{M}}_{\sigma, \mathbb{P}}(\Sym^r, x)
= \widetilde{\mathcal{M}}_{\sigma, \mathbb{P}_{q}}(\Sym^r, x)
\cdot \widetilde{\mathcal{M}}_{\sigma, q}(\Sym^r, x),  
\end{gather*}
and we derive \eqref{eq:M_Case_I_II} by considering the Fourier inverse transform of both sides. 
Therefore using \eqref{eq:M_Case_I_II} and $\Psi_1(u_1+u_2)=\Psi_1(u_1)+\Psi_1(u_2)$ we obtain 
\begin{align}\label{eq:CaseItoII}
&\int_{\mathbb{R}} \mathcal{M}_{\sigma, \mathbb{P}}(\Sym^{r}, u) \Psi_1(u) \,\frac{du}{\sqrt{2\pi}} 
\nonumber
\\
&= \int_{\mathbb{R}} \int_{\mathbb{R}} 
\mathcal{M}_{\sigma, \mathbb{P}_{q}}(\Sym^{r}, u_1) \mathcal{M}_{\sigma, q}(\Sym^r, u_2)
\Psi_1(u_1+u_2) 
\,\frac{du_1}{\sqrt{2\pi}} \frac{du_2}{\sqrt{2\pi}} 
\nonumber
\\
&= \int_{\mathbb{R}} 
\mathcal{M}_{\sigma, \mathbb{P}_{q}}(\Sym^{r}, u_1) 
\Psi_1(u_1) 
\,\frac{du_1}{\sqrt{2\pi}} 
\int_{\mathbb{R}}
\mathcal{M}_{\sigma, q}(\Sym^r, u_2) \,\frac{du_2}{\sqrt{2\pi}} 
\nonumber
\\
&\quad
+
\int_{\mathbb{R}} \mathcal{M}_{\sigma, \mathbb{P}_{q}}(\Sym^{r}, u_1) 
\,\frac{du_1}{\sqrt{2\pi}} 
\int_{\mathbb{R}} 
\mathcal{M}_{\sigma, q}(\Sym^r, u_2)
\Psi_1(u_2) 
\,\frac{du_2}{\sqrt{2\pi}}
\nonumber 
\\
&= \int_{\mathbb{R}} 
\mathcal{M}_{\sigma, \mathbb{P}_{q}}(\Sym^{r}, u_1) 
\Psi_1(u_1) 
\,\frac{du_1}{\sqrt{2\pi}} 
+ \int_{\mathbb{R}} 
\mathcal{M}_{\sigma, q}(\Sym^r, u_2)
\Psi_1(u_2) 
\,\frac{du_2}{\sqrt{2\pi}},
\end{align}
where on the last line we used \eqref{sosuuhitotunobaai} and
Proposition \ref{M} (5).
The first integral of the right-hand side of \eqref{eq:CaseItoII} can be calculated by \eqref{eq:intM} in Case I. 
Furthermore, the second integral is calculated as
\begin{gather*}
\int_{\mathbb{R}} 
\mathcal{M}_{\sigma, q}(\Sym^r, u_2)
\Psi_1(u_2) 
\,\frac{du_2}{\sqrt{2\pi}}
= c \int_{0}^{\pi} \mathscr{G}_{\sigma, q} (e^{i \theta_{q} r}) \,d^{\mathrm{ST}} \theta_{q}
\end{gather*}
by the definition of $\mathcal{M}_{\sigma, q}(\Sym^r, u)$
(see \eqref{prop-1-1}), 
which is further equal to 
\begin{gather*}
= c \sum_{j=1}^{\infty} \frac{c(0; j, r)}{j {q}^{j \sigma}}
= c
\begin{cases}
-2^{-1} {q}^{-2\sigma}
& r=1,
\\
 \sum_{j > 1} j^{-1} {q}^{-j \sigma}
& r=2 
\end{cases}
\end{gather*}
by \eqref{eq:int_p}. 
Inserting this to \eqref{eq:CaseItoII} and applying \eqref{eq:intM} in Case I, we obtain 
\begin{gather*}
\int_{\mathbb{R}} \mathcal{M}_{\sigma, \mathbb{P}}(\Sym^{r}, u) \Psi_1(u) \,\frac{du}{\sqrt{2\pi}}
= c
\begin{cases}
-2^{-1} \sum_{p \in \mathbb{P}} p^{-2\sigma}
& r=1,
\\
\sum_{p \in \mathbb{P}} \sum_{j > 1} j^{-1} p^{-j \sigma}
& r=2 
\end{cases}
\end{gather*}
which is nothing but \eqref{eq:intM} in Case II.

%
%
\section{The deduction mentioned in Remark \ref{Psi_1-valid}}\label{sec-1to2} 

We conclude this paper with describing the details of the deduction
argument mentioned in Remark \ref{Psi_1-valid}.
The proof of Theorem \ref{main2} given in Sections \ref{sec-Mine} and
\ref{sec-Mine2} is to show that both the left-hand side and the
right-hand side of \eqref{thm2-1} are equal to \eqref{thm2-2}, and 
consequently the both sides of \eqref{thm2-1} are the same. 
The argument below supplies a direct deduction of \eqref{thm2-1}, and
gives a reasoning of the ``parity result'' of $M$-functions
(which also clarifies the meaning of the results in \cite{mu}).

We assume that Theorem \ref{main1} and
\cite[Theorem 1.5]{mu} are valid for $\Psi_1(x)=cx$.
This implies that \eqref{thm2-1} is valid for $r=1,2$.
We show \eqref{thm2-1} for general $r$ by induction.
We suppose that \eqref{thm2-1} is valid for all $r < \mathfrak{r}$.
From \cite[Theorem~1.5]{mu} we know that
there exists $\mathcal{M}_{\sigma}^{\sharp}=\mathcal{M}_{\sigma,\mathbb{P}_*}^{\sharp}$ 
such that
\[
  \int_{\mathbb{R}}\mathcal{M}_{\sigma}^{\sharp}(u)\Psi_1(u) \frac{du}{\sqrt{2\pi}}
  =
  \mathrm{Avg}\Psi_1(\log L_{\mathbb{P}_q}(\mathrm{Sym}_f^{\mathfrak{r}}, \sigma)
  -
  \log L_{\mathbb{P}_q}(\mathrm{Sym}_f^{\mathfrak{r}-2}, \sigma)).
  \]
  
\begin{rem}
Actually in the statement of \cite[Theorem 1.5]{mu}, the $M$-function in
Case I and that in Case II have not been distinguished, but they should be
different as in Theorem \ref{main1} in the present paper. 
Here we write them as $\mathcal{M}_{\sigma,\mathbb{P}_q}^{\sharp}$ in Case I and
$\mathcal{M}_{\sigma,\mathbb{P}}^{\sharp}$ in Case II. 
(In fact, just before \cite[Proposition 3.2]{mu}, the parameter $y$ was 
introduced, but we should notice that this $y$ should satisfy $(34)$ in the
present paper.)
\end{rem}

This $\mathcal{M}_{\sigma}^{\sharp}(u)$ is constructed in a way similar to that of
$\mathcal{M}_{\sigma}(\Sym^r, u)$ in the present paper; that is, there exists
$\mathcal{M}_{\sigma,\mathcal{P}_q(y)}^{\sharp}(u)$, an analogue of
$\mathcal{M}_{\sigma,\mathcal{P}_q(y)}(\Sym^r,u)$ in the present paper,
which tends to $\mathcal{M}_{\sigma}^{\sharp}(u)$ as $y$ tends to $\infty$,
where the limit $y\to\infty$ is to be understood as \eqref{cases-I-II}.

Since $\Psi_1(x+y)=\Psi_1(x) + \Psi_1(y)$, we have
\begin{align}\label{cor-pr-1}
  &
  \int_{\mathbb{R}}\mathcal{M}_{\sigma}^{\sharp}(u)\Psi_1(u) \frac{du}{\sqrt{2\pi}}
  \notag\\
  =&
  \mathrm{Avg}\Psi_1(\log L_{\mathbb{P}_q}(\mathrm{Sym}_f^{\mathfrak{r}}, \sigma))
  -
  \mathrm{Avg}\Psi_1(\log L_{\mathbb{P}_q}(\mathrm{Sym}_f^{\mathfrak{r}-2}, \sigma)),
\end{align}
while by induction assumption
\begin{align}\label{cor-pr-2}
\mathrm{Avg}\Psi_1
(
\log L_{\mathbb{P}_q}(\mathrm{Sym}_f^{\mathfrak{r}-2}, \sigma)
)
=
\int_{\mathbb{R}}
\mathcal{M}_{\sigma}(\mathrm{Sym}^{\mathfrak{r}_0}, u)
\Psi_1(u) \frac{du}{\sqrt{2\pi}},
%
\end{align}
where $\mathfrak{r}_0$ is 2 if $\mathfrak{r}$ is even, 1 if
$\mathfrak{r}$ is odd.
From \eqref{cor-pr-1} and \eqref{cor-pr-2} we obtain
\[
\mathrm{Avg}
\Psi_1(\log L_{\mathbb{P}_q}(\mathrm{Sym}_f^{\mathfrak{r}}, \sigma)
=
\int_{\mathbb{R}}
\mathcal{M}_{\sigma}(\mathrm{Sym}^{\mathfrak{r}_0}, u)
\Psi_1(u) \frac{du}{\sqrt{2\pi}}
+
\int_{\mathbb{R}}
  \mathcal{M}_{\sigma}^{\sharp}(u)
  \Psi_1(u) \frac{du}{\sqrt{2\pi}}.
\]
Formula \eqref{thm2-1} for $\mathfrak{r}$ will then follow if we can show:

\begin{lem}\label{lastlemma}
We have
\begin{align}\label{integral-vanish}
\int_{\mathbb{R}}
  \mathcal{M}_{\sigma}^{\sharp}(u)
  \Psi_1(u) \frac{du}{\sqrt{2\pi}}=0.
\end{align}
\end{lem}

\begin{proof}
First treat Case I.
To prove \eqref{integral-vanish}, we first notice
\begin{gather}\label{corcor-formula}
\lim_{y \to\infty} 
\int_{\mathbb{R}} \mathcal{M}_{\sigma, \mathcal{P}_q(y)}^{\sharp}(u) \Psi_1(u) \frac{du}{\sqrt{2\pi}}
= \int_{\mathbb{R}} \mathcal{M}_{\sigma,\mathbb{P}_q}^{\sharp}(u) \Psi_1(u) \frac{du}{\sqrt{2\pi}},
\end{gather}
where $y\to\infty$ is as Case I of \eqref{cases-I-II}.
This can be shown analogously to Corollary \ref{cor:}, by using
\cite[Propositions 3.1 and 3.3]{mu}. 

We evaluate the integral on the left-hand side of
\eqref{corcor-formula}.
Notice that
\begin{align}\label{Cauchy}
\int_{\mathcal{T}}
\log(1-t_p p^{-\sigma})d^H t_{p}=0
\end{align}
for any prime $p$, because of Cauchy's theorem.
Using the formula stated in the statement of \cite[Proposition 3.1]{mu},
for any finite subset $\mathcal{P}\subset\mathbb{P}_q$ we have
\begin{align*}
&\int_{\mathbb{R}} \mathcal{M}_{\sigma, \mathcal{P}}^{\sharp}(u) \Psi_1(u) \frac{du}{\sqrt{2\pi}}
=\int_{\mathcal{T}_{\mathcal{P}}}\Psi_1
(2\Re g_{\sigma,\mathcal{P}}(t_{\mathcal{P}}))d^H t_{\mathcal{P}}
\notag\\
&\quad= -2c\Re \int_{\mathcal{T}_{\mathcal{P}}}
\sum_{p\in \mathcal{P}}\log(1-t_p p^{-\sigma})d^H t_{\mathcal{P}}
\notag\\
&\quad =-2c\Re \sum_{p\in \mathcal{P}}
\int_{\mathcal{T}}
\log(1-t_p p^{-\sigma})d^H t_{p},
\end{align*}
which is equal to 0 by \eqref{Cauchy}.
Therefore, from this result with $\mathcal{P}=\mathcal{P}_q(y)$ and 
\eqref{corcor-formula}, we obtain 
\eqref{integral-vanish} in Case I.

Next consider Case II.    As an analogue of Lemma \ref{lem-abc} (c),
we have the Euler product expression of the Fourier integral of
$\mathcal{M}_{\sigma}^{\sharp}$ (though this point is not explicitly mentioned
in \cite{mu}).
Therefore, as in the last stage of the proof of \eqref{eq:intM}, 
we obtain
$$
\mathcal{M}_{\sigma,\mathbb{P}}^{\sharp}(u)=\int_{\mathbb{R}}
\mathcal{M}_{\sigma,\mathbb{P}_q}^{\sharp}(u')
\mathcal{M}_{\sigma,q}^{\sharp}(u-u')\frac{du'}{\sqrt{2\pi}},
$$
and hence
\begin{align*}
&\int_{\mathbb{R}}\mathcal{M}_{\sigma,\mathbb{P}}^{\sharp}(u)\Psi_1(u)\frac{du_1}
{\sqrt{2\pi}}\notag\\
&\quad=\int_{\mathbb{R}}\mathcal{M}_{\sigma,\mathbb{P}_q}^{\sharp}(u_1)\Psi_1(u_1)
\frac{du}{\sqrt{2\pi}}+
\int_{\mathbb{R}}\mathcal{M}_{\sigma,q}^{\sharp}(u_2)\Psi_1(u_2)\frac{du_2}
{\sqrt{2\pi}},
\end{align*}
analogously to \eqref{eq:CaseItoII}.
On the right-hand side, the first integral vanishes because of
\eqref{integral-vanish} in Case I, and the second integral also vanishes
because this is equal to the left-hand side of \eqref{Cauchy} with $p=q$
(see the proof of \cite[Proposition 3.1]{mu}).
Thus \eqref{integral-vanish} in Case II follows, and
the proof of formula \eqref{thm2-1} for $\mathfrak{r}$ is complete.
\end{proof}
\bigskip

{\bf Acknowledgment.}
The authors express their sincere thanks to Professors Yasushi Mizusawa and Shingo Sugiyama
for pointing out inaccuracies included in an earlier version of the manuscript.

%
%

%
\noindent
Philippe Lebacque:\\
Laboratoire de Math\'{e}matiques de Besan\c{c}on,\\
UFR Sciences et techniques 16, route de Gray 25 030 Besan\c{c}on, France.\\
philippe.lebacque@univ-fcomte.fr
\vspace{0.4cm}
\\
Kohji Matsumoto:\\
Graduate School of Mathematics,
Nagoya University, Furocho, Chikusa-ku, Nagoya 464-8602, Japan.\\
kohjimat@math.nagoya-u.ac.jp
\vspace{0.4cm}
\\
Masahiro Mine:\\
Faculty of Science and Technology, Sophia University,
7-1 Kioi-cho, Chiyoda-ku, Tokyo 102-8554, Japan.\\
m-mine@sophia.ac.jp
\vspace{0.4cm}
\\
Yumiko Umegaki:\\
Department of Mathematical and Physical Sciences,
Nara Women's University,
Kitauoya Nishimachi, Nara 630-8506, Japan.\\
ichihara@cc.nara-wu.ac.jp
\end{document}